\documentclass[reqno,10pt]{amsart}

\usepackage[a4paper,left=25mm,right=25mm,top=30mm,bottom=30mm,marginpar=25mm]{geometry} 
\usepackage{amsmath}
\usepackage{amssymb}
\usepackage{amsthm}
\usepackage{amscd}
\usepackage[ansinew]{inputenc}
\usepackage{color}
\usepackage[final]{graphicx}
\usepackage{tikz}
\usepackage{bbm}

\usepackage{pgfplots}
\usetikzlibrary{intersections, pgfplots.fillbetween}

\usetikzlibrary{patterns} 
\usetikzlibrary{positioning}
\usetikzlibrary{calc}

\usetikzlibrary{intersections}

\usepackage{float}
 \usepackage[colorlinks=true, pdfstartview=FitV, 
 linkcolor=olive
 ,             citecolor=olive
 , urlcolor=olive
 ]{hyperref}

\renewcommand{\epsilon}{\varepsilon}

\numberwithin{equation}{section}

\newtheoremstyle{thmlemcorr}{10pt}{10pt}{\itshape}{}{\bfseries}{.}{10pt}{{\thmname{#1}\thmnumber{ #2}\thmnote{ (#3)}}}
\newtheoremstyle{thmlemcorr*}{10pt}{10pt}{\itshape}{}{\bfseries}{.}\newline{{\thmname{#1}\thmnumber{ #2}\thmnote{ (#3)}}}
\newtheoremstyle{defi}{10pt}{10pt}{\itshape}{}{\bfseries}{.}{10pt}{{\thmname{#1}\thmnumber{ #2}\thmnote{ (#3)}}}
\newtheoremstyle{remexample}{10pt}{10pt}{}{}{\bfseries}{.}{10pt}{{\thmname{#1}\thmnumber{ #2}\thmnote{ (#3)}}}
\newtheoremstyle{ass}{10pt}{10pt}{}{}{\bfseries}{.}{10pt}{{\thmname{#1}\thmnumber{ A#2}\thmnote{ (#3)}}}

\theoremstyle{thmlemcorr}
\newtheorem{theorem}{Theorem}
\numberwithin{theorem}{section}
\newtheorem{lemma}[theorem]{Lemma}

\newtheorem{proposition}[theorem]{Proposition}

\theoremstyle{thmlemcorr*}
\newtheorem{theorem*}{Theorem}
\newtheorem{lemma*}[theorem]{Lemma}
\newtheorem{corollary*}[theorem]{Corollary}
\newtheorem{proposition*}[theorem]{Proposition}
\newtheorem{problem*}[theorem]{Problem}
\newtheorem{conjecture*}[theorem]{Conjecture}

\theoremstyle{defi}
\newtheorem{definition}[theorem]{Definition}

\theoremstyle{remexample}
\newtheorem{remark}[theorem]{Remark}

\theoremstyle{ass}

\newcommand{\ZZ}{\mathbb{Z}}

\newcommand{\Acal}{\mathcal{A}}

\newcommand{\Bcal}{\mathcal{B}}

\newcommand{\Hcal}{\mathcal{H}}

\newcommand{\Lcal}{\mathcal{L}}
\newcommand{\Mcal}{\mathcal{M}}

\newcommand{\Ibb}{\mathbb{I}}

\newcommand{\NN}{\mathbb{N}}

\newcommand{\norm}[1]{\|#1\|}

\newcommand{\abs}[1]{|#1|}

\newcommand{\dd}{\;\mathrm{d}}

\newcommand{\N}{\mathbb{N}}
\newcommand{\R}{\mathbb{R}}
\newcommand{\RR}{\mathbb{R}}

\newcommand{\Z}{\mathbb{Z}}

\newcommand{\weakly}{\rightharpoonup}
\newcommand{\weaklystar}{\overset{*}\rightharpoonup}
\newcommand{\tostar}{\overset{*}\rightarrow}

\newcommand{\eps}{\epsilon}
\newcommand{\epsi}{\epsilon}
\newcommand{\ffi}{\varphi}

\newcommand{\Yrig}{Y_{\rm rig}}
\newcommand{\Ysoft}{Y_{\rm soft}}

\allowdisplaybreaks

 \definecolor{Korange}{rgb}{0.945,0.561,0}

\title[Homogenization problems in BV]{Homogenization in \boldmath{$BV$} of a model for layered composites in finite crystal plasticity}

\author[Elisa Davoli]{Elisa Davoli}
\address{University of Vienna, Faculty of Mathematics, Oskar-Morgenstern-Platz 1, 1090 Vienna, Austria}
\email{elisa.davoli@univie.ac.at}

\author[Rita Ferreira]{Rita Ferreira}
\address{King Abdullah University of Science and Technology (KAUST), CEMSE Division, Thuwal 23955-6900, Saudi Arabia}
\email{rita.ferreira@kaust.edu.sa}

\author[Carolin Kreisbeck]{Carolin Kreisbeck}
\address{Mathematisch Instituut, Universiteit Utrecht, Postbus 80010, 3508 TA Utrecht, The Netherlands}
\email{C.Kreisbeck@uu.nl}

\usepackage[normalem]{ulem}

\begin{document}

 
\maketitle

 \begin{abstract}  
 \vspace{-12pt}   

In this work,  we study the effective  behavior of a two-dimensional variational model within finite crystal plasticity for high-contrast bilayered composites. Precisely, we consider materials arranged into periodically alternating thin horizontal strips of an elastically rigid component and a softer one with one active slip system. The energies arising from these modeling assumptions are of integral form, featuring linear growth and non-convex differential constraints.
 We approach this non-standard homogenization problem via Gamma-convergence.  A crucial first step in the asymptotic analysis  is the characterization of rigidity properties of limits of admissible deformations in the space $BV$ of functions of bounded variation. In particular, we prove that, under suitable assumptions, the two-dimensional body may split horizontally into finitely many pieces,  each of which undergoes  shear deformation and global rotation. This allows us to identify a potential candidate for the homogenized limit energy, which we show  
to be a lower bound on the Gamma-limit. 
In the framework of non-simple materials, we present a complete Gamma-convergence result, including an explicit homogenization formula, for a regularized model with an anisotropic penalization in the layer direction.

\vspace{8pt}

 

 \noindent\textsc{Date:} \today.
 \end{abstract}

\section{Introduction} 
Metamaterials are artificially engineered composites whose  heterogeneities  are optimized  to improve structural performances. Due to their special mechanical properties, arising as a result of complex microstructures, metamaterials play a key role in industrial applications and are an increasingly active field of research. Two natural questions when dealing with composite materials are how the  effective material response  is influenced by the geometric distribution of its components, and how the mechanical properties of the components impact the overall  macroscopic behavior  of the metamaterial.

 In what follows, we investigate these questions for a special class of metamaterials with two characteristic features that are of relevance in a number of applications: (i) the material consists of two components arranged in a highly anisotropic way into periodically alternating layers, and (ii) the (elasto)plastic properties of the two components exhibit strong differences, in the sense that one is rigid, while the other one is considerably softer, allowing for large (elasto)plastic deformations. 
The analysis of variational models for such layered high-contrast materials was initiated in \cite{ChK17}. 
There, the authors derive a macroscopic description for a two-dimensional model in the context of geometrically nonlinear but rigid elasticity, assuming that the softer component can be deformed along a single active slip system with linear self-hardening. 
 
These results have been extended to general dimensions,  to energy densities with $p$-growth for $1<p<+\infty$, and to the case  with non-trivial elastic energies, which allows  treating very stiff  (but not necessarily rigid) layers,  see~\cite{ChK18, Chr18}. 

 In this paper, we carry the ideas of~\cite{ChK17} forward to a model for plastic composites without linear hardening, in the spirit of~\cite{CoT05}. This change turns the variational problem in~\cite{ChK17}, having  quadratic growth (cf.~also~\cite{Con06, CDK11}), into one with energy densities that grow merely linearly. 

The main novelty lies in the fact that the homogenization analysis  must be performed  in the class $BV$ of functions of bounded variation (see \cite{AFP00})  to account for concentration phenomena. This  gives rise to conceptual  mathematical difficulties: on the one hand, the standard convolution techniques commonly used for density arguments in $BV$ or $SBV$ cannot be directly applied because they do not
preserve the intrinsic constraints of the problem; on the other hand, constraint-preserving approximations in this weaker setting of $BV$ are rather challenging, as one needs to simultaneously regularize the absolutely continuous part of the distributional
derivative of the functions and  accommodate their jump sets. \\

  To state our results precisely, we first introduce the relevant model with its main modeling hypotheses. Throughout the article, we  analyze two versions of the model, namely with and without regularization.

Let $e_1$ and $e_2$ be the standard unit vectors in $\mathbb{R}^2$, and let    $x=(x_1,x_2)$ denote a generic point in $\mathbb{R}^2$. %
Unless specified otherwise,  $\Omega\subset\R^2$ is an $x_1$-connected,  bounded  domain  with Lipschitz boundary, that is, an open set whose slices in the $x_1$-direction are (possibly empty) open intervals (see Subsection \ref{subs:geom} for the precise definition).  For such a domain $\Omega$, we set 
\begin{align}\label{aOmega}
\text{$a_\Omega:=\inf_{x\in \Omega} x_2$ \qquad and \qquad $b_\Omega:=\sup_{x\in \Omega} x_2$,}
\end{align} 
as well as
\begin{align}\label{cOmega}
\text{$c_\Omega:=\inf_{x\in \Omega} x_1$ \qquad and \qquad $d_\Omega:=\sup_{x\in \Omega} x_1$.}
\end{align}
Assume that $\Omega$ is the reference configuration  of a body with heterogeneities in the form of periodically alternating thin horizontal layers.  
To describe the bilayered structure mathematically, consider the periodicity cell $Y:=[0,1)^2$,  which we  subdivide into 
$Y=\Ysoft\cup \Yrig$ with $\Ysoft:=[0,1)\times [0, \lambda)$ for $\lambda\in (0,1)$ and $\Yrig:=Y\setminus \Ysoft$. All sets are extended by periodicity to  $\R^2$.  The (small) parameter $\eps>0$ describes the thickness of a pair (one rigid, one softer) of fine layers,  and can be viewed as the intrinsic length scale of the system. The  collections of all rigid and soft layers in $\Omega$ can be expressed as $\eps\Yrig\cap\Omega$ and $\eps\Ysoft\cap\Omega$, respectively.  For an illustration of the geometrical assumptions,  see Figure \ref{fig:material}.

\begin{center}
\begin{figure}[h]
\begin{tikzpicture}[scale=1.5]
  \fill [Korange!20] plot [smooth, tension=0.65] coordinates
{(4., -.75)     (3.75, -.25) (2, 0)   
         (.85, -.15)      (.3,  -.5)   (.25, -.75)
        (.4, -1) (.6, -1.25) 
        (0, -1.35) (-.3, -1.5)
        (.1, -1.9)  (1, -2) (2.1, -1.9)  (3.5, -1.5) (4., -.75)};
  \draw(.6,-1.7) node      
        { \begin{normalsize}$\textcolor{black}{\Omega \subset
\RR^2}$\end{normalsize}};       
               \draw[black, thick,
        densely dotted] (2,-1) -- (2.6,-1) -- (2.6, -1.25) --
        (2, -1.25) -- (2,-1);
       \draw[black, thick,
        densely dotted] (.9+4,-3+2.5) -- (3.5+4,-3+2.5) -- (3.5+4,
-4.25+2.5) --         (.9+4, -4.25+2.5) -- (.9+4,-3+2.5);
        \draw[thick,gray!70,
         dashed] (.9+4,-3+2.5) -- (2,-1);
        \draw[thick,gray!70,
         dashed] (2.,-1.25) -- (.9+4,-4.25+2.5);
         \foreach \y in {-4.25+2.5, -4+2.5, -3.75+2.5, -3.5+2.5,
-3.25+2.5} {\fill[fill=Korange!40] (0.9+4,\y+.1) -- (3.5+4,\y+.1)
-- (3.5+4,\y+.25) -- (.9+4,\y+.25) -- cycle;};
 \draw[densely dashed, |-|] (3.6+4,-4.25+2.5) --      (3.6+4,-4+2.5);
  \draw(3.7+4.02,-4.13+2.5) node      
        {\begin{normalsize}$\textcolor{black}{\varepsilon}$\end{normalsize}}; \end{tikzpicture}

\vskip6mm

\begin{tikzpicture}[scale=.8]
%
%

       \draw[black,
        densely dotted] (1,-3) -- (3.5,-3) -- 
        (3.5, -5.5) --
        (1, -5.5) -- (1,-3);       
        
  \fill[fill=Korange!50]
(1,-3) -- (3.5,-3) -- (3.5,-4.5) -- (1,-4.5) -- cycle;      
        
  \draw(1.5,-3.4) node      
        { \begin{normalsize}$\textcolor{black}
        {Y_{\rm rig}} $\end{normalsize}};

 \draw(1.5,-4.9) node      
        { \begin{normalsize}$\textcolor{black}
        {Y_{\rm soft}} $\end{normalsize}};

  \draw(2.25,-6) node      
        { \begin{normalsize}$\textcolor{black}
        {Y=[0,1)^2} $ {reference cell} \end{normalsize}}; 
       
 \draw[densely dashed, |-|] (3.65,-5.5) --      (3.65,-4.5);
  \draw(3.85,-5) node      
        {\begin{normalsize}$\textcolor{black}{\lambda}$\end{normalsize}};
%
 
  \draw(2.8,-5.2) node      
        {\begin{small}$\textcolor{blue}
        {\underrightarrow{  s=e_1  }}$\end{small}};

\end{tikzpicture}
\caption{A bilayered $x_1$-connected domain \(\Omega\)}\label{fig:layeredmaterial}
\label{fig:material}
\end{figure}
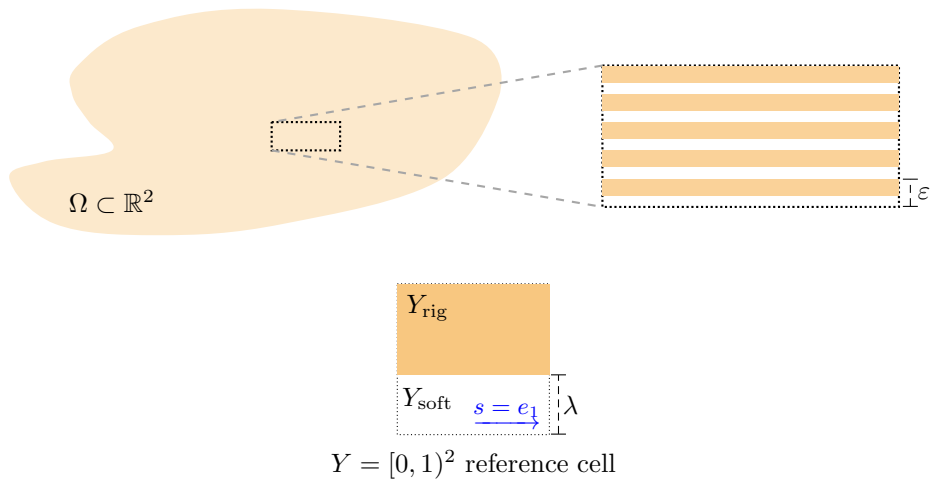
\end{center}

Following the classical theory of elastoplasticity at finite strains (see, e.g., \cite{hill} for an overview), we assume  that  the gradient of any deformation $u:\Omega\to \mathbb{R}^2$ decomposes into the product of an elastic strain, $F_{\rm el}$, and a plastic one, $F_{\rm pl}$. In the literature, different models of finite plasticity have been proposed (see, e.g., \cite{amstutz.vangoethem, davoli.francfort, grandi.stefanelli1, grandi.stefanelli2, naghdi}), as well as alternative descriptions via the theory of structured deformations (see \cite{ChDPFoOw99, ChFo97, DPOw93, BaMaMoOw17} and the references therein).   Here, we adopt the classical model by Lee on finite crystal plasticity introduced in \cite{lee, mielke, mielke2}, according to which the  deformation gradients  satisfy   
\begin{align}\label{multi_decomposition}
\nabla u=F_{\rm el}F_{\rm pl}.
\end{align}
 In addition, we suppose  that 
the elastic behavior of the body is purely rigid,  meaning  that 
\begin{align}\label{Fel}
\text{$F_{\rm el}\in SO(2)$ almost everywhere in $\Omega$,}
\end{align} 
and that the plastic part satisfies 
\begin{align}\label{Fpl}
F_{\rm pl}=\Ibb+\gamma s\otimes m,
\end{align}
 where $s\in \mathbb{R}^2$ with $|s|=1$ is the slip direction of the slip system, $m=s^{\perp}$ is the normal to the slip plane, and the map $\gamma$ measures the amount of slip. Denoting by $\Mcal_s$ the set
$$\Mcal_{s}:=\{F\in \R^{2\times 2}: \det F=1\text{ and $|Fs|
= 1$}\},$$
the multiplicative decomposition~\eqref{multi_decomposition} (under assumptions~\eqref{Fel} and~\eqref{Fpl}) is equivalent to $\nabla u\in \Mcal_s$ almost everywhere in $\Omega$. 
Whereas the material is free to glide along the slip system in the softer phase, it is required that $\gamma$ vanishes on the layers consisting of
a rigid material, i.e.,~$\gamma=0$ in $\eps\Yrig\cap \Omega$.

 Collecting the previous modeling assumptions, we define,  for $\eps>0$, the class $\Acal_\eps$ of \emph{admissible
layered deformations} by
 \begin{align}\label{eq:def-Aep}
\Acal_\eps& := \{ u\in W^{1,1}(\Omega;\R^2): \nabla u \in \Mcal_{s}
\text{ a.e.~in $\Omega$, $\nabla u\in SO(2)$ a.e.~in $\eps\Yrig\cap
\Omega$} \} \nonumber\\&\ =
\{ u\in W^{1,1}(\Omega;\R^2): \nabla u = R(\Ibb+\gamma s\otimes m)
\text{ a.e.~in $\Omega$,} \\ &\hspace{3.5cm} \text{ $R\in L^\infty(\Omega;SO(2))$ and  $\gamma\in L^1(\Omega)$ with $\gamma=0$ a.e.~in $\eps\Yrig\cap
\Omega$} \}. \nonumber
 \end{align}

The elastoplastic energy of a deformation $u\in L_0^1(\Omega;\R^2):=\{u\in L^1(\Omega;\R^2): \int_\Omega u\dd{x}=0\}$, given by
\begin{align}
E_\eps(u) 
&=\begin{cases}
\displaystyle \int_\Omega \abs{\gamma}\dd{x} & \text{for $u\in \Acal_\eps$,}
\\
\infty & \text{otherwise in } L^1_0(\Omega;\R^2), 
\end{cases}\label{eq:defEeps}
\end{align}           
represents the internal energy contribution of the system  during a single incremental step in a time-discrete variational description. This way of modeling excludes preexistent plastic distortions,
and can be considered a reasonable assumption for the first time step of a deformation process. The elastoplastic energy  can
be  complemented with terms modeling the work done by external body  or  surface forces. 

The limit behavior of sequences $(u_\eps)_\eps$ of low energy states for $(E_\eps)_\eps$ gives information about the macroscopic material response of the layered composites. 
 In the following, we  focus the analysis of this asymptotic behavior on the  $s=e_1$
case, when the slip direction is parallel to the orientation of the layers, cf.~also Figure~\ref{fig:layeredmaterial}. 
 Note that different slip directions can be treated  similarly, but the arguments are technically more involved. In fact, for $s\notin\{e_1, e_2\}$, small-scale laminate microstructures on the softer layers need to be taken into account, which requires an extra relaxation step. We refer to~\cite{CoT05} for the relaxation mechanism and to~\cite{ChK17} for the strategy of how to apply it to layered structures. \\

 An important first step towards identifying the limit behavior of the energies $(E_{\eps})_{\eps}$ (in the sense of $\Gamma$-convergence) is the proof of a general statement of asymptotic rigidity for layered structures in the context of functions of bounded variation.  The following result characterizes the weak$^\ast$ limits in $BV$ of  deformations whose gradients coincide pointwise with rotations on the rigid layers of the material.  Note that no additional constraints  are imposed on the softer components at this point.  
   
 \begin{theorem}[Asymptotic rigidity of layered structures in \boldmath{$BV$}]\label{thm:aymrigintro}
Let \(\Omega\subset\RR^2\) be an \(x_1\)-connected domain. 
Assume that $(u_\eps)_\eps\subset W^{1,1}(\Omega;\R^2)$ is a sequence satisfying
\begin{equation}
\label{eq:rigcond}
\begin{aligned}
\text{$\nabla u_\epsi\in
SO(2)$ 
a.e.\! in
$\eps\Yrig\cap \Omega$ for all $\eps$},
\end{aligned}
\end{equation}
  and  that   
$u_\eps\weaklystar u$ in $BV(\Omega;\R^2)$
for some $u\in BV(\Omega;\R^2)$  as $\eps\to 0$. Then,
\begin{equation}
\label{eq:rigcondlim}
\begin{aligned}
\text{
$u(x) =  R(x_2)x +\psi(x_2)$ for { $\Lcal^2$-} a.e. $x\in\Omega$,}
\end{aligned}
\end{equation}
where   $R\in
BV(a_\Omega,b_\Omega;SO(2))$ and $\psi\in BV(a_\Omega,b_\Omega;\R^2)$ (cf.~\eqref{aOmega}). 

 Conversely, any function
\(u\in BV(\Omega; \RR^2)\)  as in  \eqref{eq:rigcondlim}  can be attained as weak$^\ast$-limit in \(BV(\Omega; \RR^2)\) of a sequence  $(u_\eps)_\eps\subset W^{1,1}(\Omega;\R^2)$
satisfying \eqref{eq:rigcond}.   
\end{theorem}  

 To prove the first part of Theorem~\ref{thm:aymrigintro}, we adapt  the arguments in \cite{ChK17} to the $BV$-setting.  The second assertion  follows from a tailored one-dimensional density result in $BV$, which involves approximating functions that  are constant  on the rigid layers (see Lemma \ref{lem:approx_layers} below). 
Up to minor adaptations, analogous  statements  hold in higher dimensions. We refer to Remark \ref{rk:high-dim} for the specific assumptions on the geometry of the set $\Omega$ under which a higher-dimensional counterpart of Theorem \ref{thm:aymrigintro} can be proved.\\

A natural potential candidate for the  limiting behavior of $(E_{\eps})_{\eps}$ in the sense of $\Gamma$-convergence (see \cite{Bra05,Dal93}  for an introduction, as well as the references therein) is the functional
$E: L^1_0(\Omega;\R^2) \color{black}\to [0,\infty]$,  given by
\begin{align}
E(u)= \begin{cases} \displaystyle\int_\Omega |\psi'\cdot Re_1| \dd{x} + \abs{D^su}(\Omega)
& \hbox{ if } u\in  \Acal,\\
\infty & \hbox{ otherwise,} 
\label{eq:defE}
\end{cases}
\end{align} 
where 
\begin{align}\label{eq:def-A}
\begin{array}{l}
\Acal:= \{u\in BV(\Omega;\R^2)\!:\, u(x) =  R(x_2)x +\psi(x_2) \text{ for a.e.\! $x\in\Omega$ with}\\[0.2cm]
\hspace{2cm} \text{$R\in BV(a_\Omega,b_\Omega;SO(2))$, $\psi\in BV(a_\Omega,b_\Omega;\R^2)$}, \text{ and }\det \nabla u=1 \text{ a.e.\! in } \Omega\}. 
\end{array}
\end{align}

We refer to Remark~\ref{rmk:formulationsofEe} for an alternative representation of the functional \(E\).

The next theorem states that $E$ provides indeed a lower bound  for our homogenization problem.

\begin{theorem}[Lower bound on the \boldmath{$\Gamma$}-limit of \boldmath{$(E_\eps)_\eps$}\color{black}]\label{theo:Gamma_convergence_new}
Let \(\Omega\subset\RR^2\) be an \(x_1\)-connected domain, and let \(E_\epsi\) and \(E\) be the functionals introduced in \eqref{eq:defEeps} and \eqref{eq:defE}, respectively.  
Then, every sequence \((u_\eps)_\eps\subset L^1_0(\Omega;\RR^2)\) with uniformly bounded energies,  \(\sup_\epsi E_\eps(u_\eps)<\infty\), has a subsequence that converges weakly$^\ast$ in $BV(\Omega;\RR^2)$ to some \(u\in\Acal \cap L^1_0(\Omega;\RR^2) \). Additionally,
\begin{align}\label{eq:GliminfEeps}
 \Gamma(L^1) \text{-}\liminf_{\eps\to 0} E_\eps \geq E.
\end{align}
\end{theorem}

The proof of the first assertion is given in Proposition~\ref{con:compactness}. It relies on Theorem~\ref{thm:aymrigintro} in combination with a technical argument about the weak continuity properties of Jacobian determinants (see Lemma~\ref{lem:det}). 
In Section \ref{sec:par}, we   exhibit  two different  proofs of~\eqref{eq:GliminfEeps}: A first one relying on a Reshetnyak's lower semicontinuity  theorem (see, e.g.,~\cite[Theorem~2.38]{AFP00}),  and an   alternative one  exploiting the properties of the admissible
layered deformations.  The identification of $E$ as the $\Gamma$-limit of the sequence $(E_{\eps})_{\eps}$, though, remains an open problem. Indeed, verifying the optimality of the lower bound in Theorem \ref{theo:Gamma_convergence_new}  is rather challenging, as it requires to approximate elements of $\Acal$ by means of sequences in $\Acal_{\eps}$ at least in the sense of the strict convergence  in $BV$. We refer to Remark \ref{rmk:optimalitylb} for a  detailed  discussion of the main difficulties.  Even if the requirement on the convergence of the energies is dropped, recovering the jumps of maps in the effective domain of $E$ under consideration of  the non-standard differential inclusions in $\Acal_\eps$ is by itself another challenging problem.
Solving this problem requires   delicate geometrical constructions, which are currently not available for all elements in $\Acal$. 

Yet, there are two subclasses of physically relevant deformations in $\Acal$ for which we can find suitable approximations by sequences of admissible layered deformations. The precise statement is given in Theorem~\ref{thm:SBVinfty} below. 

The first of these two subclasses is
 $\Acal\cap
SBV_{\infty}(\Omega;\R^2)$ (we refer
to Subsection \ref{subs: SBV} for the definition
of the set $SBV_{\infty}$)
whose
   jump  sets
are given by a    union of finitely many
lines.
 Heuristically, this subclass  describes deformations that break $\Omega$ horizontally into a finite number of pieces,  which may get sheared and rotated individually.
  
 The second subclass is 
\begin{equation}\label{eq:def-Apar}
\begin{array}{l}
\Acal^{\parallel}
:= \big\{u\in BV(\Omega;\R^2)\!:\,  \,u(x) =  Rx +\vartheta(x_2) Re_1+c
\text{ for a.e.\! }x\in  \Omega \text{ with } 
\\[0.2cm]   
\qquad\qquad \qquad \qquad \qquad \ \  R\in SO(2), \text{ $\vartheta\in BV(a_\Omega,b_\Omega)$}, \text{ and
 }  c\in\RR^2 
\big\}.
\end{array}
\end{equation}
 
In comparison with $\Acal$, functions in $\Acal^\parallel$ satisfy two additional constraints,  namely the fact
that the rotation $R$ is constant and that the jumps of functions in $\Acal^\parallel$ are parallel to
$Re_1$. With the notation $\Acal^{\parallel}$, we intend to highlight the second feature. 
The intuition behind maps in $\Acal^\parallel$ are non-trivial macroscopic deformations that (up to a global rotation) may make the material break along finite or infinitely many horizontal lines, induce sliding of the pieces relative to each other, and cause horizontal shearing within each individual piece.  
For an illustration of the two subclasses,  see Figure~\ref{fig:AparAinfty}.

\begin{center}
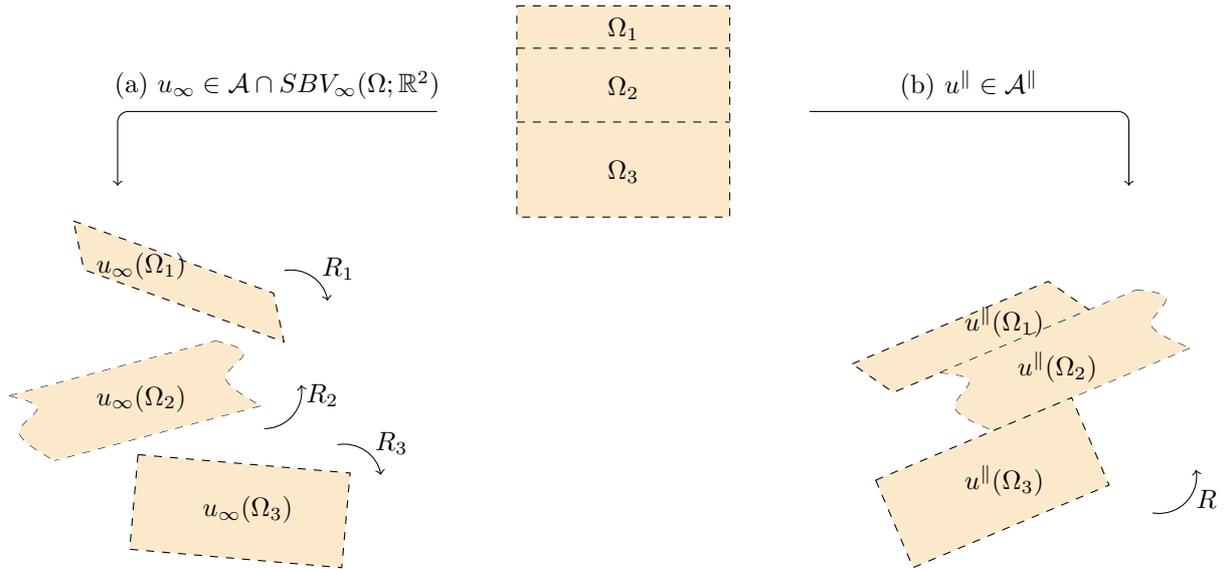
\begin{figure}[h]
\begin{tikzpicture}[scale=.7]   
\draw[dashed, black, fill=Korange!20] 
(0,0) -- (4,0) -- (4, -4) -- (0, -4)
 -- (0,0);

\draw[dashed, black] 
(0,-.8) -- (4,-.8) ;  
  
 \draw[dashed, black] 
(0,-2.2) -- (4,-2.2) ; 

\draw(2,-.45)  node 
        {$\Omega_1$};
        
\draw(2,-1.55)  node 
        {$\Omega_2$};
        
        \draw(2,-3.15)  node 
        {$\Omega_3$};

\draw [rounded corners,->] (5.5,-2) -- ++(6,0pt) -- ++(0,-40pt)
;
\draw(8.5,-1.5)  node 
        {(b) \(u^\parallel \in\mathcal{A}^\parallel  \)};

\draw [rounded corners,->] (-1.5,-2) -- ++(-6,0pt) -- ++(0,-40pt)
;
\draw(-4.5,-1.5)  node 
        {(a) \(u_\infty \in \mathcal{A}\cap SBV_\infty(\Omega;\RR^2)
 \)};
 
\end{tikzpicture} 
 
    {%
        \begin{tikzpicture}[scale=.7]

\draw[rotate=-20,
dashed, black, fill=Korange!20] 
\foreach \y in {3} \foreach \x in {0}{(5,0+\y) -- (9,0+\y) --
(9.5,-.8+\y) -- (5.5,-.8+\y) -- (5,0+\y)} ;

\draw [<-] (10.5,-.4) arc (10:100:20pt);
\draw(10.7,.2)  node 
        {$R_1$};

\draw[rotate=15,
dashed, black] 
\foreach \y in {2.5} \foreach \x in {1.2} {(5-\x,-.8-\y) -- (9-\x,-.8-\y)
 (9.5-\x,-2.2-\y) -- (5.5-\x,-2.2-\y)} ;   

\draw[rotate=15,
dashed, black,name
path=A] 
\foreach \y in {2.5} \foreach \x in {1.2} {plot [smooth, tension=0.65]
coordinates
{(5-\x,-.8-\y)     (5.45-\x, -1.15-\y) (5.-\x, -1.55-\y) (5.5-\x,-2.2-\y)
           }};
\draw[rotate=15,
dashed, black,name
path=B] 
\foreach \y in {2.5} \foreach \x in {1.2} {plot [smooth, tension=0.65]
coordinates
{(9-\x,-.8-\y)     (9.45-\x, -1.15-\y) (9.-\x, -1.55-\y) (9.5-\x,-2.2-\y)
         }          
         }       ;
\tikzfillbetween[of=A and B]{Korange!20};

\draw [<-] (10,-2) arc (10:-90:20pt);
\draw(10.4,-2.2)  node 
        {$R_2$};

\draw[rotate=-5,
dashed, black, fill=Korange!20] 
\foreach \y in {-.5} \foreach \x in {2.2} {(5+\x,-2.2+\y) --
(9+\x,-2.2+\y)
-- (9+\x,-4+\y) -- (5+\x,-4+\y) -- (5+\x,-2.2+\y)} ;
\draw [<-] (11.5,-3.7) arc (10:100:20pt);
\draw(11.7,-3.1)  node 
        {$R_3$};

\draw(7,-.45+.7)  node 
        {$u_\infty(\Omega_1)$};
        
\draw(7,-1.55-.7)  node 
        {$u_\infty(\Omega_2)$};
        
        \draw(9,-3.15-1.2)  node 
        {$u_\infty(\Omega_3)$};

        \end{tikzpicture}  %
        }%
    \hfill%
    {%
  \begin{tikzpicture}[scale=.7]
 \draw[rotate=23,
dashed, black, fill=Korange!20] 
\foreach \y in {0}  \foreach \x in {-0.955}{
(5-.5,0+\y) -- (9-.5,0+\y)
--
(9.5-.5,-.8+\y) -- (5.5-.5,-.8+\y) -- (5-.5,0+\y)};

 \draw[rotate=23,
dashed, black] 
\foreach \y in {0}  \foreach \x in {-0.955}{
(5-\x,-.8+\y) -- (9-\x,-.8+\y)
 (9.5-\x,-2.2+\y) -- (5.5-\x,-2.2+\y)};    

 \draw[rotate=23,
dashed, black, name
path=C] 
\foreach \y in {0}  \foreach \x in {-0.955}{
plot [smooth, tension=0.65] coordinates
{(5-\x,-.8+\y)     (5.45-\x, -1.15+\y) (5.-\x, -1.55+\y) (5.5-\x,-2.2+\y)
           }};
 \draw[rotate=23,
dashed, black, name path=D] 
\foreach \y in {0}  \foreach \x in {-0.955}{           
plot [smooth, tension=0.65] coordinates
{(9-\x,-.8+\y)     (9.45-\x, -1.15+\y) (9.-\x, -1.55+\y) (9.5-\x,-2.2+\y)
         }};  
         \tikzfillbetween[of=C and D]{Korange!20};
         \draw[rotate=23,
dashed, black, fill=Korange!20] 
\foreach \y in {0}  \foreach \x in {-0.955}{ (5+\x,-2.2+\y) --
(9+\x,-2.2+\y)
-- (9+\x,-4+\y) -- (5+\x,-4+\y) -- (5+\x,-2.2+\y)} ;
\draw [<-] (10.6,-.25) arc (10:-100:20pt);
\draw(10.8,-.8)  node 
        {$R$};
        
\draw(7,-.45+3)  node 
        {$u^\parallel(\Omega_1)$};
        
\draw(8,-1.15+2.85)  node 
        {$u^\parallel(\Omega_2)$};
        
        \draw(7,-3.15+2.7)  node 
        {$u^\parallel(\Omega_3)$};

   \end{tikzpicture}%
        
        }%
    \caption{A typical deformation of a reference configuration
    \(\Omega= \Omega_1 \cup \Omega_2 \cup \Omega_3\)
    through maps in (a) \(\mathcal{A}\cap SBV_\infty(\Omega;\RR^2)\)
    and (b) \(\mathcal{A}^\parallel\).}
\label{fig:AparAinfty}
\end{figure}
\end{center}

 \begin{theorem}[Approximation of maps in \boldmath{$(\Acal\cap SBV_{\infty}) \cup
\Acal^\parallel$}]\label{thm:SBVinfty}
 Let \(\Omega\subset\RR^2\) be an \(x_1\)-connected domain 
and $u\in (\Acal\cap SBV_{\infty}(\Omega;\R^2)) \cup
\Acal^\parallel$. Then, there exists a sequence $(u_\eps)_\eps\subset W^{1,1}(\Omega;\R^2)$ such that
$u_\eps\in \Acal_\eps$ for every $\eps$, and $u_\eps\weaklystar
u$ in $BV(\Omega;\R^2)$. 
\end{theorem}

 As a first step towards proving Theorem~\ref{thm:SBVinfty}, we  establish an admissible piecewise affine approximation for limiting deformations with  a single jump line  (see Lemma \ref{lem:existence_admissible1}).  The construction relies on the characterization of rank-one connections in $\Mcal_{e_1}$ proved in \cite[Lemma 3.1]{ChK17}, with transition lines stretching over the full width of $\Omega$ to avoid triple junctions (see Remark \ref{rmk:alternativeconst}).  In Propositions \ref{prop:SBVinfty} and \ref{prop:Aparallel}, we extend the arguments to $\Acal\cap SBV_{\infty}(\Omega;\R^2)$ and $\Acal^{\parallel}$, respectively. \\

Problems in finite crystal plasticity without additional regularizations are  generally  known to be challenging because of the oscillations of minimizing sequences arising as a byproduct of relaxation mechanisms in the slip systems.  This phenomenon is  one of the main reasons why a full relaxation theory in finite crystal plasticity is still missing (see \cite[Remark 3.2]{CDK13}).  In our setting,  it hampers  the full characterization of weak limits of sequences with uniformly bounded energies. 
  The observation that regularizations can help overcome the above compensated-compactness issue (see also Remark \ref{rk:cc}) motivates the introduction of a penalized version of our problem. 
 After  a  higher-order penalization of
the energy in the layer direction, we obtain the following
$\Gamma$-convergence result.  The attained limit deformations are given by the class $\Acal^{\parallel}$.

\begin{theorem}[$\Gamma$-convergence of the  regularized energies]\label{thm:Gamma_convergence}
Let \(\Omega\subset\RR^2\) be an \(x_1\)-connected domain and
\(\Acal_\epsi\) the set introduced in \eqref{eq:def-Aep}. Fix $p>2$ and $\delta>0$.  
For each $\eps>0$, let \(E^\delta_\epsi: L^1_0(\Omega;\R^2)\to [0,\infty]\) be the functional defined  
by
\begin{equation}\label{eq:def-E-ep-delta}
E^\delta_\eps(u) 
:=\begin{cases}
\displaystyle \int_\Omega \abs{\gamma}\dd{x}  + \delta \norm{\partial_1 u}_{W^{1,p}(\Omega;\R^2)}^p
& \text{for $u\in \Acal_\eps$,}\\ 
\infty & \text{otherwise.} 
\end{cases}
\end{equation}
Then, the family $(E_\eps^\delta)_\eps$ $\Gamma$-converges
with respect to the strong $L^1$-topology to the functional 
$E^\delta:L^1_0(\Omega;\R^2)\to [0,\infty]$ given by 
\begin{align*}
E^\delta(u):= \begin{cases} \displaystyle   \int_\Omega  |\vartheta'(x_2)|
\dd{x}  + \abs{D^su}(\Omega) +\ \delta|\Omega|  & \hbox{ for } u\in \Acal^{\parallel},\\
\infty & \hbox{ otherwise,}
\end{cases}
\end{align*}
 where $\vartheta'$ denotes the approximate differential of $\vartheta$  (cf.~Section~\ref{subs:BV}). 
\end{theorem}

The penalization in \eqref{eq:def-E-ep-delta} can be viewed in the spirit of non-simple materials \cite{toupin1, toupin2}.   Working with stored energy densities that depend  on the Hessian of the deformations has proved successful in overcoming lack of compactness in a variety of applications;  see, e.g.,~\cite{ball-currie, DFL04, friedrich.kruzik, mielke.roubicek, podioGuidugli}.  Very recently, there has been an effort towards weakening higher-order regularizations:  It is shown  in \cite{BeKrSc} that the full norm of the Hessian can be replaced by a control of  its  minors (gradient polyconvexity) in the context of locking materials;
for solid-solid phase transitions, an anisotropic second-order penalization  is  considered in \cite{davoli.friedrich}. Along these lines, we introduce  the regularized energies in~\eqref{eq:def-Apar} that penalize the variation of deformations only in the layer direction. This is enough to deduce that the limiting rotation (as $\eps\to 0$) is global and that it determines the direction of the limiting jump.  In Section \ref{sect:regularization}, we provide two alternative proofs of this result: A first one relying on Alberti's rank one theorem (see Section \ref{sec:prel})  in combination with the approximation result in Theorem~\ref{thm:SBVinfty},  and a second one based on separate regularizations of the regular and the singular part of the limiting maps, and inspired by \cite[Lemma~3.2]{CrDe11}.\\

This paper is organized as follows. In Section \ref{sec:prel}, we collect  a few preliminaries, including some background on (special) functions of bounded variation.  Section \ref{sec:as} is devoted to the analysis of asymptotic rigidity  for layered structures in the setting of $BV$-functions.  A characterization of limits of admissible layered deformations is provided in Section \ref{sec:ad}. Eventually, Sections \ref{sec:par} and \ref{sect:regularization} contain the proof of a lower bound for the homogenization problem without regularization (Theorem~\ref{theo:Gamma_convergence_new}) and the full $\Gamma$-convergence analysis of the regularized problem (Theorem~\ref{thm:Gamma_convergence}), respectively.

 \section{Preliminaries} 
\subsection{Notation}\label{sec:prel}
In this section,  unless mentioned otherwise, $\Omega$ is  a bounded domain in $\mathbb{R}^{N}$ with $N\in \mathbb{N}$.  Throughout the rest of the paper, we assume mostly that $N=2$. 

We represent by \(\Lcal^N\)
the \(N\)-dimensional Lebesgue measure
and by \(\Hcal^{N-1}\) the \((N-1)\)-dimensional Hausdorff measure.
 Whenever we write ``a.e.\! in \(\Omega\)", we mean   ``almost everywhere
in \(\Omega\)'' with respect to \(\Lcal^N\lfloor\Omega\). To simplify
the notation, we often omit the expression ``a.e.\! in \(\Omega\)"  in mathematical relations involving Lebesgue  measurable
functions. 
Given a  Lebesgue  measurable
set \(B\subset \RR^N\),  we also use the shorter notation $|B|=\mathcal{L}^N(B)$  for the Lebesgue measure of \(B\), while   the characteristic function of $B$ in $\R^N$ is denoted by \(\mathbbm{1}_B\) and
takes values  \(0\) and \(1\).  
 
The  set  \(SO(N):=\{R\in
\RR^{N\times N}\!:\, R R^T = \Ibb,\, \det R =1\}\), where \(\Ibb\)
is the identity matrix in \(\RR^{N\times N}\),   consists of all  proper rotations.
 We recall that for \(N=2\), \(R\in SO(2)\) if and only if there is \(\theta\in[-\pi,\pi) \) such that
\begin{equation*} 
\begin{aligned}
R = \begin{bmatrix}\cos \theta & -\sin \theta
\\
\sin \theta & \cos \theta \\
\end{bmatrix}.
\end{aligned}
\end{equation*}
 For  two vectors $a, b\in \R^d$,
$a \otimes b:=ab^T$ stands for their tensor product. If $a=(a_1, a_2)^T\in \R^2$, we set $a^\perp := (-a_2, a_1)^T$. 

 We use the standard notation for spaces of vector-valued functions; namely, $L^p_\mu(\Omega;\R^d)$ with $p\in [1, \infty]$ and a positive measure $\mu$  for $L^p$-spaces, $W^{1,p}(\Omega;\R^d)$ with $p\in [1, \infty]$ for Sobolev spaces, $C(\Omega;\R^d)$ for the space of continuous functions, $C^\infty(\Omega;\R^d)$ and $C^\infty_c(\Omega;\R^d)$ for the spaces of smooth functions without and with compact support, and $C^{ 0, \alpha}(\Omega;\R^d)$ with $\alpha\in [0,1]$ for H\"older spaces. We  denote by $C_0(\Omega;\R^d)$ the space of continuous functions that vanish on the boundary of \(\Omega\).   Moreover, $\Mcal(\Omega;\R^d)$ is the space of finite vector-valued Radon measures. 
In the case of scalar-valued functions and measures, we omit the codomain; for instance,  we write $L^1(\Omega)$ instead of $L^1(\Omega;\R)$. 

The duality pairing between $C_0(\Omega;\R^d)$ and  $\Mcal(\Omega;\R^d)$ is represented by $\langle \mu, \zeta \rangle := \int_\Omega \zeta \dd{\mu}$, and   
 $\mu \boldsymbol \otimes \nu$  denotes the product measure of two measures $\mu$ and $\nu$.  \color{black}

Throughout this manuscript, \(\epsi\) stands for a  small (positive)
parameter, and is usually thought of as taking values on a positive
sequence converging to zero.

\subsection{Functions of bounded variation}\label{subs:BV} 
We  adopt the standard notations
for the space $BV(\Omega;\mathbb{R}^d)$ of
 vector-valued  functions of bounded variation, and refer the reader to \cite{AFP00}
 for a thorough treatment of this space. Here, we only recall
 some of its basic properties.

A function $u\in L^1(\Omega;\mathbb{R}^d)$ is  called  a \textit{function of bounded variation}, written $u\in BV(\Omega;\mathbb{R}^d)$, if its distributional derivative $Du$ satisfies $Du\in \mathcal{M}(\Omega;  \mathbb{R}^{d\times N})$. The space \(BV(\Omega;\mathbb{R}^d)\) is a Banach space
when endowed with the norm \(\Vert u\Vert_{BV(\Omega;\mathbb{R}^d)}
:= \Vert u\Vert_{L^1(\Omega;\mathbb{R}^d)} + |Du|(\Omega)\), where \(|Du|\in \Mcal(\Omega)\) is the
 total variation of \(Du\). 

Let  $D^a u$ and $D^s u$ denote the absolutely continuous and the singular part of the Radon--Nikodym decomposition of $Du
$ with respect to \(\Lcal^N\lfloor\Omega\),  and let $D^j u$ and $D^c u$  be  the jump and Cantor parts of $Du
$. The following chain of equalities holds:
\begin{align}\label{Du}
 Du&=D^a u+D^s u= \nabla u \mathcal{L}^N\lfloor\Omega+ D^su = \nabla u \mathcal{L}^N\lfloor\Omega + D^ju+D^c u \nonumber \\
&=\nabla u \mathcal{L}^N\lfloor\Omega+(u^+-u^-)\otimes \nu_u\mathcal{H}^{N-1}\lfloor J_u+D^c u,
\end{align}
where $\nabla u$ is the approximate differential of $u$ (that is, the density of \(D^a u\)), $u^+$ and $u^-$ are the approximate one-sided limits at the jump points, $J_u$ is the jump set of $u$, and $\nu_u$ is the normal to $J_u$ (cf.~\cite[Chapter~3]{AFP00}). 

Following \cite[p.
186]{AFP00}, we can exploit the polar decomposition of a measure and
the fact that all parts of the derivative of \(u\) in~\eqref{Du} are mutually
singular to write $Du=g_u|Du|$ with a map $g_u\in L^1_{|Du|}(\Omega;  \mathbb{R}^{d\times N})$ satisfying $|g_u|=1$ for \(|Du|\)-a.e.\! \(x\in\Omega\) and
\begin{equation*}
\begin{aligned}
 D^au = g_u|D^a u|,  \quad D^su=g_u|D^su|,\quad  D^j u = g_u|D^j u|, \quad D^c u =
g_u|D^c u|.
\end{aligned}
\end{equation*}
Note that%
 \begin{align}
&g_u(x) = \frac{\nabla u(x)}{|\nabla u(x)|} \text{ for } \Lcal^N\text{-a.e.\!
} x\in \Omega \text{ such that } |\nabla u(x)| \not=0, \nonumber\\
& g_u(x) = \frac{u(x^+) - u(x^-)}{|u(x^+) - u(x^-)|}\otimes \nu_u  (x)
\text{ for } \Hcal^{N-1}\text{-a.e.\!
} x\in J_u, \label{eq:polarDj}\\
& g_u(x) = \bar g_u(x)\otimes  n_u (x) 
\text{ for } |D^c u|\text{-a.e.\!
} x\in \Omega \text{  with suitable  Borel maps } \bar g_u:\Omega \to  \RR^{d} ,\,  n_u :\Omega\to \RR^N.
\label{eq:polarDc}
\end{align}
The last equality  relies on  Alberti's rank-one theorem (see \cite{alberti}).
 
 Let $u\in BV(\Omega;\mathbb{R}^d)$ and \((u_j)_{j\in \N} \subset  BV(\Omega;\mathbb{R}^d)\) be a sequence.  One says  that $(u_j)_{j\in \N}$
 weakly* converges to   $u$ in $BV(\Omega;\mathbb{R}^d)$,
 written \(u_j\weaklystar u\) in \(BV(\Omega;\mathbb{R}^d)\),
 if \(u_j \to u\) in \(L^1(\Omega;\mathbb{R}^d)\) and \(Du_j \weaklystar Du\) in \(\Mcal(\Omega;\mathbb{R}^{d\times N})\).  The sequence \((u_j)_{j\in \N}
\) is said to  converge strictly to $u$ in $BV(\Omega;\mathbb{R}^d)$,
 written \(u_j\tostar u\) in \(BV(\Omega;\mathbb{R}^d)\),
 if \(u_j \to u\) in \(L^1(\Omega;\mathbb{R}^d)\) and \(|Du_j|(\Omega)\to
 |Du|(\Omega)\). We recall that strict convergence
 in \(BV(\Omega;\mathbb{R}^d)\)
 implies weak*  convergence
 in \(BV(\Omega;\mathbb{R}^d)\). Moreover, from every bounded
 sequence in \(BV(\Omega;\mathbb{R}^d)\) one can extract a weakly* convergent subsequence  (see~\cite[Theorem~3.23]{AFP00}).

In the one-dimensional setting,  i.e.,~for $\varphi\in BV(a,b;\RR^d)$ with $\Omega=(a, b)\subset \R^N$ and $N=1$,  we  write \(\ffi'\) in place of
 \(\nabla \ffi\) to denote the approximate differential of \(\ffi\). Accordingly, we use the notation
$Du=\varphi'\Lcal^1 +D^s\varphi$
for the decomposition of the distributional derivative of \(\ffi\)
with respect to the Lebesgue measure. 

A function $\varphi \in BV(a,b;  \R^d)$ is called a  jump or Cantor function if $D\varphi=D^j\varphi$ or $D\varphi=D^c\varphi$, respectively.  We denote the sets of all jump and Cantor functions by $BV^j(a,b;\R^d)$ and $BV^c(a,b;\R^d)$, respectively.
As shown in~\cite[Corollary~3.33]{AFP00}, it is a special property of the one-dimensional setting that 
\begin{align}\label{1dsplitting}
BV(a,b;\R^d) = W^{1,1}(a,b;\R^d) + BV^j(a,b;\R^d) + BV^c(a,b;\R^d).
\end{align}

Throughout this paper,  two-dimensional  functions of the form
\begin{equation}
\label{eq:form-u}
\begin{aligned}
u(x) = R(x_2)x + \psi(x_2)
\end{aligned}
\end{equation}
with \(x=(x_1,x_2) \in  \Omega=Q:=  (c,d) \times(a,b)\subset\RR^2\),
where \(R\in BV(a,b;SO(2))\) and \(\psi\in BV(a,b;\RR^2)\), play
a fundamental role. 
 Maps $u$ as in \eqref{eq:form-u} satisfy \(u\in BV( \Omega ;\RR^2)\). Denoting by  \(D_1u
:=Du\otimes e_1\)
and \(D_2u:=Du\otimes e_2\), the first and second columns of \(Du\), respectively,
we have for all \(\zeta\in C_0( \Omega)\) that
\begin{equation*}
\begin{aligned}
\langle D_1u, \zeta\rangle = \, &
\int_{ \Omega} \zeta (x) R(x_2)e_1\dd{x_1}\!\!\dd{x_2},\\
\langle D_2u, \zeta\rangle = \,&\int_{ \Omega} \big(\zeta (x) R(x_2)e_2
+ R'(x_2)x+\psi'(x_2)\big)\dd{x_1}\!\!\dd{x_2}  \\ &+ \int_{ \Omega } \zeta(x)
 x_1 
\dd{x_1}\!\!\dd{D^sR(x_2)e_1} + \int_{ \Omega } \zeta(x)
 x_2 
\dd{x_1}\!\!\dd{D^sR(x_2)e_2}+ \int_{ \Omega} \zeta(x) 
\dd{x_1}\!\!\dd{D^s\psi(x_2)}. 
\end{aligned}
\end{equation*}
Hence, \(Du=D^au + D^s u\) with
\begin{equation}\label{eq:onDu1}
\begin{aligned}
&D^au = \big(R+(R' x+\psi')\otimes e_2)\Lcal^2\lfloor  \Omega,  \\ 
&D^su=\big(\big(x^T \Lcal^1\lfloor(c,d){ \boldsymbol\otimes } D^sR^T\big)^T \color{black} + \Lcal^1\lfloor(c,d){ \boldsymbol\otimes } D^s\psi \big)  \otimes e_2,
\end{aligned}
\end{equation}
where \(\Lcal^1\lfloor(c,d){ \boldsymbol\otimes } D^sR^T\)  and  \(\Lcal^1\lfloor(c,d){ \boldsymbol\otimes } D^s\psi\)   denote the restrictions
to the Borel \(\sigma\)-algebra on  \(\Omega=Q\)  of the product measures
between \(\Lcal^1\lfloor(c,d)\) and \( D^sR^T\)   and \(D^s\psi\), respectively.

We observe further that there exists \(\theta \in BV(a,b; [-\pi,\pi])\)
such that 
\begin{equation}\label{eq:Rtrigon}
\begin{aligned}
R = \begin{bmatrix}\cos \theta & -\sin \theta \\
\sin \theta & \cos \theta\\
\end{bmatrix} \enspace \text{ and } \enspace R' =\theta'\begin{bmatrix}-\sin \theta & -\cos \theta
\\
\cos \theta& -\sin
\theta \\
\end{bmatrix},
\end{aligned} \end{equation}
where the representation of $R'$ follows from the chain rule in BV; see, e.g.,~\cite[Theorem~3.96]{AFP00}.

%


 \subsection{Special functions of bounded variation}\label{subs: SBV}
A function $u\in BV(\Omega;\RR^d)$ is said to be a \textit{special function of bounded variation}, written $u\in SBV(\Omega;\RR^d)$, if the Cantor part of its distributional derivative  satisfies
$$D^c u=0.$$
 In particular,  it holds for every  $u\in SBV(\Omega;\RR^d)$ that
$$Du=\nabla u\mathcal{L}^N  \lfloor \Omega  +(u^+-u^-)\otimes \nu_u\mathcal{H}^{N-1}\lfloor J_u.$$
The space $SBV(\Omega;\RR^d)$ is a  proper subspace  of $BV(\Omega;\RR^d)$(c.f.  \cite[Corollary 4.3]{AFP00}).

Next, we recall the definition of the space $SBV_{\infty}(\Omega;\RR^d)$ of special functions of bounded variation  with  bounded gradient and jump length, which is given by
 \begin{equation*}
 SBV_{\infty}(\Omega;\mathbb{R}^d):=\{u\in SBV(\Omega;\mathbb{R}^d)\!:\,\nabla u\in L^{\infty}(\Omega;  \RR^{d\times N})\text{ and }\mathcal{H}^{N-1}(J_u)<+\infty\}.
 \end{equation*}
It is shown in~\cite{CGP07} that the distributional curl of $\nabla u$ for $u\in SBV_{\infty}(\Omega;  \R^d)$ 
is a measure concentrated on $J_u$.

Finally, we introduce the  space 
\begin{align}\label{eq:def-PC}
PC(a,b;  \R^d ) = SBV_\infty(a, b; \R^d) \cap \{u\in BV(a,b;  \R^d)\!:\, D^a u=0\},
\end{align}
 which contains  piecewise constant one-dimensional
functions
with values in  $\R^d$.

\subsection{ Geometry  of the domain}\label{subs:geom}

In this section, we specify our main assumptions on the geometry of $\Omega$,  which, as mentioned in the Introduction, will mostly be a bounded Lipschitz domain in $\R^2$.  Let us first recall from \cite[Section 3]{ChK18} the definitions of \emph{locally one-dimensional} and \emph{one-dimensional} functions. 
\begin{definition}[Locally one-dimensional functions in the \boldmath{$e_2$}-direction]
Let $\Omega\subset \R^2$ be open. A function $f:\Omega\to  \R^d$ is \emph{locally one-dimensional in the $e_2$-direction} if for every $x\in \Omega$, there exists an open cuboid $Q_x\subset \Omega$,  containing \(x\) and with sides parallel to the standard coordinate axes, such that   for all $y=(y_1, y_2), z=(z_1, z_2)\in Q_x$, 
\begin{align}\label{loc1d}
f(y)=f(z)\qquad \text{if $y_2=z_2$.} 
\end{align}  We say that $f$ is \emph{ (globally)  one-dimensional} in the $e_2$-direction  if~\eqref{loc1d}  holds for every $y,z\in \Omega$. 
\end{definition}
Analogous arguments to those in \cite[Section 3]{ChK18} show that a function  $f\in BV(\Omega;\R^d)$  satisfying $D_1 f=0$ is locally one-dimensional in the $e_2$-direction. The following geometrical requirement is the counterpart of \cite[Definitions 3.6 and 3.7]{ChK18} in our setting. 

\begin{definition}[\boldmath{$x_1$}-connectedness]\label{def:nice-sets}
We say that an open set $\Omega\subset \R^2$ is  \emph{$x_1$-connected}   if for every $t\in \mathbb{R}$, the set  $\{x_2=t\}\cap \Omega$ is a (possibly empty) interval.
\end{definition}

In what follows, we  always assume that the set $\Omega\subset \R^2$ is an $x_1$-connected domain. Under this geometrical assumption, the notions of   locally and globally one-dimensional functions in the $e_2$-direction coincide.  We refer to \cite[Section 3]{ChK18} for an extended discussion on the topic, as well as for some explicit geometrical examples.

\section{Asymptotic rigidity of layered structures in $BV$}\label{sec:as}

In this section, we prove Theorem~\ref{thm:aymrigintro},  which characterizes the asymptotic behavior of deformations of bilayered materials that correspond to rigid body motions on the stiff layers, but do not experience any further structural constraints on the softer layers. This qualitative result is not just limited to applications in crystal plasticity, but can be useful for a larger class of layered composites where fracture may occur.

 We start by introducing some notation. Assume that \(\Omega\subset
\RR^2\) is an  \(x_1\)-connected  domain. For $\eps>0$, let 
\begin{align}
\label{eq:def-B-ep}
\Bcal_\eps& := \{ u\in W^{1,1}(\Omega;\R^2): \text{$\nabla u\in SO(2)$ in $\eps\Yrig\cap \Omega$} \}
 \end{align}
represent  the class of  \emph{layered
deformations with rigid components,}  and let  
\begin{align}
\Bcal_0:= \{&u\in BV(\Omega;\R^2):  \text{ there exists }(u_{\eps})_{\eps}\subset W^{1,1}(\Omega;\R^2) \text{ with $u_\eps\in \Bcal_\eps$ for all $\eps$}\label{eq:def-B0}\\
&\hspace{2.75cm}\text{such that } u_\eps\weaklystar u \text{ in $BV(\Omega;\R^2)$} \} \notag
\end{align} 

 be  the  associated set of asymptotically attainable
deformations.

We aim at proving that \(\Bcal_0\) coincides with the set
of \emph{asymptotically
rigid deformations} given by 
 \begin{align}
\Bcal 
:= \big\{u\in BV(\Omega;\R^2)\!:\, &\,\,  u(x) =  R(x_2)x +\psi(x_2) \text{ for a.e.\! }x\in  \Omega \notag\\
& \text{ with  $R\in BV(a_\Omega, b_\Omega;SO(2))$ and $\psi\in BV(a_\Omega, b_\Omega;\R^2)$}  \big\}, \label{eq:def-B}
\end{align}
cf.~\eqref{aOmega}. 
This identity will be a consequence of Propositions~\ref{prop:asymptotic_rigidity2}
and \ref{prop:approx1} below.

\begin{proposition}[Limiting behavior of maps in $\Bcal_{\eps}$]\label{prop:asymptotic_rigidity2}
 Let  $\Omega=(0,1)\times (-1,1)$. Then, 
\begin{align}\label{B0subsetB} \Bcal_0 \subset \Bcal,
\end{align} 
where \(\Bcal_0\) and \(\Bcal\) are the sets introduced in \eqref{eq:def-B0} and \eqref{eq:def-B}, respectively. 
\end{proposition}

\begin{proof}  The proof is inspired by and generalizes ideas from~\cite[Proposition~2.1]{ChK17}.  Let \(u\in \Bcal_0\). Then, there exists a sequence
 $(u_\eps)_\eps\subset
W^{1,1}(\Omega;\R^2)$  satisfying $\nabla u_\epsi\in
SO(2)$ a.e.\! in
$\eps\Yrig\cap \Omega$ for all $\eps$, and    
$u_\eps\weaklystar u$ in $BV(\Omega;\R^2)$.

Fix \(0<\epsi<1\), and let \(I_\epsi:=\{i\in \Z\!:\, ( \R  \times \eps(i-1, i))\cap \Omega \not= \emptyset \}\). For each
 $i\in I_{\eps}$, we define a strip, $P_\eps^i$, by setting
\begin{equation*}
\begin{aligned}
P_{\eps}^i:=( \R  \times \eps[i-1, i))\cap \Omega.
\end{aligned}
\end{equation*}
Note that
if \(i\in\Z\) is such that \(|i|>1+\lceil\frac{1}{\eps}\rceil\),
then \(i\not\in I_\epsi\). Moreover, defining \(i^+_\epsi:= \max
I_\epsi\) and \(i^-_\epsi:= \min I_\epsi\), then
\begin{itemize}
\item[i)] for \(i^-_\epsi<i<i^+_\epsi\), $P_{\eps}^i$ is the
union of two neighboring connected components of $\eps\Yrig\cap\Omega$
and $\eps\Ysoft\cap\Omega$;
\item[ii)]  we may have \(\eps\Ysoft\cap P_{\eps}^{i^-_\epsi}
=\emptyset \) or \(\eps\Yrig\cap P_{\eps}^{i^+_\epsi}
=\emptyset \).
\end{itemize}
From Reshetnyak's theorem, we infer that on each nonempty rigid layer $\eps\Yrig\cap
P_\eps^i$ with $i\in I_\eps$, the gradient $\nabla u_\eps$ is
constant and coincides with a rotation $R_\eps^i\in SO(2)$. Moreover,
there exists \(b^i_\epsi\in\RR^2\) such that \(u_\epsi(x) = R^i_\epsi
x + b^i_\epsi\) in  $\eps\Yrig\cap
P_\eps^i$.

   Using
these rotations $R_\eps^i$, we define a piecewise constant function,  
  $\Sigma_\eps:
(-1,1) \to \R^{2\times 2}$,  by setting \(\Sigma_\eps(t) = \sum_{i\in
I_\eps} R_\eps^i \mathbbm{1}_{\eps[ i-1,1)}(t)\)
for \(t\in(-1,1) \), where  \(R^{i^+_\epsi}_\epsi:= R^{i^+_\epsi-1}_\epsi\)
if \(\eps\Yrig\cap P_{\eps}^{i^+_\epsi}
=\emptyset \).
We claim that there exist a subsequence of $(\Sigma_\eps)_\eps$,
which we do not relabel,
 and a function $R\in BV( -1,1 ; SO(2))$ such that 
\begin{equation}
\label{eq:convSigma_ep}
\begin{aligned}
\Sigma_\eps \to R \quad \text{ in $L^1(-1,1;\R^{2\times 2})$.}
\end{aligned}
\end{equation}

To prove \eqref{eq:convSigma_ep}, we first observe that the total variation of the one-dimensional function $\Sigma_\eps$  coincides with its pointwise variation, and can be calculated to be 
\begin{align}
|D\Sigma_\eps| (-1,1) &  = \sum_{ i\in I_\epsi \backslash\{i^-_\epsi\}} \abs{R_\eps^i-R_\eps^{i-1}} = \sqrt{2} \sum_{
i\in I_\epsi \backslash\{i^-_\epsi\}} \abs{R_\eps^ie_1-R_\eps^{i-1}e_1}
. \label{eq:var-sigma-ep1}
\end{align}

Next, we show that the right-hand side of \eqref{eq:var-sigma-ep1}
is uniformly bounded. By  linear
interpolation in the $x_2$-direction on the softer layers, it follows
 for all  \(i\in I_\epsi \backslash\{i^-_\epsi\}\) if \(\eps\Yrig\cap P_{\eps}^{i^+_\epsi}
\not=\emptyset\) and   \(i\in I_\epsi \backslash\{i^\pm_\epsi\}\) if \(\eps\Yrig\cap P_{\eps}^{i^+_\epsi}
=\emptyset\) that
\begin{align}
\int_{\eps\Ysoft\cap P_\eps^i} \abs{\nabla u_\eps e_2} \dd{x}
 & = \int^{1}_{0} \int_{\eps(i-1)}^{\eps(i -1+\lambda)} \abs{\partial_2
u_\eps(x_1, x_2)}\dd{x_2}\dd{x_1} \notag \\ &\geq \int^1_{0} \abs{u_\eps(x_1,
\eps(i -1+\lambda)) - u_\eps(x_1, \eps(i-1))}\dd{x_1} \notag \\ & 
=  \int^1_{0} \abs{(R_\eps^{i}
e_1-R_\eps^{i-1}e_1)x_1 + b^{i}_\epsi - b^{i-1}_\epsi}\dd{x_1} \geq \frac14\abs{R_\eps^{i}
e_1-R_\eps^{i-1}e_1}. \label{eq:var-sigma-ep2}
\end{align}
 The first estimate is a consequence of Jensen's inequality, and optimization over translations yields the second one. To be more precise, the last estimate in~\eqref{eq:var-sigma-ep2} is based on the observation that for any 
given $a\in \R^2\backslash\{0\}$,
\begin{equation*}
\begin{aligned}
\min_{b\in \R^2} \int^1_{0} | ta+b|\dd{t} = \min_{\alpha,\,\beta\in \R} \int^1_{0} |(t+\alpha) a+\beta a^\perp|\dd{t} =  |a| \min_{\alpha\in
\R}\int^1_{0} |t+\alpha|\dd{t} =\frac{ |a|}{4}.
\end{aligned}
\end{equation*}
From \eqref{eq:var-sigma-ep1} and \eqref{eq:var-sigma-ep2}, since  $(u_\eps)_\eps\subset W^{1,1}(\Omega;\R^2)$ as a weakly$^\ast$ converging sequence is uniformly bounded in $BV(\Omega;\R^2)$, and
recalling that \(R^{i^+_\epsi}_\epsi  =  R^{i^+_\epsi-1}_\epsi\)
if \(\eps\Yrig\cap P_{\eps}^{i^+_\epsi}
=\emptyset \),  we
conclude that 
\begin{align}
 |D\Sigma_\eps| (-1,1)  \leq 4\sqrt{2}\int_\Omega \abs{\nabla
u_\eps}\dd{x}\leq C. \label{eq:var-sigma-ep}
\end{align}
 The convergence in~\eqref{eq:convSigma_ep} follows now from the weak$^\ast$ relative compactness of  bounded sequences in $BV(-1,1;\R^{2\times 2})$ (see Section~\ref{subs:BV}), together with the fact that strong $L^1$-convergence is length and angle preserving. The latter guarantees that the limit function $R\in BV(-1,1;\R^{2\times 2})$ takes values only in $SO(2)$. 

 Next, we show that there is \(\psi\in BV (-1,1 ;\RR^2) \) such
that 
\begin{equation}
\label{eq:claimuinB}
\begin{aligned}
u(x) = R(x_2)x + \psi(x_2)
\end{aligned}
\end{equation}
for a.e.\!  \(x\in\Omega\), which
implies that \(u\in \Bcal\) and concludes the proof. To this end, we define auxiliary functions $\sigma_\eps$, $b_\eps \in
L^\infty(\Omega;\R^2)$ for $\eps>0$ by setting
 \begin{align*}
 \sigma_\eps(x) = \sum_{i\in I_\eps} (R_\eps^ix) \mathbbm{1}_{P_\eps^i}(x)
 \qquad\text{and}\qquad b_\eps(x)=\sum_{i\in I_\eps} b_\eps^i
\mathbbm{1}_{P_\eps^i}(x)
 \end{align*}
 for $x\in \Omega$, where  \(R^{i^+_\epsi}_\epsi:= R^{i^+_\epsi-1}_\epsi\)
and   \(b^{i^+_\epsi}_\epsi:= b^{i^+_\epsi-1}_\epsi\)
 if \(\eps\Yrig\cap P_{\eps}^{i^+_\epsi}
=\emptyset \). Further, let $w_\eps:=\sigma_\eps+b_\eps$.

By  Poincar{\'e}'s inequality applied in the  $x_2$-direction,
we obtain
\begin{align*}
\int_\Omega \abs{u_\eps-w_\eps}\dd{x}  =\, &\sum_{i\in  I_\eps:\,  \eps\Ysoft\cap P_{\eps}^{i}
\not=\emptyset}
\int^1_{0}\int_{\max\{\eps( i-1),-1\}}^{\min\{\eps (i-1+\lambda),1\}}
\abs{u_\eps-w_\eps}\dd{x_2}\dd{x_1} \\  \leq\, &\eps\lambda \sum_{i\in
I_\eps} \int_{\eps\Ysoft \cap P_\eps^i} \abs{\partial_2 u_\eps-R_\eps^i
e_2}\dd{x}\leq\eps\lambda(\norm{u_\eps}_{W^{1,1}(\Omega;\R^2)}+\abs{\Omega})\leq
C\eps.
\end{align*}
Consequently,
  \begin{align}\label{conv_weps}
 w_\eps\to u\qquad \text{ in $L^1(\Omega;\R^2).$}
 \end{align}
Moreover,   for $x\in \Omega$, 
\begin{equation*}
\begin{aligned}
|\sigma_\epsi (x)  - R(x_2)x| \leq \Big|\sum_{i\in I_{\eps}} (R^i_\epsi - R(x_2))
\mathbbm{1}_{P_\eps^i}(x)\Big| | x| \leq \sqrt2 |\Sigma_\epsi(x_2) - R(x_2)|,
\end{aligned}
\end{equation*}
which, together with \eqref{eq:convSigma_ep}, proves that
\begin{align}\label{conv_sigmaeps}
 \sigma_\eps\to \sigma\qquad \text{ in $L^1(\Omega;\R^2),$}
 \end{align}
where \(\sigma(x):=R(x_2)x \in BV(\Omega;\RR^2) \). \color{black}

Finally, exploiting~\eqref{conv_weps} and  \eqref{conv_sigmaeps}, we conclude that there exists $b\in BV(\Omega;\R^2)$ such that $b_\eps\to b$ in $L^1(\Omega;\R^2)$.  In view of the one-dimensional character  of the stripes \(P^i_\epsi\), we infer that $\partial_1
b=0$. Eventually, identifying \(b\) with a function $\psi\in BV(-1,1;\R^2)$  yields  \eqref{eq:claimuinB}. 

\end{proof}

Next, we prove that the converse inclusion  of~\eqref{B0subsetB}   holds.  In the following,  let $I_{\rm rig}$ be the projection of $Y_{\rm rig}$ onto the second component; that is, $I_{\rm rig}$ corresponds to the $1$-periodic extension of the interval $[\lambda,1)$. Analogously, we write $I_{\rm soft}$ for the $1$-periodic extension of $[0, \lambda)$.

\begin{proposition}[Approximation of maps in $\Bcal$] \label{prop:approx1}
 Let  $\Omega=(0,1)\times (-1,1)$. Then, 
\begin{align}\label{B0supsetB} \Bcal_0 \supset \Bcal.
\end{align} 
Here, \(\Bcal_0\) and \(\Bcal\) are the sets from~\eqref{eq:def-B0} and \eqref{eq:def-B}, respectively. 
\end{proposition}

\begin{proof}
Let \(u\in\Bcal\), and let  $R\in BV(-1,1;SO(2))$ and $\psi\in BV(-1,1;\R^2)$ be such that 
\begin{align*}
u(x)=R(x_2)x+\psi(x_2)
\end{align*}
for a.e.\! \(x\in \Omega\).
Using Lemma~\ref{lem:approx_layers} below,  as well as the fact that strict convergence implies weak$^\ast$ convergence in $BV$,  we construct sequences $(R_\eps)_\eps\subset W^{1, \infty}(-1,1;SO(2))$ and $(\psi_\eps)_\eps\subset W^{1, \infty} (-1,1;\R^2) $ such that 
\begin{align}
R_\eps'=0
\quad &\text{ and } \quad \psi_\eps'=0\quad \text{\  \ on } \eps I_{\rm rig}\cap (-1,1),  \label{eq:restIrig}\\
R_\eps\weaklystar R \text{ in $BV(-1,1;\R^{2\times 2})$} \quad &\text{ and }\quad  \psi_\eps\weaklystar \psi\quad \text{ in $BV(-1,1;\R^2)$.}\label{eq:conv1}
\end{align}

Define $u_\eps(x) := R_\eps(x_2) x + \psi_\eps(x_2)$ for \(x\in\Omega\). Then, $u_{\eps}\in W^{1, \infty}(\Omega;\R^2)$ for every $\eps$, with 
\begin{align*}
\nabla u_\eps (x) = R_\eps(x_2) + R_\eps'(x_2) x\otimes
e_2 + \psi_\eps'(x_2)\otimes e_2 \end{align*} 
for a.e.\! \(x\in \Omega\).
In particular, \(\nabla u_\epsi = R_\epsi \in SO(2)\) a.e.\! in $\eps\Yrig\cap \Omega$ by \eqref{eq:restIrig}; hence, \(u_\epsi\in \Bcal_\epsi\). Moreover,  \(\sup_\epsi \Vert \nabla u_\epsi\Vert_{L^1(\Omega;\RR^{2\times 2})}<\infty\) and  \(u_\epsi
\to u\) in \(L^1(\Omega;\R^2)\) by \eqref{eq:conv1}, from which we conclude that   $u_\eps\weaklystar u$  in $BV(\Omega;\R^2)$. This completes the proof. 
\end{proof}

 The next lemma states  a one-dimensional approximation
result of $BV$-maps by Lipschitz functions that are constant 
on $\eps I_{\rm rig}$, which was an important ingredient in the previous proof.

\begin{lemma}[\boldmath{$1D$}-approximation by maps  constant on  \boldmath{$\eps I_{\rm rig}$}]\label{lem:approx_layers}
Let  $I=(a, b)\subset \R$ and   $w\in BV(I;\R^d)$.  Then,  
there exists a sequence $(w_\eps)_\eps\subset W^{1, \infty}(I;\R^d)$
with the following three properties:
\begin{itemize}
\item[$(i)$] $w_\eps\to w$ in $L^1(I;\R^d)$;\\[-0.2cm]
\item[$(ii)$] $\displaystyle \int_I |w_\eps'| \dd{t} \to |Dw|(I)$;\\[-0.2cm]
\item[$(iii)$] $w_\eps' = 0$ on $\eps I_{\rm rig}\cap  I$.
\end{itemize}
Moreover, if \(w\) takes values in \(SO(2)\) and \(w\in BV(I;SO(2))\), then  each \(w_\epsi\) may be taken in \(W^{1,\infty}(I;SO(2))\).
\end{lemma}

\begin{proof}
Let $w\in BV(I;  \R^d)$. By~\cite[Theorem~3.9, Remark~~3.22]{AFP00},  $w$ can be approximated by a sequence of smooth functions $(v_\delta)_\delta\subset C^\infty(\bar{I}; \R^d)$ in the sense of strict convergence in $BV$; that is, 
\begin{align}\label{conv12}
v_\delta\to w \text{ in $L^1(I;  \R^d)$}\quad \text{and}\quad \int_I |v_\delta'|\dd{t}\to |Dw|( I)
\end{align}
as $\delta\to 0$. To obtain property $(iii)$, we will reparametrize $v_\delta$ so that it is \textit{stopped} on the set $\eps I_{\rm rig}$ and \textit{accelerated} otherwise, and eventually apply a diagonalization argument.

We start by introducing  for every $\eps>0$  a Lipschitz function $\varphi_\eps: \RR\to \RR$ defined by
\begin{equation*}
\begin{aligned}
\ffi_\epsi(t):=\begin{cases}
\frac{1}{\lambda} (t-i\epsi) + i\epsi & \text{if } i\epsi \leq t \leq i\epsi + \lambda\epsi,\\
(i+1)\epsi & \text{if } i\epsi
+ \lambda\epsi\leq t < \epsi (i+1),
\end{cases}
\end{aligned}
\end{equation*}
for each \(i\in\ZZ\) and \(t\in \epsi[i, i+1)\). For all \(t\in\RR\), we have \(t\leq \ffi_\epsi(t) \leq t + \epsi(1-\lambda)\) and \(\ffi_\epsi' (t) = \psi (\frac{t}{\epsi})\), where \(\psi\) is the 1-periodic function such that \(\psi(t) =\frac{1}{\lambda}\) if \(0\leq t\leq \lambda\), and  \(\psi(t) =0\) if \(\lambda< t< 1\). By the Riemann--Lebesgue lemma on weak convergence of periodically oscillating
sequences, it follows that $\psi (\frac{\cdot}{\epsi})\weaklystar 1$  in $L^{\infty}(\RR)$.
Thus, $\varphi_{\eps}\weaklystar \ffi$  in $W^{1,\infty}_{loc}(
\R)$, where \(\ffi(t):=t\).
In particular, $\varphi_\eps$ converges uniformly to \(\ffi\)  on every compact set 
$K\subset \RR$. 

Next, we define  for $\eps>0$  a Lipschitz function $\tilde\varphi_\eps: \bar I\to \bar I$
by setting
\begin{equation*}
\begin{aligned}
\tilde\ffi_\epsi(t):=\begin{cases}
\ffi_\epsi(t) & \text{if } a \leq t \leq b_\epsi,
\\
b & b_\epsi\leq t \leq b, 
\end{cases}
\end{aligned}
\end{equation*}
where \(b_\epsi\in (a,b]\) is such that \(\ffi_\epsi (b_\epsi) =b\). Note that by definition of \(\ffi_\epsi\), there exists at least one such  \(b_\epsi\). 
We claim that \(b_\epsi \to b\)  as $\eps\to 0$.  In fact, extracting a subsequence if necessary, we have \(b_\epsi \to c\) for some \(c\in [a,b]\). Then,  
\begin{equation*}
\begin{aligned}
|b-c| = |\ffi_\epsi(b_\epsi) - \ffi(c)| \leq |\ffi_\epsi(b_\epsi) - \ffi_\epsi(c)| + |\ffi_\epsi(c) - \ffi(c)|\leq \tfrac{1}{\lambda}|b_\epsi - c|+ |\ffi_\epsi(c) - \ffi(c)|,
\end{aligned}
\end{equation*}
from which we infer that \(b=c\) by letting \(\epsi\to0\). Because the limit does not depend on the subsequence, the whole sequence \((b_\epsi)_\epsi\) converges to \(b\). Consequently, \(\tilde \ffi_\epsi(t) \to \ffi(t)=t\) for all \(t\in \bar I\), and since also 
\(\Vert \tilde \ffi_\epsi\Vert_{W^{1,\infty}(I)} = O(1)\) as \(\epsi\to0\),  we deduce that 
\begin{equation}
\label{eq:ontildeffiepsi}
\begin{aligned}
\tilde \ffi_\epsi \weaklystar \ffi \text{ in } W^{1,\infty}(I) \quad \text{and} \quad \Vert \tilde \ffi_\epsi - \ffi\Vert_{L^\infty(I)} \to 0.
\end{aligned}
\end{equation}

Finally, we set $w_{\eps, \delta} := v_\delta\circ \tilde\varphi_\eps\in W^{1, \infty}(I; \R^d)$, and observe that
\begin{align*}
\norm{w_{\eps,\delta} - w}_{L^1(I; \R^d)} \leq \norm{v_\delta\circ \tilde\varphi_\eps- v_\delta}_{L^1(I; \R^d)} + \norm{v_\delta - w}_{L^1(I; \R^d)} \quad \text{and}\quad 
\int_I |w_{\eps, \delta}'|\dd{t} = \int_I |v_\delta'\circ
\tilde\varphi_\eps|\, \tilde \varphi_{\eps}' \dd{t}. 
\end{align*}
Hence, by ~\eqref{conv12}, \eqref{eq:ontildeffiepsi},   the boundedness of each $v_\delta$ and $v_\delta'$, and  a weak-strong
convergence argument, it follows that 
\begin{align}
&\lim_{\delta\to 0}\lim_{\eps\to 0}\norm{w_{\eps, \delta} - w}_{L^1  (I;\R^d)}=0, \label{conv23}\\ & \lim_{\delta\to 0}\lim_{\eps\to 0}\int_I |w_{\eps, \delta}'|\dd{t} =\lim_{\delta\to 0} \int_I |v_\delta'\circ
\varphi|\, \varphi' \dd{t} =\lim_{\delta\to
0} \int_I |v_\delta'|\dd{t}=|Dw|( I). \label{conv24}
\end{align}

 In view of~\eqref{conv23}  and~\eqref{conv24},  we  apply  Attouch's diagonalization lemma \cite{Att84} to find a sequence  $(w_\eps)_\eps\subset W^{1,1}(I;\R^d)$  with $w_\eps:=w_{\eps, \delta(\eps)}$ satisfying $(i)$ and $(ii)$. We observe further that  each $w_\eps$ 
satisfies $(iii)$ by construction. 

To conclude, we address  the issue of constraint-preserving approximations for \(w\in BV(I;SO(2))\).  In this case, we argue as above, but replace the density argument leading to~\eqref{conv12} by its analogue for $BV$ functions with values on manifolds, see
\cite[Theorem~1.2]{GiM07}. This allows us to assume that \(v_\delta\in C^\infty(\bar{I}; SO(2)) \), and eventually yields  \(w_\epsi \in W^{1,\infty}(I;SO(2))\).  
\end{proof}

We are now in a position to prove Theorem~\ref{thm:aymrigintro}.

\begin{proof}[Proof of Theorem~\ref{thm:aymrigintro}]
In view of the discussion on  locally and globally one-dimensional  functions in Section~\ref{subs:geom}, it suffices
to prove the   statement  on rectangles with sides parallel
to the axes.  A simple modification
of the proofs of Propositions~\ref{prop:asymptotic_rigidity2}
and~\ref{prop:approx1} shows that these results hold for any
such rectangle. Then, Theorem~\ref{thm:aymrigintro}
follows  by  extension and exhaustion arguments in the spirit of~ \cite[Lemma~A.2]{ChK18}. 
\end{proof}

\begin{remark}[The higher dimensional setting]\label{rk:high-dim}
We point out that the results of Theorem~\ref{thm:aymrigintro} continue to hold for domains $\Omega\subset \R^N$, $N\in \mathbb{N}$, satisfying the flatness and cross-connectedness assumptions in \cite[Definitions 3.6 and 3.7]{ChK18}. We omit the proof here as it follows from that of Theorem~\ref{thm:aymrigintro} up to minor adaptations.  Notice in particular that~\cite[Lemma~A1]{ChK17} provides a higher-dimensional version of~\eqref{eq:var-sigma-ep2}.
\end{remark}

We conclude this section by  characterizing two  special subsets
of \(\Bcal\) (see \eqref{eq:def-B}), which will be useful in the following. Using  \eqref{eq:onDu1},  it can be checked that  
\begin{align}\label{BcapW11}
 \Bcal\cap W^{1,1}(\Omega;\R^2) = \big\{u \in W^{1,1}(\Omega;\R^2)\!:\,
 &\,u(x)
= R(x_2)x +\psi(x_2) \text{ for a.e.\! }x\in  \Omega, 
\\
 &\text{with }  R\in W^{1,1}(a_\Omega,b_\Omega;SO(2))  \text{ and } \psi\in W^{1,1}(a_\Omega,b_\Omega;\R^2)\big
\}\notag
\end{align}
and 
 \begin{align}\label{BcapSBV}
 \Bcal\cap SBV(\Omega;\R^2) = \big\{u\in SBV(\Omega;\R^2)\!:\,&\,u(x)
= R(x_2)x +\psi(x_2) \text{ for a.e.\! }x\in  \Omega,
\\
 &\text{with }  R\in SBV(a_\Omega,b_\Omega;SO(2)) \text{ and }\psi\in SBV(a_\Omega,b_\Omega;\R^2)\big\}. \notag
 \end{align}
 
 By definition,  and accounting  for the fact that \(R\) takes values in \(SO(2)\), the jump set of $u\in \Bcal\cap SBV(\Omega;\R^2)$
is related to the jump sets of  $R$ and $\psi$  via  
\begin{align*}
J_u =  [(c_\Omega,d_\Omega)\times
(J_R\cup J_\psi)]\cap\Omega,
\end{align*} 
cf.~\eqref{cOmega}.

\section{Asymptotic behavior of  admissible layered
deformations}\label{sec:ad}
In this section, we prove Theorem~\ref{thm:SBVinfty},  which characterizes the asymptotic behavior of deformations of bilayered materials that coincide with rigid body rotations on the stiffer layers, and are subject to a single slip constraint on the softer layers. The latter is described with the help of the set 
\begin{equation}\label{Me1}
\begin{aligned}
\Mcal_{e_1}&=\{F\in \R^{2\times 2}\!: \, \det F=1\text{ and $|Fe_1| = 1$}\}\\
&= \{F\in \R^{2\times 2}\!:\, F=R(\Ibb+\gamma e_1\otimes
e_2) \text{ with $R\in SO(2)$ and $\gamma\in \R$}\}. 
\end{aligned}
\end{equation}
As in the previous section, we consider $\Omega=(0,1)\times (-1,1)$ for simplicity.  The results for general $x_1$-connected domains follow as  in the proof of Theorem~\ref{thm:aymrigintro}.

Using the representations of \(\Mcal_{e_1}\)  in~\eqref{Me1}  and recalling the
sets $\mathcal{B}_{\eps}$ introduced in \eqref{eq:def-B-ep},
  the sets of admissible
layered deformations defined in \eqref{eq:def-Aep}
admit the  equivalent representations
\begin{align}
\Acal_\eps &  =  \Bcal_\eps \cap \{u\in W^{1,1}(\Omega;\R^2): \nabla
u\in \Mcal_{e_1}\text{  a.e. in } \Omega\} \notag\\
&=\{ u\in W^{1,1}(\Omega;\R^2): \nabla u = R(\Ibb+\gamma e_1\otimes e_2) \text{ with $R\in L^\infty(\Omega;SO(2))$ and}\notag\\ 
& \qquad \qquad \qquad \qquad \qquad \text{$\gamma\in L^1(\Omega)$ such that $\gamma=0$ in $\eps\Yrig\cap\Omega\}$}. \label{eq:def-Aepalt}
\end{align} 
In the sequel, according to the context, we will  always  adopt the most convenient representation.

In analogy with \(\Bcal_0\) defined in \eqref{eq:def-B0},
we introduce the set 
\begin{align}\label{eq:def-A0}
\Acal_0:= \{u\in BV(\Omega;\R^2): &\text{ there exists }(u_{\eps})_{\eps}\subset W^{1,1}(\Omega;\R^2)\text{  with $u_\eps\in \Acal_\eps$ for all $\eps$ }\\
&\text{ such that  } u_\eps\weaklystar u \text{ in
$BV(\Omega;\R^2)$} \}\notag
\end{align} 
of \emph{asymptotically
admissible deformations}.  We aim at  characterizing $\Acal_0$, or suitable subclasses thereof,  in terms
of the set \(\Acal\) introduced in~\eqref{eq:def-A}. Note that
\begin{align}\label{eq:def-Aalt}
\Acal & = \Bcal \cap \{u\in BV(\Omega;\R^2): \det \nabla u=1\text{   a.e.~in  $\Omega$}\},
\end{align}
where $\mathcal{B}$ is given by \eqref{eq:def-B}.
Moreover, recalling the notation for the distributional derivative of one-dimensional \(BV\)-functions discussed in Section~\ref{subs:BV},   we can equivalently express $\Acal$ as follows.

\begin{proposition}\label{prop:charA} Let   \(\Omega=(0,1) \times (-1,1)\). Then, \(\Acal\) from~\eqref{eq:def-A} admits these  two alternative representations: 
\begin{align}
 \Acal
=  \{ u\in BV(\Omega;\R^2)\!: &\,\,\nabla u(x)= R(x_2)(\Ibb+\gamma( x_2 )
e_1\otimes e_2) \text{ for a.e.\! }x\in  \Omega, \text{ 
with }\notag
\\ 
 &\,\,R\in BV(-1,1;SO(2)),\,  \gamma\in L^1( -1,1),\, \text{and } (D^su)e_1=0\}\label{eq:charAa}
\end{align}
and
\begin{align}\label{charA}
\Acal  =\{u\in BV(\Omega;\R^2)\!:\, & \, u(x) =  R(x_2)x +\psi(x_2) \text{ for a.e.\! }x\in  \Omega, \text{ 
with }  R\in BV(-1,1;SO(2)) \notag
\\
 &\,  \text{and } \psi\in BV(-1,1;\R^2)   \text{  such that 
 $\psi' \cdot Re_2=0$ and $R'=0$ a.e.\!~in $(-1,1)$}\}.
\end{align}
\end{proposition}

\begin{proof}
 Let \(\tilde A\)
and \(\hat \Acal\) denote the sets on the right-hand side of \eqref{eq:charAa} and \eqref{charA}, respectively. We will show
that
 \(\Acal\subset\tilde\Acal \cap \hat\Acal \),  \(\hat\Acal\subset\Acal\),
and  \(\tilde\Acal\subset\hat\Acal\),  from which
\eqref{eq:charAa} and \eqref{charA} follow.
  
We start by
proving that \(\Acal\subset\tilde\Acal \cap \hat\Acal \). 
Fix \(u\in\Acal\). Due to \eqref{eq:onDu1}, we have
\((D^su)e_1=0\) and
\begin{align}\label{eq:gradofu}
\nabla u = R+(R'x + \psi')\otimes e_2=R(\Ibb+R^T(R'x + \psi')\otimes
e_2).
\end{align}
We first observe that the condition \(\det\nabla u = 1\) becomes \(1+R^T(R'x + \psi')\cdot e_2 =1\) or, equivalently, \((R'x
+ \psi')\cdot Re_2 =0\). This condition, together with the independence of $R$, $R',$ and $\psi'$
on $x_1$, yields 
\begin{align}\label{eq:bydetcond}
R'e_1\cdot Re_2=0 \qquad \text{and}\qquad (x_2R'e_2+\psi')\cdot
Re_2=0.
\end{align}
Let \(\theta\in BV(-1,1;[-\pi,\pi])\) be as in \eqref{eq:Rtrigon}.
Then, the first condition in \eqref{eq:bydetcond} 
gives \(\theta'  = 0\); consequently, also \(R'  =  0\).
Thus, the second equation in \eqref{eq:bydetcond} becomes \(\psi' \cdot Re_2=0\), which shows that \(u\in \hat \Acal\).
Moreover,  \(\psi'
\cdot Re_2=0\) is  equivalent  to \(R^T\psi' \cdot e_2=0\); hence,
 \(u\in \tilde \Acal\) with \(\gamma := Re_1\cdot\psi'\). Thus,
 \(\Acal\subset\tilde\Acal \cap \hat\Acal \).

Next, we observe that if \(u\in \hat\Acal\), then, using \eqref{eq:gradofu},
we have 
$$\det\nabla u = 1+R^T(R'x
+ \psi')\cdot e_2 =1+R^T\psi' \cdot e_2= 1+\psi' \cdot R e_2
=1.$$ 
Hence, \(u\in\Acal\), which shows that \(\hat\Acal\subset\Acal\).

Finally, we prove that  \(\tilde\Acal\subset\hat\Acal\). Let \(u\in\tilde\Acal\).
Then, $(Du)e_1=(\nabla u) e_1  \Lcal^2\lfloor\Omega + (D^su)e_1
=Re_1  \Lcal^2\lfloor\Omega$. By this identity and 
the Du  Bois-Reymond
lemma  (see   \cite{idczak}, for instance), we can find  $\phi\in BV(-1,1;\mathbb{R}^2)$  such that
$$u(x)=R(x_2)x_1e_1+\phi(x_2).$$

In particular, \(\nabla u(x) = R(x_2)e_1\otimes e_1 +\ (R'(x_2)x_1 e_1
+ \phi'(x_2))\otimes e_2\). Consequently, using the expression
for \(\nabla u\) given by the definition of \(\tilde \Acal\),
together with the independence
of $R$, $R'$, \(\gamma\), and $\phi'$
on $x_1$,
we conclude that
\begin{equation*}
\begin{aligned}
R' = 0 \enspace \text{ and } \enspace \phi'= R
e_2 +  \gamma R e_1.
\end{aligned}
\end{equation*}

Finally, set \(\psi(x_2):= \phi(x_2) -R(x_2)x_2e_2  \) for \(x_2\in
(-1,1)\). Then,  we have \(\psi\in BV(-1,1;\RR^2) \), which  satisfies \(\psi'\cdot Re_2= \gamma R e_1  \cdot Re_2=0\), because \(R \in SO(2)\) in $(-1,1)$, and also  \(u(x) =  R(x_2)x +\psi(x_2)\).   Thus,
\(u\in \hat \Acal\), which  implies   \(\tilde\Acal\subset\hat\Acal\). 
\end{proof}

 The following lemma on weak continuity of Jacobian determinants for gradients in $W^{1,1}(\Omega;\R^2)$ with suitable additional properties will be instrumental in the proof of the
inclusion \(\Acal_0 \subset \Acal\).

\begin{lemma}[Weak continuity properties of Jacobian determinants]\label{lem:det}
 Let $\Omega\subset \R^2$ be a bounded Lipschitz domain, and let 
$(u_\eps)_{\eps} \subset W^{1,1}(\Omega;\R^2)$ be a uniformly bounded sequence
satisfying $\det \nabla u_\eps=1$ a.e.~in $\Omega$ for all $\eps$ and 
\begin{align}\label{Linftybound}
\norm{\partial_1 u_\eps}_{L^\infty(\Omega;\R^{2})}\leq C,
\end{align}
where \(C\) is a positive constant independent of $\eps$.
If $u_\eps\to u$ in $L^1(\Omega;\R^2)$ for some $u\in BV(\Omega;\R^2)$, then
$\det \nabla u=1$ a.e.\!~in $\Omega$.
\end{lemma}

\begin{proof}
The claim in Lemma~\ref{lem:det} would be an immediate consequence of 
\cite[Theorem~2]{FLM05} if in place of \eqref{Linftybound}, we required 
\begin{equation}\label{(1)inFLM}
\begin{aligned}
(\operatorname{adj} \nabla u_\eps)_\eps \subset L^2(\Omega;\RR^{2\times
2}), 
\end{aligned}
\end{equation}
which, because of the structure of the adjoint matrix in this two-dimensional setting, 
is equivalent to \(\nabla u_\eps\in L^2(\Omega;\RR^{2\times
2})\) for all \(\eps\).
 Even though we are not assuming this here, it is still possible to validate the arguments of \cite[ Proof of Theorem~2]{FLM05}  in  our context,  as we detail next.

Since \(|\operatorname{adj} \nabla u_\eps| = | \nabla u_\eps|\),
it can be checked that in order to mimic the proof  of 
\cite[Theorem~2]{FLM05} with \(N=2\), we are only left to prove the following: If \((\varphi_j)_{j\in\N}\) is a sequence of standard mollifiers and \(\Omega'\) is an arbitrary open set compactly contained in \(\Omega\), then
\((\det \nabla u_{\eps,j})_{j\in\N} \) converges to \(\det \nabla u_\eps\)  in \(L^1(\Omega')\)  as $j\to \infty$ for all \(\eps\), where \( u_{\eps,j} := \varphi_j
* u_\eps\). 

 In Step~4 of the proof of 
\cite[Theorem~2]{FLM05}, this convergence is a consequence of the Vitali--Lebesgue lemma
using  \eqref{(1)inFLM},  the bound \(|\det A| \leq |\operatorname{adj} A|^2\) for all \(A\in \R^{2\times2}\) (see 
\cite[(7)]{FLM05}), and well-known properties of mollifiers. 

Here, similar arguments  can be invoked, but instead of the estimate \(|\det A| \leq |\operatorname{adj}
A|^2\)  for $A\in \R^{2\times 2}$,   we use the fact that  \eqref{Linftybound}  
yields
\begin{equation*}
\begin{aligned}
|\det\nabla u_{\eps,j}|= | (\partial_1 u_{\eps,j})^\perp \cdot \partial_2 u_{\eps,j} | \leq C |\partial_2 u_{\eps,j}| \leq C|\nabla u_{\eps, j}|
\end{aligned}
\end{equation*}
 a.e.~in $\Omega$. 
Hence, since \( u_{\eps,j} \to u_\eps\) in \(W^{1,1}(\Omega';\R^{2})\) and pointwise
a.e.\!~in \(\Omega\)  as $j\to \infty$,  we conclude that \((\det\nabla u_{\eps,j})_{j\in\N} \) converges to \(\det \nabla u_\eps\)  in \(L^1(\Omega')\)   as $j\to \infty$ 
for all \(\eps \) by the Vitali--Lebesgue
lemma.
\end{proof}

We obtain from the following proposition that weak$^\ast$ limits of sequences
in $\Acal_\eps$ belong to $\Acal$.

\begin{proposition}[Asymptotic behavior of sequences 
in $\Acal_{\eps}$]\label{con:compactness}
 Let  $\Omega=(0,1)\times (-1,1)$. Then, 
\begin{align}\label{A0subsetA} \Acal_0 \subset \Acal,
\end{align} 
where \(\Acal_0\) and \(\Acal\) are the sets introduced in \eqref{eq:def-A0} and \eqref{eq:def-A}, respectively. 
\end{proposition}

\begin{proof}
The statement follows from the inclusion \(\Acal_\epsi \subset \Bcal_\epsi\)
(see \eqref{eq:def-Aepalt}) and  the identity \eqref{eq:def-Aalt} in conjunction with 
Proposition~\ref{prop:asymptotic_rigidity2} and Lemma~\ref{lem:det},
observing that the condition $\nabla u_\eps\in \Mcal_{e_1}$ a.e.~in $\Omega$ guarantees $|\partial_1 u_\eps|=|\nabla u_\eps e_1|=1$ a.e.~in $\Omega$, and hence $\norm{\partial_1
u_\eps}_{L^\infty(\Omega;\R^2)}=1$ for any $\eps$.
\end{proof}

The question whether  the set $\Acal$ can be further identified as limiting
set for sequences in $\Acal_{\eps}$, namely, whether the equality \(\Acal_0
= \Acal\)  
  is true,  cannot be answered at this point. However, as stated in Theorem~\ref{thm:SBVinfty}, the inclusions \(\Acal_0\supset \Acal\cap  SBV_\infty(\Omega;\R^2)\)  and  \(\Acal_0\supset
\Acal^\parallel\) hold. Before proving these inclusions, we discuss  a further characterization of some 
special subsets of $\Acal$.
\begin{remark}[Structure of subsets of $\Acal$]\label{rk:subsets-A}
 Similarly to~\eqref{BcapW11}  and \eqref{BcapSBV}, using fine properties of one-dimensional \(BV\) functions, the sets \(\Acal \cap W^{1,1}(\Omega;\R^2)\), \(\Acal \cap SBV(\Omega;\R^2)\), and \(\Acal\cap SBV_\infty(\Omega;\R^2)\) can be characterized as follows.
\vskip1mm

{\parindent=0mm
(a)}  In view of  \eqref{eq:onDu1} and \eqref{charA},  one observes that 
\begin{align*}
\Acal \cap W^{1,1}(\Omega;\R^2) 
&= \{u\in W^{1,1}(\Omega;\R^2) \!:\, 
u(x)= Rx+\theta(x_2)Re_1+c 
\text{ for a.e.\! $x\in\Omega$,}\\
&\hspace{3.4cm} \text{with $R\in SO(2)$, $\theta\in W^{1,1}(-1,1)$, $c\in \mathbb{R}^2$}\}\\ 
&= \{u\in W^{1,1}(\Omega;\R^2)\!:\,  \nabla u(x)= R(\Ibb+ \gamma(x_2) e_1\otimes e_2) \text{ for
a.e.\! $x\in\Omega$, }\\
&\hspace{3.4cm}\text{with $R\in SO(2)$, $\gamma\in L^1(-1,1)$}\}.
\end{align*} 

Additionally, as a consequence of the construction of the recovery sequence in the $\Gamma$-convergence homogenisation result \cite[Theorem~1.1]{ChK17}, we also  know that 
\begin{align*}
\Acal\cap W^{1,1}(\Omega;\R^2) = \{u\in W^{1,1}(\Omega;\R^2)\!:\, &\text{ there exists }(u_{\eps})_{\eps}\subset W^{1,1}(\Omega;\R^2)\text{ with $u_\eps\in \Acal_\eps$ for all $\eps$ }\\
 &\text{ such that } u_\eps\weakly u \text{ in $W^{1,1}(\Omega;\R^2)$}\}. 
\end{align*}

 \vskip1mm

{\parindent=0mm
(b)} Using   \eqref{eq:onDu1} and \eqref{charA} once more, we have 
\begin{align*}
\Acal \cap SBV(\Omega;\R^2) 
= \{u\in SBV(\Omega;\R^2)\!:\, &\, u(x)= R(x_2)x+\psi(x_2) \text{ for a.e.\! $x\in\Omega$,} \\ & \text{with }\,R\in SBV(-1,1;SO(2)) \text{ and } \psi\in SBV(-1,1;\mathbb{R}^2)\\
&\,\text{such that $R'=0$ and $\psi' \cdot Re_2=0$ \!  a.e. in  \((-1,1)\)}\}.
\end{align*}
Note that both $J_R$
and $J_{\psi}$ are given by an at most countable union of points in $(-1, 1)$,  which implies that $J_u$ consists of at most countably many  segments parallel to $e_1$. It is not possible to conclude that the  functions $R$  
are  piecewise constant according to \cite[Definition 4.21]{AFP00}, as we have, a priori, no control  on $\mathcal{H}^0(J_R)$  
(cf.~\cite[Example 4.24]{AFP00}).

\vskip2mm
{\parindent=0mm%
(c)}   With  (b) and  \cite[Theorem 4.23]{AFP00}, and recalling
\eqref{eq:def-PC},  it follows that 
 \begin{align*}
\Acal\cap SBV_\infty(\Omega;\R^2) & = \{u\in SBV_{\infty}(\Omega;\R^2)\!:\, u(x)= R(x_2)x+\psi(x_2) \text{ for a.e.\! $x\in\Omega$,}\\
&\hskip35mm\text{ with  $R \in PC(-1,1;SO(2))$  and  $\psi\in SBV_{\infty}(-1,1;\mathbb{R}^2)$}\\
&\hskip35mm\text{ such that
 $\psi' \cdot Re_2=0$ a.e.\! in \((-1,1)\)} \}.
\end{align*}
Here, both $J_R$
and $J_{\psi}$ are finite sets of points in $(-1, 1)$, and  $J_u$ is given by a finite union of segments
parallel to $e_1$.
 Alternatively, one can express $\Acal\cap SBV_\infty(\Omega;\R^2)$ with the help of a Caccioppoli partition of $\Omega$ into finitely many horizontal strips; precisely, 
 \begin{align*}
\Acal\cap SBV_\infty(\Omega;\R^2) &  =\{ u\in SBV_\infty(\Omega;\R^2)\!:\, \nabla u|_{E_i}=R_i(\Ibb + \gamma_i e_1\otimes e_2), \text{ with $\{E_i\}_{i=1}^{n}$ a partition of $\Omega$ }\\&\hskip37mm\text{such that $E_i=(\R\times I_i)\cap \Omega$ with \(I_i\subset (-1,1)\) for $i=1, \ldots, n$, }\\
&\hskip37mm\text{$R_i\in SO(2)$ and $\gamma_i\in L^1(E_i)$ with $\partial_1\gamma_i=0$ for $i=1, \ldots, n$}  \}. 
\end{align*}

\end{remark}

In the following lemma, we construct an admissible   piecewise affine  approximation for  basic limit deformations in $ \Acal\cap  SBV_\infty(\Omega;\allowbreak\R^2)$  with a non-trivial jump along the horizontal line at $x_2=0$. Based on this construction, we  will then  establish the inclusion  \(\Acal_0\supset \Acal\cap  SBV_\infty(\Omega;\R^2)\)  in  Proposition~\ref{prop:SBVinfty} below.

\begin{lemma}[Approximation of maps in \boldmath{$\Acal\cap SBV_\infty$} with  a single  jump]\label{lem:existence_admissible1}
Let $\Omega=(0,1)\times (-1,1)$, and let $u\in \Acal\cap SBV_\infty(\Omega;\R^2)$ be such that $u(x)=R(x_2)x+\psi(x_2)$ for a.e.\! $x\in \Omega$, where 
\begin{equation*}
\begin{aligned}
R(t):=\begin{cases}
R^+ & \text{if } t\in[0,1)\\
R^- & \text{if }  t\in (-1,0)
\end{cases} 
\quad \text{and} \quad 
\psi(t):=\begin{cases}
\psi^+ & \text{if } t\in [0,1)\\
\psi^- & \text{if }  t\in(-1, 0)
\end{cases} \qquad \text{for } t\in (-1,1),
\end{aligned}
\end{equation*}
with  some \(R^\pm\in SO(2)\) and \(\psi^\pm \in \RR^2\).
Then, there exists a sequence $(u_\eps)_\eps\subset W^{1,1}(\Omega;\RR^2)  $  with \(\int_\Omega u_\epsi\dd{x} = \int_\Omega u\dd{x}\) and $u_\epsi\in \Acal_\eps$ for all \(\epsi,\) and such that  $u_\eps\weaklystar
u$ in $BV(\Omega;\R^2)$.
\end{lemma}

\begin{proof}
We start by observing that for $u$ as in the statement of the lemma, there holds
\begin{equation}\label{eq:expDu}
\begin{aligned}
Du = R \Lcal^2\lfloor \Omega + [(R^+-R^-)e_1x_1
+(\psi^+-\psi^-)]\otimes e_2  \Hcal^1\lfloor
\bigl( (0,1) \times \{0\}\bigr).
\end{aligned}
\end{equation} 
Let \(S\in SO(2)\) be such that \textit{(i)} \(S\not= R^\pm\); \textit{(ii)} \(Se_1 \) and \(R^+e_1\) are linearly independent; \textit{(iii)} ~\(\theta^\pm\in (-\pi, \pi)\setminus\{0\}\) is
 the  rotation angle of \(S^T R^\pm\),  cf.~\eqref{eq:Rtrigon}.   Due to  \textit{(ii)}, there exist \(\alpha\), \(\beta\in \RR\) such that
\begin{equation}
\begin{aligned}\label{eq:p+-p-}
\psi^+ - \psi^- = \alpha R^+e_1 + \beta Se_1.
\end{aligned}
\end{equation}

For each \(\eps>0\), set
\begin{equation}\label{eq::defhl+l-epsi}
\begin{aligned}
\gamma^+_\epsi:= \frac{4\alpha}{\epsi\lambda},\enspace
\gamma^-_\epsi:= \frac{4\beta}{\epsi\lambda},\enspace
\mu^\pm_\epsi:=
\pm\frac{4}{\epsi\lambda} + \tan\Big( \frac{\theta^\pm}{2}\Big),
 \enspace \tilde \mu^\pm_\epsi:=
\pm\frac{4}{\epsi\lambda} - \tan\Big( \frac{\theta^\pm}{2}\Big),
\end{aligned}
\end{equation}
and let  \(V_\epsi \in L^1(\Omega;\RR^{2\times 2})\) 
be the function defined by
\begin{equation}
\label{eq::grad-u-eps}
\begin{aligned}
 V_\epsi(x)  = \begin{cases}
R^+ &\text{if } x\in  (0,1) \times (\epsi \lambda, 1),\\
R^+(\Ibb +\gamma^+_\epsi e_1 \otimes e_2) &\text{if }
x\in  (0,1) \times (\frac{3\epsi\lambda}4, \epsi\lambda ),\\
R^+(\Ibb +  \mu^+_\epsi e_1 \otimes e_2) & \text{if
} x_1\in  (0,1) \text{ and } x_2\in (-\frac{\epsi\lambda}4
 x_1 +\frac{3\epsi\lambda}4 ,\frac{3\epsi\lambda}4 ),\\ 
S(\Ibb +  \tilde\mu^+_\epsi e_1 \otimes e_2) & \text{if
} x_1\in  (0,1) \text{ and } x_2\in (\frac{\epsi\lambda}2,-\frac{\epsi\lambda}4  x_1 +\frac{3\epsi\lambda}4 ),\\  
S(\Ibb +  \gamma^-_\epsi e_1 \otimes e_2) &\text{if
} x\in  (0,1) \times (\frac{\epsi\lambda}4 ,\frac{\epsi\lambda}2  ),\\
S(\Ibb +  \tilde\mu^-_\epsi e_1 \otimes e_2) &  \text{if
} x_1\in  (0,1) \text{ and } x_2\in (\frac{\epsi\lambda}4  x_1,\frac{\epsi\lambda}4), \\
R^-(\Ibb +  \mu^-_\epsi e_1 \otimes e_2) &  \text{if
} x\in  (0,1) \text{ and } x_2\in (0,\frac{\epsi\lambda}4  x_1),\\ 
R^- &\text{if } x\in  (0,1) \times (-1,0),
\end{cases}
\end{aligned}
\end{equation}
see Figure~\ref{figgen}. 
%
%
\begin{center}
\begin{figure}[h]
\begin{tikzpicture}[scale=2.19]
  
\draw[dashed, black] 
(0,0) -- (4,0) -- (4, -2) -- (0, -2)
 -- (0,0);

 \draw[Korange!20,fill=Korange!40] 
(0,0) -- (0,.35) -- (4,.35) -- (4,0)
 -- (0,0);
 
  \draw[Korange!20,fill=Korange!40] 
(0,-2) -- (0,-2.35) -- (4,-2.35) -- (4,-2)
 -- (0,-2);


\draw[dashed,black] 
(0,0.03) -- (0,.35) ;

\draw[dashed,black] 
(4,0.03) -- (4,.35) ;

\draw[dashed,black] 
(0,-2.03) -- (0,-2.35) ;

\draw[dashed,black] 
(4,-2.03) -- (4,-2.35) ;


\draw[dashed,black] 
(0,-.5) -- (4,-.5) ;

\draw[dashed,black] 
(0,-.5) -- (4,-1) ;

\draw[dashed,black] 
(0,-1) -- (4,-1) ;

\draw[dashed,black] 
(0,-1.5) -- (4,-1.5) ;

\draw[dashed,black] 
(0,-2) -- (4,-1.5) ;

\draw(-.2,-2) node     
        {\begin{normalsize} $\textcolor{black}{0}$\end{normalsize}
        };

\draw(-.2,-1.5) node     
        {\begin{normalsize} $\textcolor{black}{\frac{\varepsilon\lambda}{4}}$\end{normalsize}
        };
        
\draw(-.2,-1) node     
        {\begin{normalsize} $\textcolor{black}{\frac{\varepsilon\lambda}{2}}$\end{normalsize}
        };        

\draw(-.2,-.5) node     
        {\begin{normalsize} $\textcolor{black}{\frac{3\varepsilon\lambda}{4}}$\end{normalsize}
        };

\draw(-.2,0) node     
        {\begin{normalsize} $\textcolor{black}{\varepsilon\lambda}$\end{normalsize}
        };

\draw(-.03,-2.5) node     
        {\begin{normalsize} $\textcolor{black}{0}$\end{normalsize}
        };
        
\draw(3.95,-2.5) node     
        {\begin{normalsize} $\textcolor{black}{1}$\end{normalsize}
        };  
        

\draw(3.26,-1.8) node     
        {\begin{normalsize} $\textcolor{black}{R^-(\mathbb{I}
        + \mu^-_\epsi e_1 \otimes e_2)}$\end{normalsize}
        };  

\draw(.64,-1.7) node     
        {\begin{normalsize} $\textcolor{black}{S(\mathbb{I}
        + \tilde\mu_\epsi^- e_1 \otimes e_2)}$\end{normalsize}
        };  
        
\draw(2,-1.25) node     
        {\begin{normalsize} $\textcolor{black}{S(\mathbb{I}
        +  \gamma_\epsi^-  e_1 \otimes
e_2)}$\end{normalsize}
        };                
 
\draw(.64,-.8) node     
        {\begin{normalsize} $\textcolor{black}{S(\mathbb{I}
        + \tilde\mu_\epsi^+ e_1 \otimes
e_2)}$\end{normalsize}
        };                     

\draw(3.26,-.7) node     
        {\begin{normalsize} $\textcolor{black}{R^+(\mathbb{I}
        + \mu^+_\epsi e_1 \otimes e_2)}$\end{normalsize}
        };

\draw(2,-.25) node     
        {\begin{normalsize} $\textcolor{black}{R^+(\mathbb{I}
        + \gamma^+_\epsi e_1 \otimes
e_2)}$\end{normalsize}
        };       

 
\draw(2,.18) node     
        {\begin{normalsize} $\textcolor{black}{R^+}$\end{normalsize}
        };  
 
\draw(2,-2.17) node     
        {\begin{normalsize} $\textcolor{black}{R^-}$\end{normalsize}
        };   
%

\draw(4.2,-1) node     
        {\begin{normalsize} $\textcolor{black}{
        }$\end{normalsize}
        };
\end{tikzpicture}
\caption{Construction of  $V_{\eps}$.}
\label{figgen}
\end{figure}
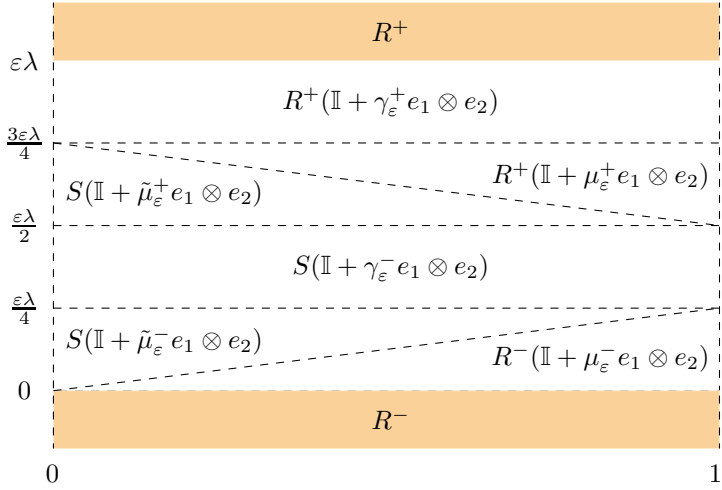
\end{center}

By construction, each function $V_\eps$ takes values only in $\Mcal_{e_1}$, and its piecewise definition is chosen such that neighboring matrices in Figure~\ref{figgen} are rank-one-connected along their separating lines 
according to \cite[Lemma~3.1]{ChK17}. Hence, there exists  a Lipschitz function \(u_\epsi \in W^{1,\infty}(\Omega;\RR^2)\) 
such that \(\nabla u_\epsi = V_\eps \).   By adding a suitable
constant, we may assume that \(\int_\Omega u_\epsi\dd{x} = \int_\Omega u\dd{x}\).  In view of  the Poincar\'e--Wirtinger inequality  and~\eqref{eq::grad-u-eps},  $(u_\eps)_\eps$ is a uniformly
bounded sequence
in $W^{1,1}(\Omega;\R^2)$  satisfying  $u_\epsi\in \Acal_\eps$ for all \(\epsi\) (cf.~\eqref{eq:def-Aepalt}).

  To prove that  
  $u_\eps\weaklystar
u$ in $BV(\Omega;\R^2)$, it suffices to show that
\begin{equation}
\label{eq::rc_weakconv}
\begin{aligned}
Du_\eps\weaklystar Du \text{  in } \Mcal(\Omega;\R^{2\times 2}),
\end{aligned}
\end{equation}
or, equivalently, in view of \eqref{eq:expDu}, that for every \(\varphi \in C_0(\Omega;\RR^2)\),
\begin{align}
\lim_{\epsi \to 0} \int_\Omega \nabla u_\epsi(x)  \varphi(x)\dd
x & =  \int_\Omega R(x_2)\varphi(x)
\dd x+ \int^1_{ 0}  [(R^+-R^-)e_1x_1
+(\psi^+-\psi^-)]\otimes e_2 \varphi(x_1,0) \dd x_1.\label{eq::rc_weakconv1}
\end{align}

Clearly,  
\begin{align}
\lim_{\epsi\to0} \int_{(0,1) \times [(-1,0) \cup (\epsi \lambda,
1) 
]} \nabla u_\epsi(x) \varphi(x)\dd
x &= \lim_{\epsi\to0} \int_{(0,1) \times [(-1,0) \cup (\epsi \lambda,
1)]} R(x_2)\varphi(x)\dd
x \notag\\
&= \int_\Omega R(x_2)\varphi(x)
\dd x.\label{eq::rc_weakconv2}
\end{align}

Moreover, using  \eqref{eq::defhl+l-epsi}, a change of variables,
and  Lebegue's dominated  convergence theorem  together with the
continuity and boundedness
 of \(\varphi\),  we have 
\begin{align}
&\lim_{\epsi\to0} \int_{(0,1) \times (0,\frac{\epsi\lambda}4  x_1)} \nabla u_\epsi(x) \varphi(x)\dd x \notag\\
&\quad = \lim_{\epsi\to0} \int_0^1\!\! \int_0^{\frac{\epsi\lambda}4
 x_1} R^-\big(\mathbb{I}
        + \tan\big( \tfrac{\theta^-}{2}\big)e_1 \otimes e_2- \tfrac{4}{\epsi\lambda}e_1 \otimes e_2\big) \varphi(x)\dd x_2\!\dd x_1
 \notag\\
&\quad = \lim_{\epsi\to0} \int_0^1\!\! \int_0^{
 x_1} R^-\big(\tfrac{\epsi\lambda}4\mathbb{I}
        + \tfrac{\epsi\lambda}4\tan\big( \tfrac{\theta^-}{2}\big)e_1 \otimes e_2 -e_1
\otimes e_2 \big)
\varphi(x_1,\tfrac{\epsi\lambda}4z)\dd z\!\dd x_1\notag\\
&\quad = - \int_0^1\!\! \int_0^{
 x_1}R^-e_1
\otimes e_2
\varphi(x_1,0)\dd z\!\dd x_1 = - \int_0^1  x_1 R^-e_1 
\otimes e_2
\varphi(x_1,0)\dd x_1.
\label{eq:limve1}
\end{align}
Similarly, 
\begin{align}
&\lim_{\epsi\to0} \int_{(0,1) \times (\frac{\epsi\lambda}4  x_1,\frac{\epsi\lambda}4)} \nabla u_\epsi(x) \varphi(x)\dd x \notag\\
&\quad = \lim_{\epsi\to0} \int_0^1\!\! \int_{x_1}^{
 1} S\big(\tfrac{\epsi\lambda}4\mathbb{I}
        - \tfrac{\epsi\lambda}4\tan\big( \tfrac{\theta^-}{2}\big)e_1
\otimes e_2 -e_1
\otimes e_2 \big)
\varphi(x_1,\tfrac{\epsi\lambda}4z)\dd z\!\dd x_1\notag\\
&\quad =  \int_0^1  (x_1-1)S e_1
\otimes e_2
\varphi(x_1,0)\dd x_1,
\label{eq:limve2}\\
&\lim_{\epsi\to0} \int_{(0,1) \times (\frac{\epsi\lambda}4 ,\frac{\epsi\lambda}2
 )}
\nabla u_\epsi(x) \varphi(x)\dd x \notag\\
&\quad = \lim_{\epsi\to0} \int_0^1\!\! \int_{1}^{
 2} S\big(\tfrac{\epsi\lambda}4\mathbb{I}
        +\beta e_1
\otimes e_2 \big)
\varphi(x_1,\tfrac{\epsi\lambda}4z)\dd z\!\dd x_1 =  \int_0^1 \beta Se_1
\otimes e_2
\varphi(x_1,0)\dd x_1,
\label{eq:limve3}
\\
&\lim_{\epsi\to0} \int_{(0,1) \times (\frac{\epsi\lambda}2,-\frac{\epsi\lambda}4
 x_1 +\frac{3\epsi\lambda}4 )}
\nabla u_\epsi(x) \varphi(x)\dd x \notag\\
&\quad = \lim_{\epsi\to0} \int_0^1\!\! \int_{2}^{
 3-x_1} S\big(\tfrac{\epsi\lambda}4\mathbb{I}
        - \tfrac{\epsi\lambda}4\tan\big( \tfrac{\theta^+}{2}\big)e_1
\otimes e_2 +e_1
\otimes e_2 \big)
\varphi(x_1,\tfrac{\epsi\lambda}4z)\dd z\!\dd x_1\notag\\
&\quad =  \int_0^1  (1-x_1)Se_1 
\otimes e_2
\varphi(x_1,0)\dd x_1,
\label{eq:limve4}
\\
&\lim_{\epsi\to0} \int_{(0,1) \times (-\frac{\epsi\lambda}4
 x_1 +\frac{3\epsi\lambda}4 ,\frac{3\epsi\lambda}4 )}
\nabla u_\epsi(x) \varphi(x)\dd x \notag\\
&\quad = \lim_{\epsi\to0} \int_0^1\!\! \int_{3-x_1}^{
 3} R^+\big(\tfrac{\epsi\lambda}4\mathbb{I}
        + \tfrac{\epsi\lambda}4\tan\big( \tfrac{\theta^+}{2}\big)e_1
\otimes e_2 +e_1
\otimes e_2 \big)
\varphi(x_1,\tfrac{\epsi\lambda}4z)\dd z\!\dd x_1\notag\\
&\quad =  \int_0^1  x_1 R^+e_1 
\otimes e_2
\varphi(x_1,0)\dd x_1,
\label{eq:limve5}
\end{align}
and 
\begin{align}
&\lim_{\epsi\to0} \int_{(0,1) \times (\frac{3\epsi\lambda}4, \epsi\lambda
 )}
\nabla u_\epsi(x) \varphi(x)\dd x \notag\\
&\quad = \lim_{\epsi\to0} \int_0^1\!\! \int_{3}^{
 4} R^+\big(\tfrac{\epsi\lambda}4\mathbb{I}
        +\alpha e_1
\otimes e_2 \big)
\varphi(x_1,\tfrac{\epsi\lambda}4z)\dd z\!\dd x_1 =  \int_0^1
\alpha R^+ e_1
\otimes e_2
\varphi(x_1,0)\dd x_1.
\label{eq:limve6}
\end{align}
 Combining \eqref{eq::rc_weakconv2}--\eqref{eq:limve6} and \eqref{eq:p+-p-}, we finally  obtain \eqref{eq::rc_weakconv1}. 
\end{proof}

\begin{remark}[On the construction in Lemma~\ref{lem:existence_admissible1}]\label{rmk:alternativeconst}
Notice that the main idea of the construction in the proof of Lemma~\ref{lem:existence_admissible1} for dealing with
jumps is to use piecewise affine functions  that are as simple as possible  to accommodate them.  Since triple junctions where two of the three angles add up to $\pi$  are not compatible (compare with~\cite[Lemma~3.1]{ChK17}),  we work with
inclined  interfaces  that stretch over the full width of $\Omega$.

Let  $u\in \Acal\cap SBV_\infty(\Omega;\R^2)$ be as in Lemma~\ref{lem:existence_admissible1},  and assume that either  \(R^+\not= \pm R^-\) or  \(R^+= R^-\). In these cases,  we can simplify  the
construction of  \((u_\epsi)_\epsi\) in the previous proof.
We focus here on stating the counterparts of Figure~\ref{figgen} and \eqref{eq::defhl+l-epsi}, and omit the  detailed calculations, which are very similar to ~\eqref{eq::rc_weakconv2}--\eqref{eq:limve6}.  Note further that these constructions are not just simpler, but also energetically more favorable, 
see Remark~\ref{rmk:optimalitylb} below for more details. 

\begin{itemize}
\item[(i)] 
If  \(R^+\not= \pm R^-\),  we may  replace the construction depicted
in Figure~\ref{figgen} by:
\begin{center}
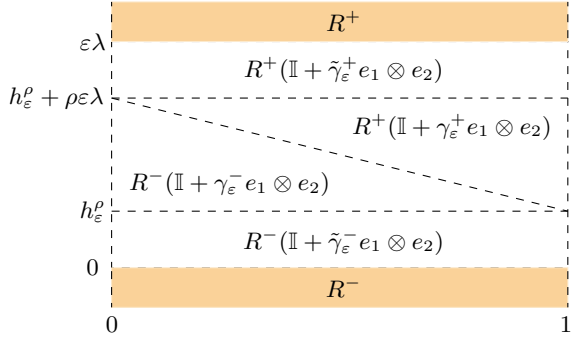
\begin{figure}[H]
\begin{tikzpicture}[scale=1.5]
  
\draw[dashed, black] 
(0,0) -- (4,0) -- (4, -2) -- (0, -2)
 -- (0,0);

 \draw[Korange!20,fill=Korange!40] 
(0,0) -- (0,.35) -- (4,.35) -- (4,0)
 -- (0,0);
 
  \draw[Korange!20,fill=Korange!40] 
(0,-2) -- (0,-2.35) -- (4,-2.35) -- (4,-2)
 -- (0,-2);


\draw[dashed,black] 
(0,0.03) -- (0,.35) ;

\draw[dashed,black] 
(4,0.03) -- (4,.35) ;

\draw[dashed,black] 
(0,-2.03) -- (0,-2.35) ;

\draw[dashed,black] 
(4,-2.03) -- (4,-2.35) ;


\draw[dashed,black] (0,-.5) -- (4,-.5) ;

\draw[dashed,black] (0,-.5) -- (4,-1.5) ;


\draw[dashed,black] (0,-1.5) -- (4,-1.5) ;


\draw(-.2,-2) node     
        {\begin{small} $\textcolor{black}{0}$\end{small}
        };

\draw(-.2,-1.5) node     
        {\begin{small} $\textcolor{black}{h_\epsi^\rho}$\end{small}
        };

\draw(-.52,-.5) node     
        {\begin{small} $\textcolor{black}{h_\epsi^\rho + \rho\epsi\lambda}$\end{small}
        };

\draw(-.2,0) node     
        {\begin{small} $\textcolor{black}{\varepsilon\lambda}$\end{small}
        };

\draw(-.03,-2.5) node     
        {\begin{small} $\textcolor{black}{0}$\end{small}
        };
        
\draw(3.95,-2.5) node     
        {\begin{small} $\textcolor{black}{1}$\end{small}
        };  
        

\draw(2,-1.77) node     
        {\begin{small} $\textcolor{black}{R^-(\mathbb{I}
        + \tilde \gamma^-_\epsi  e_1 \otimes e_2)}$\end{small}
        };

\draw(1,-1.25) node     
        {\begin{small} $\textcolor{black}{R^-(\mathbb{I}
        + \gamma_\epsi^- e_1 \otimes
e_2)}$\end{small}
        };

\draw(2.95,-.75) node     
        {\begin{small} $\textcolor{black}{R^+(\mathbb{I}
        + \gamma^+_\epsi e_1 \otimes e_2)}$\end{small}
        };

\draw(2,-.25) node     
        {\begin{small}  $\textcolor{black}{R^+(\mathbb{I}
        + \tilde\gamma^+_\epsi e_1 \otimes
e_2)}$\end{small}
        };       

 
\draw(2,.18) node     
        {\begin{small} $\textcolor{black}{R^+}$\end{small}
        };  
 
\draw(2,-2.17) node     
        {\begin{small} $\textcolor{black}{R^-}$\end{small}
        };   
%

\draw(7,0.1) node     
        {\begin{small} $\textcolor{black}{
        \psi^+ - \psi^- = \alpha R^+e_1 + \beta R^-e_1
        }$
        \end{small}
        };

\draw(7,-0.3) node     
        {\begin{small} $\textcolor{black}{
        \theta\in (-\pi,\pi)  \setminus\{0\}   \text{ rotation angle of } (R^-)^T R^+
        }$
        \end{small}
        }; 
        
\draw(7,-0.7) node     
        {\begin{small} $\textcolor{black}{
         \rho\in(0,1), \enspace h_\eps^\rho:= \frac{\eps\lambda - \rho\eps \lambda}{2}
        }$
        \end{small}
        };

\draw(7,-1.15) node     
        {\begin{small} $\textcolor{black}{       
\gamma^+_\epsi:=\frac{1}{\rho\epsi\lambda} + \tan(\frac{\theta}{2}
),\enspace
\gamma^-_\epsi:=\frac{1}{\rho\epsi\lambda} - \tan(\frac{\theta}{2}
)
        }$
        \end{small}
        };                
               
\draw(7,-1.6) node     
        {\begin{small} $\textcolor{black}{
     \tilde \gamma^+_\epsi  \text{ satisfies } \alpha=\lim_{\epsi \to 0} \tilde \gamma^+_\epsi(\epsi
\lambda -
h^\rho_\epsi - \rho\epsi\lambda)
        }$
        \end{small}
        };  
        
\draw(7,-2) node     
        {\begin{small} $\textcolor{black}{
       \tilde \gamma^-_\epsi  \text{ satisfies }    \beta -1=\lim_{\epsi
\to 0} \tilde \gamma^-_\epsi 
h^\rho_\epsi
        }$
        \end{small}
        };          
\end{tikzpicture}
\caption{Alternative construction of   $V_{\eps}$  if \(R^+\not= \pm R^-\).}
\label{figi}
\end{figure}
\end{center}\vspace{-5mm}

\item[(ii)] 
If  \(R\) is constant,  i.e.,~\(R^+=  R^-\),  and  $\psi^+ - \psi^-$ is not parallel to $Re_1$,  the construction in Figure~\ref{figgen} can be replaced  by:
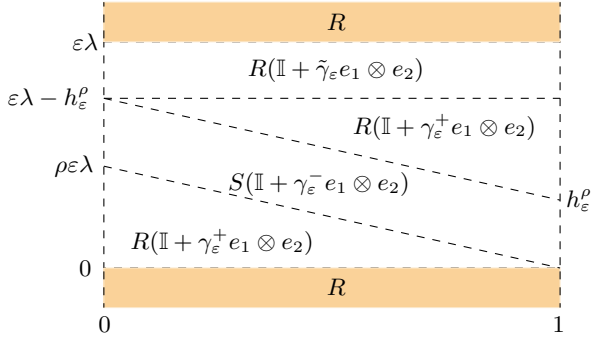
\begin{center}
\begin{figure}[H]
\begin{tikzpicture}[scale=1.5]
  
\draw[dashed, black] 
(0,0) -- (4,0) -- (4, -2) -- (0, -2)
 -- (0,0);

 \draw[Korange!20,fill=Korange!40] 
(0,0) -- (0,.35) -- (4,.35) -- (4,0)
 -- (0,0);
 
  \draw[Korange!20,fill=Korange!40] 
(0,-2) -- (0,-2.35) -- (4,-2.35) -- (4,-2)
 -- (0,-2);


\draw[dashed,black] 
(0,0.03) -- (0,.35) ;

\draw[dashed,black] 
(4,0.03) -- (4,.35) ;

\draw[dashed,black] 
(0,-2.03) -- (0,-2.35) ;

\draw[dashed,black] 
(4,-2.03) -- (4,-2.35) ;

\draw[dashed,black] (0,-.5) -- (4,-.5) ;

\draw[dashed,black] (0,-.5) -- (4,-1.4) ;

\draw[dashed,black] (0,-1.1) -- (4,-2) ;



\draw(-.2,-2) node     
        {\begin{small} $\textcolor{black}{0}$\end{small}
        };

\draw(-.3,-1.1) node     
        {\begin{small} $\textcolor{black}{\rho\epsi\lambda}$\end{small}
        };

\draw(-.52,-.5) node     
        {\begin{small} $\textcolor{black}{\epsi\lambda -h_\epsi^\rho }$\end{small}
        };

\draw(-.2,0) node     
        {\begin{small} $\textcolor{black}{\varepsilon\lambda}$\end{small}
        };

\draw(-.03,-2.5) node     
        {\begin{small} $\textcolor{black}{0}$\end{small}
        };
        
\draw(3.95,-2.5) node     
        {\begin{small} $\textcolor{black}{1}$\end{small}
        };  
        
         \draw(4.13,-1.4) node     
        {\begin{small} $\textcolor{black}{h_{\varepsilon}^{\rho}}$\end{small}
        };


\draw(1,-1.77) node     
        {\begin{small} $\textcolor{black}{R(\mathbb{I}
        + \gamma_\epsi^+ e_1 \otimes e_2)}$\end{small}
        };

\draw(1.85,-1.25) node     
        {\begin{small} $\textcolor{black}{S(\mathbb{I}
        + \gamma_\epsi^- e_1 \otimes
e_2)}$\end{small}
        };

\draw(2.95,-.75) node     
        {\begin{small} $\textcolor{black}{R(\mathbb{I}
        + \gamma_\epsi^+ e_1 \otimes e_2)}$\end{small}
        };

\draw(2,-.25) node     
        {\begin{small}  $\textcolor{black}{R(\mathbb{I}
        + \tilde\gamma_\epsi e_1 \otimes
e_2)}$\end{small}
        };       

 
\draw(2,.18) node     
        {\begin{small} $\textcolor{black}{R}$\end{small}
        };  
 
\draw(2,-2.17) node     
        {\begin{small} $\textcolor{black}{R}$\end{small}
        };   
%

\draw(7,0.1) node     
        {\begin{small} $\textcolor{black}{
        S\in SO(2)\!:  Re_1 \text{ and } Se_1
        \text{ are linearly independent}
        }$
        \end{small}
        };

\draw(7,-0.7) node     
        {\begin{small} $\textcolor{black}{
       \theta\in (-\pi,\pi) \setminus \{0\}  \text{ rotation angle of } R^T
S \color{black}
        }$
        \end{small}
        }; 
        
\draw(7,-0.3) node     
        {\begin{small} $\textcolor{black}{
         \psi^+ - \psi^- = \alpha Re_1 + \beta Se_1,\enspace \beta\not=0, \enspace \iota:={\rm sign}(\beta)
        }$
        \end{small}
        };         

\draw(7,-1.15) node     
        {\begin{small} $\textcolor{black}{
         \rho:=\tfrac{\iota}{2\beta +\iota}\in(0,1),
\enspace h_\eps^\rho:= \frac{\eps\lambda
- \rho\eps \lambda}{2}
        }$
        \end{small}
        };  
        
\draw(7,-1.6) node     
        {\begin{small} $\textcolor{black}{       
\gamma_\epsi^+:=\iota\frac{1}{\rho\epsi\lambda} + \tan(\frac{\theta}{2}
),\enspace
\gamma^-_\epsi:=\iota\frac{1}{\rho\epsi\lambda} - \tan(\frac{\theta}{2}
)
        }$
        \end{small}
        };                
               
\draw(7,-2) node     
        {\begin{small} $\textcolor{black}{
     \tilde \gamma_\epsi  \text{ satisfies } \alpha-\iota=\lim_{\epsi
\to 0} \tilde \gamma_\epsi
h^\rho_\epsi 
        }$
        \end{small}
        };

\end{tikzpicture}
\caption{Alternative construction of   $V_{\eps}$  if   \(R\) is constant and  $\psi^+ - \psi^-$
is not parallel to $Re_1$.}
\label{figii}
\end{figure}
\end{center}\vspace{-5mm}

\item[(iii)] 
If  \(R\) is constant,  i.e.,~\(R^+=  R^-\),  and  $\psi^+ - \psi^-$
is  parallel to $Re_1$, then  we can use the following construction in place of Figure~\ref{figgen}:
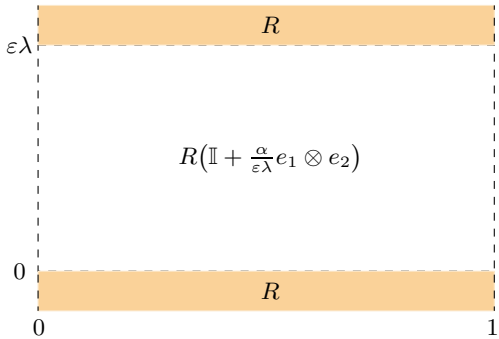
\begin{center}
\begin{figure}[H]
\begin{tikzpicture}[scale=1.5]
  
\draw[dashed, black] 
(0,0) -- (4,0) -- (4, -2) -- (0, -2)
 -- (0,0);

 \draw[Korange!20,fill=Korange!40] 
(0,0) -- (0,.35) -- (4,.35) -- (4,0)
 -- (0,0);
 
  \draw[Korange!20,fill=Korange!40] 
(0,-2) -- (0,-2.35) -- (4,-2.35) -- (4,-2)
 -- (0,-2);


\draw[dashed,black] 
(0,0.03) -- (0,.35) ;

\draw[dashed,black] 
(4,0.03) -- (4,.35) ;

\draw[dashed,black] 
(0,-2.03) -- (0,-2.35) ;

\draw[dashed,black] 
(4,-2.03) -- (4,-2.35) ;







\draw(-.2,-2) node     
        {\begin{small} $\textcolor{black}{0}$\end{small}
        };

\draw(-.2,-1.5) node     
        {\begin{small} $\textcolor{black}{}$\end{small}
        };

\draw(-.52,-.5) node     
        {\begin{small} $\textcolor{black}{}$\end{small}
        };

\draw(-.2,0) node     
        {\begin{small} $\textcolor{black}{\varepsilon\lambda}$\end{small}
        };

\draw(-.03,-2.5) node     
        {\begin{small} $\textcolor{black}{0}$\end{small}
        };
        
\draw(3.95,-2.5) node     
        {\begin{small} $\textcolor{black}{1}$\end{small}
        };  
        

\draw(2,-1) node     
        {\begin{small} $\textcolor{black}{R\big(\mathbb{I}
        + \frac{\alpha}{\epsi \lambda} e_1 \otimes e_2\big)}$\end{small}
        };

 
\draw(2,.18) node     
        {\begin{small} $\textcolor{black}{R}$\end{small}
        };  
 
\draw(2,-2.17) node     
        {\begin{small} $\textcolor{black}{R}$\end{small}
        };   
%

\draw(7,-.7) node     
        {\begin{small} $\textcolor{black}{
        \psi^+ - \psi^- = \alpha Re_1
        }$
        \end{small}
        };

\draw(7,-1.2) node     
        {\begin{small} $\textcolor{black}{
        \alpha = \iota
        |\psi^+ - \psi^-|,\enspace \iota:= {\rm sign}(( \psi^+
- \psi^-)\cdot
        Re_1)
        }$
        \end{small}
        }; 
\end{tikzpicture}
\caption{Alternative construction of  $V_{\eps}$  if   \(R\) is
constant and  $\psi^+ - \psi^-$
is  parallel to $Re_1$.}
\label{figiii}
\end{figure}%
\end{center}\vspace{-5mm}
\end{itemize}
Note that in case~(i), the slope $\rho$ of the interfaces can
attain any
value between $0$ and $1$, while in~(ii),  $\rho$
is determined
by the value of  $\beta$.  In terms of the energies, the construction
in case~(iii) provides an optimal approximation, which will be
detailed in Section~\ref{sect:regularization}.
\end{remark}

We proceed by extending Lemma~\ref{lem:existence_admissible1} to arbitrary
functions  \(u\in\Acal\cap SBV_\infty(\Omega;\RR^2)\).

\begin{proposition}\label{prop:SBVinfty}
Let $\Omega=(0,1)\times (-1,1)$.  Then, for every $u\in \Acal\cap SBV_{\infty}(\Omega;\R^2)$, there  exists a sequence $(u_\eps)_\eps\subset W^{1,1}(\Omega;\RR^2)$
with \(\int_\Omega u_\epsi\dd{x} = \int_\Omega u\dd{x} \) and $u_\epsi\in \Acal_\eps$ for all \(\epsi\), and such that   $u_\eps\weaklystar
u$ in $BV(\Omega;\R^2)$ or, in other words, 
\begin{align*}
\Acal\cap SBV_\infty(\Omega;\R^2) \subset \Acal_0,
\end{align*} 
cf.~\eqref{eq:def-A0}. 
\end{proposition}
\begin{proof}
In view of Remark~\ref{rk:subsets-A}~(c),  it holds that  \(J_u
= \bigcup_{i=1}^\ell (0,1) \times \{a_i\}\) for some \(\ell\in\NN\)
and \(a_i\in(-1,1)\) with \(a_1 < a_2<\dots<a_\ell\), and setting \(a_0:=-1\) and \(a_{\ell+1}:=1\),  gives 
\begin{align}
Du= &\,\sum_{i=0}^\ell R_i(\Ibb + \gamma  e_1\otimes e_2)\Lcal^2\lfloor  \bigl((0,1)\times (a_i, a_{i+1})\bigr)
\notag\\ &+ \sum_{i=1}^{\ell} [(R_{i} - R_{i-1}) x_1 e_1 +(R_ia_i
e_2 +\psi_i^+ - R_{i-1}a_i
e_2 -\psi_{i}^-)]\otimes e_2 \Hcal^1\lfloor \bigl((0,1)\times\{a_i\}\bigr),\label{eq:Du-Apc}
\end{align} 
where  \(\gamma\in L^1(-1,1)\), and  \(R_i\in SO(2)\) and \(\psi_i\in\RR^2\)
for \(i=0,...,\ell\). 

We now perform a similar construction  as in  Lemma~\ref{lem:existence_admissible1}
in a convenient  softer  layer \textit{near} each \(a_i\), accounting for the possibility
that one or more of the jump lines may not intersect \(\eps\Ysoft\cap
\Omega\), and replacing
\(R^+\) by \(R_{i}\),  \(R^-\) by \(R_{i-1}\), \(\psi^+\) by \(R_ia_i
e_2 +\psi_i^+ \), and \(\psi^-\) by
\(R_{i-1}a_i
e_2 +\psi_{i}^-\).

 To be precise,  fix  \(\epsi>0\) and \(i\in \{1, ..., \ell\} \).  Let \(S_i\in SO(2)\) be such that \textit{(i)} \(S_i\not\in\{R_{i-1},R_i\}\);
\textit{(ii)} \(S_ie_1 \) and \(R_{i}e_1\) are linearly independent;
\textit{(iii)}~%
\(\theta_{i}^-\), \(\theta_{i}^+ \in (-\pi,\pi)  \setminus\{0\}\)  are the rotation angles of \(S_i^T R_{i-1}\) and \(S_i^T R_{i}\), respectively. By \textit{(ii)},
there exist \(\alpha_i\), \(\beta_i\in \RR\) such that
\begin{equation}
\begin{aligned}
\label{eq:p+-p-Apc}
R_ia_i
e_2 +\psi_i^+ - R_{i-1}a_i
e_2 -\psi_{i}^- = \alpha_i R_ie_1 + \beta _iS_ie_1.
\end{aligned}
\end{equation}
 Moreover, we set 
\begin{equation*}
\begin{aligned}
\gamma^+_{\epsi,i}:= \frac{4\alpha_i}{\epsi\lambda},\enspace
\gamma_{\epsi,i}:= \frac{4\beta_i}{\epsi\lambda},\enspace
\mu^\pm_{\epsi,i}:=
\pm\frac{4}{\epsi\lambda} + \tan\Big( \frac{\theta^\pm_i}{2}\Big),
 \enspace \tilde \mu^\pm_{\epsi,i}:=
\pm\frac{4}{\epsi\lambda} - \tan\Big( \frac{\theta^\pm_i}{2}\Big),
\end{aligned}
\end{equation*}
and let \(\kappa_\epsi^i
\in \ZZ\) be the unique integer such that \(a_i \in \epsi[ \kappa_\epsi^i, \kappa_\epsi^i+1)\).
 Observing that   \(a_i\not = a_j\)  for $i, j\in \{1, \ldots, \ell\}$ with  \(i\not= j\) and
\(a_i\in(-1,1)\)  for all $i\in \{1, \ldots, \ell\}$,  we may assume
that the sets $\{\epsi[ \kappa_\epsi^i, \kappa_\epsi^i+1)\}_{ i=1, \ldots, \ell }$ are pairwise disjoint, and that  \(\bigcup_{i=1}^\ell \epsi[ \kappa_\epsi^i,
\kappa_\epsi^i+1] \subset(-1,1)\)  (this is true for sufficiently small \(\epsi>0\)).
Finally,  with  \( \kappa_\epsi^0:= -\lambda-\tfrac1\epsi\) and \( \kappa_\epsi^{\ell
+1}:= \tfrac1\epsi\), let  \( V_\epsi  \in L^1(\Omega;\RR^{2\times 2})\)
be the function defined by 
\begin{equation*}
\begin{aligned}
 V_\epsi(x)  := \begin{cases}
R_{i}(\Ibb + \tfrac{\gamma}{\lambda} \mathbbm{1}_{\eps\Ysoft}e_1 \otimes e_2) &\text{if }
x\in  (0,1) \times (
\epsi\lambda+\epsi\kappa_\epsi^i,\epsi\kappa_\epsi^{i+1} )\text{ for some } i\in\{0,
.., \ell\},\\
R_{i}(\Ibb +\gamma^+_{\epsi,i} e_1 \otimes e_2) &\text{if }
x\in  (0,1) \times (\frac{3\epsi\lambda}4+\epsi\kappa_\epsi^i, \epsi\lambda+\epsi\kappa_\epsi^i )\text{ for some } i\in\{1, .., \ell\},\\
R_{i}(\Ibb +  \mu^+_{\epsi,i} e_1 \otimes e_2) & \text{if
} x_1\in  (0,1) \text{ and } x_2\in (-\frac{\epsi\lambda}4
 x_1 +\frac{3\epsi\lambda}4 +\epsi\kappa_\epsi^i,\frac{3\epsi\lambda}4 +\epsi\kappa_\epsi^i)\\&\text{for some } i\in\{1, .., \ell\},\\ 
S_i(\Ibb +  \tilde\mu^+_{\epsi,i} e_1 \otimes e_2) & \text{if
} x_1\in  (0,1) \text{ and } x_2\in (\frac{\epsi\lambda}2+\epsi\kappa_\epsi^i,-\frac{\epsi\lambda}4
 x_1 +\frac{3\epsi\lambda}4 +\epsi\kappa_\epsi^i)\\&\text{for some }  i\in\{1, .., \ell\}, \\  
S_i(\Ibb +  \gamma_{\epsi,i} e_1 \otimes e_2) &\text{if
} x\in  (0,1) \times (\frac{\epsi\lambda}4 +\epsi\kappa_\epsi^i,\frac{\epsi\lambda}2
 +\epsi\kappa_\epsi^i)\text{ for some } i\in\{1, .., \ell\},\\
S_i(\Ibb +  \tilde\mu^-_{\epsi,i} e_1 \otimes e_2) &\text{if
}  x_1\in  (0,1) \text{ and } x_2\in(\frac{\epsi\lambda}4  x_1+\epsi\kappa_\epsi^i,\frac{\epsi\lambda}4+\epsi\kappa_\epsi^i)  \text{ for some } i\in\{1, .., \ell\},\\
R_{i-1}(\Ibb +  \mu^-_{\epsi,i} e_1 \otimes e_2) &\text{if
}  x_1\in  (0,1) \text{ and } x_2\in (\epsi\kappa_\epsi^i,\frac{\epsi\lambda}4  x_1+\epsi\kappa_\epsi^i)  \text{ for some } i\in\{1, .., \ell\}.
\end{cases}
\end{aligned}
\end{equation*}

As in the proof of Lemma~\ref{lem:existence_admissible1}, invoking \cite[Lemma~3.1]{ChK17}  on  rank-one connections in
\(\Mcal_{e_1}\), we find that  $V_\eps$ is a gradient field, meaning that  there is \(u_\epsi \in W^{1,  \infty }(\Omega;\RR^2)\)
such that  \(\nabla u_\epsi =V_\eps\).  Adding a suitable
constant  allows us to  assume that \(\int_\Omega u_\epsi\dd{x} = \int_\Omega
u\dd{x}\).  By construction,  $(u_\eps)_\eps$
is a uniformly
bounded sequence
in $W^{1,1}(\Omega;\R^2)$ such that $u_\epsi\in \Acal_\eps$ for
all \(\epsi\) (see \eqref{eq:def-Aepalt}).   To prove
that  
  $u_\eps\weaklystar
u$ in $BV(\Omega;\R^2)$, it suffices to show that%
\begin{equation}
\label{eq::rc_weakconv-Apc}
\begin{aligned}
Du_\eps\weaklystar Du \text{  in } \Mcal(\Omega;\R^{2\times 2}).
\end{aligned}
\end{equation}

The proof of \eqref{eq::rc_weakconv-Apc} follows along the lines of \eqref{eq::rc_weakconv}. For this reason, we only highlight the main differences.  First, note that  the conditions \(\epsi\kappa_\epsi^0=  - \epsi \lambda -1= -\epsi\lambda + a_0\), 
\(\epsi\kappa_\epsi^{\ell
+1}=1= a_{\ell+1}\), and  \(\epsi
\kappa_\epsi^i\leq a_i \leq 
\epsi(\kappa_\epsi^i+1)\) yield
\begin{equation*}
\begin{aligned}
\lim_{\epsi \to0} \epsi \kappa_\epsi^i = a_i\enspace \text{ for all } i\in\{0,...,\ell + 1\}.
\end{aligned}
\end{equation*}
Hence,   \(\mathbbm{1}_{(0,1)\times (
\epsi\lambda+\epsi\kappa_\epsi^i,\epsi\kappa_\epsi^{i+1} )} \to \mathbbm{1}_{(0,1)\times (a_i,a_{i+1} )}\) 
and \(\gamma\mathbbm{1}_{(0,1)\times (
\epsi\lambda+\epsi\kappa_\epsi^i,\epsi\kappa_\epsi^{i+1} )} \to
\gamma \mathbbm{1}_{(0,1)\times (a_i,a_{i+1} )}\) in \(L^1(\Omega)\)
for \(i\in\{0,...,\ell
+ 1\}\). 
On the other hand, by the Riemann--Lebesgue lemma, we have \(\mathbbm{1}_{\eps\Ysoft}
\weaklystar \lambda\) in \(L^\infty(\RR^2)\);   thus, 
\begin{equation*}
\begin{aligned}
\lim_{\epsi\to0} \int_{(0,1) \times (
\epsi\lambda+\epsi\kappa_\epsi^i,\epsi\kappa_\epsi^{i+1} )} \nabla u_\epsi(x) \ffi(x)\dd x &= \lim_{\epsi\to0} \int_{\Omega} R_{i}(\Ibb + \tfrac{\gamma}{\lambda} \mathbbm{1}_{\eps\Ysoft}e_1
\otimes e_2)\mathbbm{1}_{(0,1)\times (
\epsi\lambda+\epsi\kappa_\epsi^i,\epsi\kappa_\epsi^{i+1} )}
\ffi(x)\dd x\\&= \int_{(0,1) \times (a_i, a_{i+1})} R_i
(\Ibb +\gamma(x_2) e_1\otimes  e_2)\ffi(x)\dd x
\end{aligned}
\end{equation*}
for all \(i\in\{0,...,\ell
\}\) and \(\ffi\in C_0(\Omega)\). Arguing as in \eqref{eq:limve1} with the change of variables \(z= \tfrac4{\epsi\lambda}(x_2 - \epsi \kappa_\epsi^i)\), leads to 
\begin{align*}
&\lim_{\epsi\to0} \int_{(0,1) \times (\epsi\kappa_\epsi^i,\frac{\epsi\lambda}4  x_1+\epsi\kappa_\epsi^i)} \nabla
u_\epsi(x) \varphi(x)\dd x \notag\\
&\quad = \lim_{\epsi\to0} \int_0^1\!\! \int_{\epsi\kappa_\epsi^i}^{\frac{\epsi\lambda}4
 x_1 + \epsi\kappa_\epsi^i}  R_{i-1}  \big(\mathbb{I}
        + \tan\big( \tfrac{\theta^-_i}{2}\big)e_1 \otimes e_2- \tfrac{4}{\epsi\lambda}e_1
\otimes e_2\big) \varphi(x)\dd x_2\!\dd x_1
 \notag\\
&\quad = \lim_{\epsi\to0} \int_0^1\!\! \int_0^{
 x_1}  R_{i-1}  \big(\tfrac{\epsi\lambda}4\mathbb{I}
        + \tfrac{\epsi\lambda}4\tan\big( \tfrac{\theta^-_i}{2}\big)e_1 \otimes
e_2 -e_1
\otimes e_2 \big)
\varphi(x_1,\tfrac{\epsi\lambda}4z+ \epsi\kappa_\epsi^i)\dd z\!\dd x_1\notag\\
&\quad = - \int_0^1\!\! \int_0^{
 x_1}  R_{i-1}  e_1
\otimes e_2
\varphi(x_1,a_i)\dd z\!\dd x_1 = - \int_0^1  R_{i-1}  x_1e_1
\otimes e_2
\varphi(x_1,a_i)\dd x_1
\end{align*}
 for all \(i\in\{1,...,\ell
\}\) and \(\ffi\in C_0(\Omega)\).  Similarly,  one can calculate the counterparts to \eqref{eq:limve2}--\eqref{eq:limve6} in the present setting. In view of~\eqref{eq:Du-Apc} and \eqref{eq:p+-p-Apc},  we deduce \eqref{eq::rc_weakconv-Apc}, which ends the proof. 
\end{proof}

\begin{remark}[On the construction in Proposition~\ref{prop:SBVinfty}]\label{rmk:onSBVinfty}
We observe that the sequence  of Lipschitz functions  \((u_\epsi)_\epsi\) constructed
in Proposition~\ref{prop:SBVinfty}  to approximate a given $u\in \Acal\cap SBV_\infty(\Omega;\R^2)$  is such that
\begin{equation*}
\begin{aligned}
\lim_{\epsi\to0} \int_\Omega |\nabla u_\epsi|\dd x  \sim |D u|(\Omega) + 2\ell.
\end{aligned}
\end{equation*}
In other words, the asymptotic behavior of the total variation
of \((u_\epsi)_\epsi\) incorporates a positive term that is proportional
to the number of jumps of the  limit  function. This fact  prevents
us from bootstrapping the argument in Proposition~\ref{prop:SBVinfty}
to generalize it  to an arbitrary function in \(\Acal \cap SBV(\Omega;\RR^2)\). 
\end{remark}

 An analogous  statement to Proposition~\ref{prop:SBVinfty} holds in \(\Acal^\parallel\).

\begin{proposition}\label{prop:Aparallel}
 Let $\Omega=(0,1)\times (-1,1)$.
 If $u\in \Acal^\parallel$,   then there exists
a sequence $(u_\eps)_\eps\subset W^{1,1}(\Omega;\RR^2)$
such that $u_\epsi\in \Acal_\eps$ for all \(\epsi\) and  $u_\eps\weaklystar
u$ in $BV(\Omega;\R^2)$;  that is, 
\begin{align*}
\Acal^\parallel \subset \Acal_0.
\end{align*} 

\end{proposition}

\begin{proof} 
Let $u\in \mathcal{A}^\parallel$.  Based on~\eqref{eq:def-Apar} and
 \eqref{1dsplitting}, we can split $u$ into $u=v+w,$ 
 where 
 \begin{align}\label{vw}
 v(x) :=  Rx +\vartheta_a(x_2)
Re_1+c  \quad \text{and} \quad w(x):=\vartheta_s(x_2)
Re_1 \qquad  \text{ for $x\in \Omega$, }
\end{align}
with $R\in SO(2)$,
\(c\in\RR^2\), \(\vartheta_a\in W^{1,1}(-1,1)\), and \(\vartheta_s\in BV(-1,1)\) such that \(\vartheta_s'=0\).
 By construction, we have  that $v\in W^{1,1}(\Omega;\R^2)$
with $\nabla v(x)= R(\Ibb +\vartheta_a'(x_2) e_1\otimes e_2)$.

 For every $\eps>0$, let  $v_\eps \in W^{1,1}(\Omega;\R^2)$ be the function satisfying $\int_{\Omega} v_\eps \dd{x} = \int_{\Omega}
v\dd{x}$ and 
\begin{equation}
\label{eq:grad-v-ep}
\nabla v_\eps (x)=R\Big(\mathbb{I}  + \frac{\vartheta_a' (x_2)}{\lambda}\mathbbm{1}_{\eps\Ysoft}(x)e_1\otimes
e_2\Big).
\end{equation} 
 By the  Riemann--Lebesgue lemma,
\begin{align}\label{conv56}
v_\eps\weakly v\quad\text{ in $W^{1,1}(\Omega;\R^{2\times 2})$.}
\end{align}

On the other hand, applying Lemma~\ref{lem:approx_layers} to $\vartheta_s$, we can find a  sequence $(\vartheta_\eps)_\eps\subset W^{1, \infty}(-1,1)$ such that \(\vartheta_\epsi \tostar \vartheta_s\) in \(BV(-1,1)\)
and $\vartheta_\epsi' = 0$ on $\eps I_{\rm rig}\cap  (-1,1)$.  Then, setting
$w_\eps(x):= \vartheta_\epsi(x_2) Re_1 +\int_\Omega (w-\vartheta_\epsi(x_2) Re_1)\dd{x}$  yields  
\begin{align}\label{iden77}
\nabla w_\epsi(x)= \vartheta_\epsi'(x_2) Re_1\otimes e_2 =
\vartheta_\epsi'(x_2) \mathbbm{1}_{\eps\Ysoft}Re_1\otimes
e_2
\end{align}
and  
\begin{align}\label{conv57}
w_\eps \tostar w\quad \text{ in $BV(\Omega;\R^2)$.}
\end{align}

 We define the maps  \(u_\eps
:= v_\eps+w_\eps \) in $W^{1,1}(\Omega;\RR^2)$ for every $\eps$,
  
 and infer from~\eqref{eq:grad-v-ep} and~\eqref{iden77}  that 
\begin{align*}
\nabla u_\epsi
=R(\mathbb{I}  + \gamma_\epsi e_1\otimes
e_2\big),
\end{align*} where \(\gamma_\epsi (x):=
\big(\frac{\vartheta_a'(x_2)}{\lambda} +\vartheta_\epsi'(x_2)\big)\mathbbm{1}_{\eps\Ysoft}(x) \)
 is a function in \(L^1(\Omega)\) satisfying  $\gamma_\epsi=0$ in $\eps\Yrig\cap\Omega$. In particular, $u_\eps\in \Acal_\eps$ for all $\eps$.

Combining~\eqref{conv56} and~\eqref{conv57} shows that
 $u_\eps \weaklystar v+w=u$ in $BV(\Omega;\R^2)$, which finishes the proof. 
\end{proof}

Finally, we prove Theorem~\ref{thm:SBVinfty}.

\begin{proof}[Proof of Theorem~\ref{thm:SBVinfty}]
In view of the discussion in Section~\ref{subs:geom}, it suffices
to prove the  statement  on a  rectangle  of the form  \((c_\Omega,d_\Omega) \times (a_\Omega,b_\Omega)\),
where we recall~\eqref{aOmega} and~\eqref{cOmega}. 
 A simple modification
of the proofs of Propositions~\ref{con:compactness},
\ref{prop:SBVinfty}, and~\ref{prop:Aparallel}
 shows that these results hold for any
such  rectangles,  from which  Theorem~\ref{thm:SBVinfty}
follows. 
\end{proof}

\section{  A lower bound on the homogenized energy  }\label{sec:par}


In this section, we present  partial results for the homogenization problem for layered composites with rigid components discussed in the Introduction. More precisely, we establish a lower bound estimate on the asymptotic behavior of the sequence of energies $(E_{\eps})_{\eps}$ (see \eqref{eq:defEeps}), and highlight the main difficulties in the construction of matching upper bounds. Note that the following analysis is restricted to the case $s=e_1$.

As a start, we first give alternative representations for the involved energies,  which will be useful in the sequel.

\begin{remark}[Equivalent formulations for \boldmath{$E_{\eps}$} and \boldmath{$E$}]
\label{rmk:formulationsofEe} 
 
 In view of the definition of \(\Acal_\epsi\) (see~\eqref{eq:def-Aep}), it is straightforward to check  that the functional $E_\eps$ in \eqref{eq:defEeps} satisfies
\begin{align*}
E_\eps(u)  & =\begin{cases} \displaystyle \int_\Omega \sqrt{\abs{\partial_2
u}^2-1}\dd{x} & \text{if $u\in \Acal_\eps$,}\\[0.1cm]
\infty & \text{otherwise,}
\end{cases} =\begin{cases} \displaystyle \int_\Omega \sqrt{|\nabla u|^2 -2\det
\nabla u}\dd{x} & \text{if $u\in \Acal_\eps$,}\\[0.1cm]
\infty & \text{otherwise,}
\end{cases}
\end{align*}
for \(u\in L^1_0(\Omega;\RR^2)\).  Similarly, according to Proposition~\ref{prop:charA},  the functional $E$ from \eqref{eq:defE} can be expressed as
\begin{align*}
E(u) & = \begin{cases}
\displaystyle \int_{\Omega} |\gamma| \dd{x} + |D^s u|(\Omega)& \text{if
$u\in \Acal$,}\\
\infty & \text{otherwise,}
\end{cases}\end{align*}
for \(u\in L^1_0(\Omega;\RR^2)\). 
\end{remark}

 We can now provide a bound from below on $\Gamma$-$\liminf_{\eps\to 0}E_{\eps}$ and prove Theorem~\ref{theo:Gamma_convergence_new}.

\begin{proof}[Proof of Theorem \ref{theo:Gamma_convergence_new}]
For clarity, we subdivide the proof into two steps. In the first one, we establish the compactness property. In the second step, we  provide two alternative proofs of \eqref{eq:GliminfEeps}. The first proof is based on
a Reshetnyak's lower semicontinuity result, while the second version is more elementary, relying 
on the weak$^\ast$ lower semicontinuity of the total variation of a measure. Either
of the arguments highlights a different feature of the representation of
$\Acal$.

\medskip

\noindent\emph{Step~1: Compactness}.
Assume that \((u_\epsi)_\epsi\subset L^1_0(\Omega;\RR^2)\) is such that \(\sup_\epsi E_\eps(u_\eps)<\infty\). Then, \(u_\epsi\in \Acal_\epsi\) and \(\sup_\epsi 
\norm{\nabla u_\eps}_{L^1(\Omega;\R^{2\times 2})}
<\infty\). Hence, using the Poincar\'e--Wirtinger inequality,  there exist a subsequence \((u_{\epsi_j})_{j\in\NN}\) and $u\in L^1_0(\Omega;\RR^2) \cap BV(\Omega;\R^2)$ such that $u_{\eps_j}\weaklystar u$ in $BV(\Omega;\R^2)$. By  Proposition~\ref{con:compactness}, we conclude that $u\in L^1_0(\Omega;\RR^2) \cap \Acal$. 
\medskip

\noindent\emph{Step~2:  Lower bound.} Let \((u_\epsi)_\epsi\subset L^1_0(\Omega;\RR^2)\) and  \(u\in L^1_0(\Omega;\RR^2)\)   be  such that $u_\eps\to u$ in $L^1(\Omega;\R^2)$. We want to  show  that 
\begin{equation}
\label{eq:lowerbounde1}
\begin{aligned}
\liminf_{\eps \to 0} E_\eps( u_\eps) \geq E(u). 
\end{aligned}
\end{equation}

To prove \eqref{eq:lowerbounde1}, one may assume without loss of generality that the limit inferior on the right-hand side of \eqref{eq:lowerbounde1} is actually a limit and that this limit is finite.  Then, \(u_\eps \in \Acal_\eps\) and $E_\eps(u_\eps)<C$  for all $\eps$, where  \(C>0\) is a constant independent of \(\eps\).
 Hence, by Step~1, \(u_\eps \weaklystar u\) in \(BV(\Omega;\R^2)\) and \(u\in\Acal\). 

 \emph{Step~2a: Version~I. }  We observe that the map 
$\R^{2\times 2} \ni F\mapsto \sqrt{|F|^2-2\det F} $ is convex (see \cite{CoT05}) and  one-homogeneous. Consequently,  it follows from Remark~\ref{rmk:formulationsofEe}   and Reshetnyak's
lower semicontinuity theorem (see \cite[Theorem~2.38]{AFP00}), under consideration of our notation for the polar decomposition  $Du =g_u|Du|$ introduced in Section~\ref{subs:BV}, that 
\begin{equation}\label{eq:lbe11}
\begin{aligned}
\liminf_{\eps\to 0} E_\eps(u_\eps) = \liminf_{\eps\to0}\int_\Omega \sqrt{|\nabla u_\eps|^2 -2\det
\nabla u_\eps}\dd{x} \geq\int_\Omega  \sqrt{|g_u|^2-2\det g_u}
\dd{|Du|}. 
\end{aligned}
\end{equation}

 Since  \(\nabla u=R(\Ibb+\gamma e_1\otimes e_2)  \) with \(R\in BV(\Omega;SO(2)) \) and \((D^s u)e_1=0\) (see~\eqref{eq:charAa}), we have \(|\nabla u|^2
-2\det
\nabla u = |\gamma|^2\) for \(\Lcal^2\)-a.e.~in \(\Omega\) and \(\det g_u
= 0\) for \(|D^su|\)-a.e.~in \(\Omega\). Thus, 
\begin{equation}\label{eq:lbe12}
\begin{aligned}
&\int_\Omega \sqrt{ |g_u|^2-2\det g_u}
\dd{|Du|}\\
& \quad = \int_\Omega \sqrt{|\nabla u|^2
-2\det
\nabla u}\dd{x} + \int_\Omega \sqrt{|g_u|^2-2\det g_u\color{black}}\dd{|D^su|}\\
&\quad = \int_\Omega |\gamma| \dd{x} + \abs{D^su}(\Omega) = E(u),
\end{aligned}
\end{equation}
where we also used that the relation \(|g_u| = 1\) holds \(|D^su|\)-a.e.~in \(\Omega\).

 From \eqref{eq:lbe11} and \eqref{eq:lbe12}, we deduce \eqref{eq:lowerbounde1}.
 
\medskip
 \noindent\textit{Step~2b: Version~II. }  By the definition of $\Acal_\eps$ and~\eqref{Me1}, 
\begin{align*}
\nabla u_\eps = R_\eps + \gamma_\eps R_\eps e_1\otimes e_2
\end{align*}
with $R_\eps\in L^\infty(\Omega;SO(2))$ and $\gamma_\eps\in L^1(\Omega)$. 
Since $|\gamma_\eps R_\eps e_1\otimes e_2| =|\gamma_\eps|$ due to  \(|R_\epsi e_1| =1\), the estimate \(E_\epsi(u_\epsi)=\int_\Omega |\gamma_\epsi|\dd x  < C\)  implies  that \((\gamma_\eps R_\eps e_1\otimes e_2)_\epsi\) is uniformly bounded in $L^1(\Omega;\R^{2\times
2})$. Hence, after extracting a subsequence if necessary (not relabeled), 
\begin{align*}
\big(\gamma_\eps R_\eps e_1 \otimes e_2\big) \Lcal^2\lfloor \Omega\weaklystar \nu \qquad \text{in $\Mcal(\Omega;\R^{2\times 2})$}\end{align*}
 for some $\nu\in \Mcal(\Omega;\R^{2\times 2})$. Note further that the convergence $\nabla u_{\eps}\Lcal^2\lfloor \Omega\weaklystar Du$ in $\Mcal(\Omega;\R^{2\times 2})$ along with \eqref{eq:charAa}  yields also  $R_\eps\weaklystar R$ in $L^\infty(\Omega;\R^{2\times 2})$,  where $R\in L^\infty(\Omega;SO(2))$ satisfies in particular that $(\nabla u)e_1=Re_1$.  
 Hence, we have  
\begin{align*}
\nu = Du-R\Lcal^2\lfloor \Omega = (\gamma Re_1  \otimes  e_2 )\Lcal^2\lfloor \Omega +\ D^su,
\end{align*}
where the last equality follows again from~\eqref{eq:charAa}, and by the lower semicontinuity of the total variation, 
\begin{align*}
 \liminf_{\eps\to 0} E_\epsi(u_\epsi) &= \liminf_{\eps\to 0}\int_{\Omega} |\gamma_\eps| \dd{x}= \liminf_{\eps\to 0}\int_\Omega |\gamma_\eps R_\eps e_1\otimes e_2|\dd{x}   \\ & \geq |\nu|(\Omega)=\int_{\Omega} | \gamma R e_1 \otimes  e_2 |\dd x + |D^su|(\Omega) = \int_{\Omega}|\gamma|\dd{x} + |D^su|(\Omega) = E(u).\qedhere
\end{align*}
\end{proof}

\begin{remark}[Discussion regarding optimality  of the lower bound]\label{rmk:optimalitylb}
(a) The  lower bound \eqref{eq:GliminfEeps}
is optimal in \(\Acal\cap W^{1,1}(\Omega;\R^2)\cap L^1_0(\Omega;\RR^2)\)
and, more generally  (cf.~also Remark~\ref{rk:subsets-A}),  in the set \(\Acal^{\parallel}\cap L^1_0(\Omega;\RR^2)\)
introduced in \eqref{eq:def-Apar}.
Precisely, we have 
\begin{align}
\Gamma(L^1)\text{-}\lim_{\eps\to 0} E_\eps(u)
= E(u)\label{eq:GlimsupEeps}
\end{align}
for all \(u\in \Acal^{\parallel}\cap L^1_0(\Omega;\RR^2)\). 
In view of  \eqref{eq:GliminfEeps},
the proof of \eqref{eq:GlimsupEeps} is  directly  
related to the ability to construct
a recovery sequence. We  detail 
two alternative constructions
 for \(u\in \Acal^{\parallel}\)
in Section~\ref{sect:regularization}
below.  For illustration, we treat here the simpler special case where \(u\in \Acal\cap W^{1,1}(\Omega;\R^2)\cap L^1_0(\Omega;\RR^2)\). 

If $u\in \Acal\cap W^{1,1}(\Omega;\R^2)\cap L^1_0(\Omega;\RR^2)$, then $\nabla u= R(\Ibb + \gamma e_1\otimes e_2)$ for some $R\in SO(2)$ and $\gamma\in L^1(\Omega)$ such that $\partial_1\gamma =0$  (see Remark~\ref{rk:subsets-A}\,(a)).  As in the proof of  Proposition~\ref{prop:Aparallel},
    we take $(u_\eps)_\eps\subset W^{1,1}(\Omega;\R^2)\cap L^1_0(\Omega;\R^2)$  such that  $\nabla u_\eps=R(\Ibb +\frac{\gamma}{\lambda}\mathbbm{1}_{\eps\Ysoft} e_1\otimes e_2)$  for all $\eps$. 
Then, by the Riemann--Lebesgue
lemma, \(u_\epsi \weaklystar\ u\) in
\(BV(\Omega;\RR^2)\) and \(\lim_{\epsi\to0}
E_\epsi(u_\epsi) = E(u)\).

(b) The question whether \eqref{eq:GlimsupEeps}
holds for a larger class than \(\Acal^\parallel\)
is open at this point. We observe that the gradient-based constructions in  Lemma~\ref{lem:existence_admissible1},
Remark~\ref{rmk:alternativeconst}~(i)--(ii), and
Proposition~\ref{prop:SBVinfty}  yield upper bounds on the $\Gamma\text{-}\limsup$, which, however, do not match  the lower bound of  Theorem~\ref{theo:Gamma_convergence_new}.  This  indicates
that,
in general, a more tailored approach
will be necessary. 

(c)  The upper bounds on the $\Gamma\text{-}\limsup$ of $(E_\eps)_\eps$ resulting from  
 Lemma~\ref{lem:existence_admissible1},
Remark~\ref{rmk:alternativeconst}~(i)--(ii), and
Proposition~\ref{prop:SBVinfty} can be quantified. As previously mentioned, the
constructions in Remark~\ref{rmk:alternativeconst}~(iii)  and
Proposition~\ref{prop:Aparallel} are even recovery sequences.
This is not the case for the general construction in Lemma~\ref{lem:existence_admissible1} and for those highlighted in Remark~\ref{rmk:alternativeconst}~(i)--(ii).  In the following, we suppose that \(u\in \Acal \cap SBV_\infty(\Omega;\R^2)\cap L^1_0(\Omega;\RR^2)\) has a single jump as in the statement of Lemma~\ref{lem:existence_admissible1}; i.e.,~
\begin{align*}
u(x)=\mathbbm{1}_{(0,1)\times (0,1)}(x)(R^+(x_2)x+\psi^+(x_2)) + \mathbbm{1}_{(0,1)\times (-1,0)} (x)(R^-(x_2)x+\psi^-(x_2))
\end{align*} with $R^{\pm}\in SO(2)$ and $\psi^\pm \in \R^2$.  Then,
\begin{equation*}
\begin{aligned}
E(u)=\int_0^1 |(R^+ - R^-)e_1 x_1 +  (\psi^+
- \psi^-)|\dd x_1,
\end{aligned}
\end{equation*}
 which can be estimated from above by 
\begin{align}\label{est6}
E(u)\leq  |R^+e_1-R^-e_1| \int_0^1 x_1 \dd{x_1} + |\psi^+-\psi^-|\leq 1 + |\psi^+-\psi^-|.
\end{align}

For the sequence \((u_\epsi)_\epsi\) constructed
in Lemma~\ref{lem:existence_admissible1}  (and Lemma~\ref{prop:SBVinfty}),  we obtain, recalling~\eqref{eq:p+-p-}, that
\begin{equation*}
\begin{aligned}
\lim_{\epsi\to0} E_\epsi(u_\epsi) = |\alpha| + |\beta| + 2 > |\alpha| + |\beta| +1
\geq    E(u).
\end{aligned}
\end{equation*}

\smallskip   

 Regarding the construction of \((u_\epsi)_\epsi\) in Remark~\ref{rmk:alternativeconst}~(i), it follows that 
\begin{equation*}
\begin{aligned}
\lim_{\epsi\to0} E_\epsi(u_\epsi) = |\alpha| + |\beta-1| + 1. 
\end{aligned}
\end{equation*}

This limit is strictly greater than $E(u)$ as we will show next. 
If $|\beta-1|>|\beta| $ (i.e.,~if $\beta<\frac{1}{2}$), this is an immediate consequence of~\eqref{est6}.
For \( \frac{1}{2}\leq \beta <1\), we use that \(\psi^+ - \psi^- = \alpha R^+e_1 + \beta R^-e_1\) 
yields
\begin{equation*}
\begin{aligned}
E(u)\leq \int_0^1|x_1 + \alpha| \dd x_1 +  \int_0^\beta (\beta-
x_1)  \dd x_1 +  \int_\beta^1 (x_1 -\beta)  \dd x_1 \leq 1 +
|\alpha| + \beta(\beta-1) <  1 + |\alpha| + |\beta-1|.
\end{aligned}
\end{equation*}
 If \(\beta \geq 1\), we note that $\lim_{\epsi\to0} E_\epsi(u_\epsi) = |\alpha| + \beta$,
and subdivide the estimate
of \(E(u)\) into three cases. Recalling the assumption \(R^+\not=\pm
R^- \),  we set \(c:=R^+e_1 \cdot R^-e_1 \in (-1,1)\)  to obtain 

\begin{equation*}
\begin{aligned}
E(u) =\int_0^1 \sqrt{(x_1+\alpha)^2 + (\beta - x_1)^2 + 2c(x_1+\alpha)(\beta - x_1) }\dd x_1.
\end{aligned}
\end{equation*}
Then,  we have for \(\alpha \geq 0\) that 
\begin{equation*}
\begin{aligned}
E(u) < \int_0^1 \sqrt{(x_1+\alpha)^2 + (\beta - x_1)^2 + 2(x_1+\alpha)(\beta
- x_1) }\dd x_1 = |\alpha + \beta| \leq  |\alpha| + \beta,
\end{aligned}
\end{equation*}
 for  \(\alpha \leq -1\) that 
\begin{equation*}
\begin{aligned}
E(u) &< \int_0^1 \sqrt{(x_1+\alpha)^2 + (\beta - x_1)^2 - 2(x_1+\alpha)(\beta
- x_1) }\dd x_1 = \int_0^1 (-2x_1-\alpha+\beta) \dd x_1\\ & = -1
-\alpha + \beta <-\alpha + \beta = |\alpha| + \beta, 
\end{aligned}
\end{equation*}
 and for  \(-1<\alpha <0\) that 
\begin{equation*}
\begin{aligned}
E(u) <\int_0^{-\alpha} (-2x_1-\alpha+\beta)\dd x_1 + \int_{-\alpha}^1
|\alpha+ \beta|\dd x_1 = \alpha+\beta +\alpha^2 <-\alpha + \beta = |\alpha| + \beta.
\end{aligned}
\end{equation*}

Summing up, we have shown that in the context of Remark~\ref{rmk:alternativeconst}\,(i),
\begin{equation*}
\begin{aligned}
\lim_{\epsi\to0} E_\epsi(u_\epsi) > E(u).
\end{aligned}
\end{equation*}

Finally, we consider  the sequence \((u_\epsi)_\epsi\)  constructed
in Remark~\ref{rmk:alternativeconst}~(ii).  Then,
\begin{equation*}
\begin{aligned}
\lim_{\epsi\to0} E_\epsi(u_\epsi) =   |\alpha - \iota| + |\beta| + 1, 
\end{aligned}
\end{equation*}
and since \(R^+=R^-\) in this case, 
\begin{equation*}
\begin{aligned}
E(u)= \sqrt{\alpha^2 +\beta^2 + 2\alpha\beta Re_1 \cdot
Se_1}.
\end{aligned}
\end{equation*}

Using the fact that \(Re_1 \cdot
Se_1 \in (-1,1)\), it can be checked that, also  here, 
we have%
\begin{equation*}
\begin{aligned}
\lim_{\epsi\to0} E_\epsi(u_\epsi) > E(u).
\end{aligned}
\end{equation*}

 \end{remark}

\section{ Homogenization of the regularized problem  }\label{sect:regularization}

This section is devoted to the proof of our main $\Gamma$-convergence result,  Theorem~\ref{thm:Gamma_convergence}.
We first provide an alternative characterization of the class
\(\Acal^\parallel\) of \emph{restricted asymptotically admissible deformations} introduced in \eqref{eq:def-Apar}.

\begin{lemma}\label{lem:charApar} 
Let \(\Omega=(0,1) \times (-1,1)\).
Then, \(\Acal^\parallel\) as in~\eqref{eq:def-Apar} admits the representation
\begin{align}\label{eq:def-Aparalt}
  \Acal^{\parallel}  = \{u\in BV(\Omega;\R^2)\!
:\, &\,  \nabla u=R(\Ibb+\gamma  e_1\otimes e_2) \text{ with
$R\in SO(2)$,  $\gamma\in L^1(\Omega)$ such that $\partial_1 \gamma=0$,} \notag\\ 
&  \,  D^su = (\varrho\otimes e_2)|D^su| \text{ with }\varrho\in L^1_{|D^su|}(\Omega;\mathbb{R}^2)
\text{ such that  } \\ &  |\varrho| =1 \text{ and } \varrho || Re_1 \text{ for } |D^su|\text{-a.e.\! in } \Omega\}. \nonumber %
\end{align}
\end{lemma} 

\begin{proof} 
Let \(\tilde \Acal^{\parallel}\) denote the set on the right-hand
side of \eqref{eq:def-Aparalt}. Arguing as in the beginning of
the proof of Proposition~\ref{prop:Aparallel} (precisely, with the notation of ~\eqref{eq:def-Apar}, we set $\gamma(x)=\vartheta_a'(x_2)$ for $x\in \Omega$, and observe that $(D^su)e_2=\mathcal{L}^1\lfloor{(0,1)} \boldsymbol\otimes D^s\vartheta_sRe_1$) 
and  exploiting  the polar
decomposition of measures (cf.~\eqref{eq:polarDj} and \eqref{eq:polarDc})
gives rise to $\Acal^{\parallel}\subset \tilde \Acal^{\parallel}$. 
  Conversely,  the inclusion \(\tilde \Acal^{\parallel}
  \subset \Acal\),  which follows from~\eqref{eq:charAa},  along with~\eqref{charA} yields that 
 $ \tilde \Acal^{\parallel}
 \subset \Acal^{\parallel}$. 
\end{proof}

We are now in a position to prove the $\Gamma$-convergence of the energies  $(E^{\eps}_{\delta})_\eps$ in 
\eqref{eq:def-E-ep-delta} as $\eps\to 0$. 

\begin{proof}[Proof of Theorem~\ref{thm:Gamma_convergence}]
As before  in the proofs of Theorems~\ref{thm:aymrigintro} and~\ref{thm:SBVinfty},  one may assume without loss
of generality that $\Omega=(0,1)\times (-1,1)$. We proceed in three steps.

\medskip 
\noindent\emph{Step~1: Compactness}. Let $(u_\eps)_\eps\subset W^{1,1}(\Omega;\mathbb{R}^2)  \cap L_0^1(\Omega;\R^2)$  be a  sequence  such that 
$E_\eps^{\delta}(u_\eps)\leq C$ for all $\eps>0$. Then,  because  $u_\eps\in \Acal_\eps$ for all $\eps$,
\begin{align}\label{nablaueps4}
\nabla u_\eps= R_\eps(\Ibb+\gamma_\eps e_1\otimes e_2)\in L^1(\Omega;\R^{2\times
2}),
\end{align}
and $\norm{\gamma_\eps}_{L^1(\Omega)}\leq
C$  for every $\eps>0$. Additionally, since  each map $R_\eps$ takes value in the set of proper rotations,  it holds that 
\(\Vert R_\epsi\Vert_{L^\infty(\Omega;\RR^{2\times
2})}^2=  2\)  for
all $\eps>0$. Consequently, 
  along with the Poincar\'e-Wirtinger inequality, 
\begin{align*}
\norm{ u_\eps}_{W^{1,1} (\Omega;\R^{2})}\leq C.
\end{align*} 
We further know that   $\norm{\partial_1 u_\eps}_{W^{1,
p}(\Omega;\R^2)}^p
=\norm{R_\eps e_1}_{W^{1, p}(\Omega;\R^2)}^p \leq  C/\delta$  for any $\eps$. 
Thus, after extracting subsequences
if necessary, one can find $u\in BV(\Omega;\R^2)$,  \(\gamma \in
\Mcal(\Omega)\), and   \(R\in
W^{1,p}(\Omega;\R^{2\times 2}) \)
such that  
\begin{align}
&u_\eps \weaklystar u \quad \text{ in $BV(\Omega;\R^2),$} \label{eq:uep-bv}\\
& \gamma_\eps \Lcal^{2}\weaklystar \gamma\quad\text{ in $\Mcal(\Omega)$,
} \notag
\\
&  R_\eps\weakly R\quad \text{  in $W^{1,p}(\Omega;\R^{2\times 2})$}.\label{eq:conv-w1p}
\end{align}

Recalling the compact embedding $W^{1,
p}(\Omega)\hookrightarrow\hookrightarrow C^{0, \alpha}(\overline{\Omega})$
for some $0<\alpha<1-\frac{2}{p}$, it follows even that $R\in W^{1,p}(\Omega;SO(2))\cap C^{0, \alpha}(\overline{\Omega};  \R^{2\times 2})$ and 
\begin{align}
R_\eps\to R\quad \text{ in $L^\infty(\Omega;\R^{2\times 2})$.}\label{strong_convergence}
\end{align}

  As a consequence of Proposition~\ref{con:compactness}, it holds that $u\in \Acal$. From  Proposition~\ref{prop:charA} and Alberti's rank one theorem (cf.~Section \ref{sec:prel}), we can further  infer  that  $R\in SO(2)$, $\gamma\in L^1(\Omega)$ with $\partial_1 \gamma=0$, and  that  $Du$ satisfies  
\begin{equation}\label{eq:almostAM}
\begin{aligned}
\nabla u=R(\Ibb+\gamma
e_1\otimes e_2)\quad \text{ and }\quad  D^su=(\varrho\otimes
e_2 )|D^su|,  
\end{aligned}
\end{equation}
 where $\varrho \in L_{|D^su|}(\Omega;\R^2)$ with  \(|\varrho|=1\) for \(|D^su|\)-a.e.~in $\Omega$.  To conclude that $u\in \Acal^\parallel$, in view of Lemma~\ref{lem:charApar}, it remains to   show  that
\begin{align}\label{Rconst}
 \varrho || Re_1\quad \text{\(|D^su|\)-a.e.~in $\Omega$.}
\end{align}

 To prove~\eqref{Rconst},  we first observe  that  for every $\eps$, 
the identity $(\nabla u_{\eps})e_2=R_{\eps}e_2+\gamma_{\eps}R_{\eps}e_1$, which follows from $u_\eps\in \Acal_\eps$,
yields 
\begin{align}\label{eq89}
\int_{\Omega}[(\nabla u_{\eps})e_2\cdot R_{\eps}e_2-1]\varphi
\dd{x}=0
\end{align}
for  all  $\varphi\in C^{\infty}_c(\Omega)$. Thus, 
by \eqref{eq:uep-bv}  and \eqref{strong_convergence} in combination with a weak-strong convergence argument, taking the limit $\eps\to 0$
in~\eqref{eq89} leads to 
\begin{align*}
\int_{\Omega}\varphi \dd{x} = \int_{\Omega}\varphi Re_2\cdot
\dd{((Du)e_2)} = \int_{\Omega}\varphi Re_2\cdot (\nabla u)e_2
\dd{x}+\int_{\Omega}
\varphi Re_2\cdot \dd{((D^su)e_2)}
\end{align*}
for every $\varphi\in C^{\infty}_c(\Omega)$.  Next, we plug in 
 the identities $(\nabla u)e_2=R e_2 +\gamma R e_1$ and \((D^su)e_2 = \varrho|D^s
u|\) (see \eqref{eq:almostAM}) to derive
that \begin{align*}
0=\int_{\Omega} \varphi Re_2\cdot \dd{((D^su)e_2)}=\int_{\Omega}\ffi
Re_2\cdot \varrho \dd{|D^su|}
\end{align*}
for every $\varphi\in C^{\infty}_c(\Omega)$,  which  completes the
proof of \eqref{Rconst}.

\medskip 
\noindent\emph{Step~2: Lower bound}.
Let \((u_\epsi)\subset L^1_0(\Omega;\RR^2)\)
and  \(u\in L^1_0(\Omega;\RR^2)\)  
be  such that $u_\eps\to u$ in $L^1(\Omega;\R^2)$.
We want to  show  that 
\begin{equation}
\label{eq:lowerbounddelta}
\begin{aligned}
E^\delta(u) \leq \liminf_{\eps \to 0} E^\delta_\eps(
u_\eps).
\end{aligned}
\end{equation}

  To prove \eqref{eq:lowerbounddelta}, we proceed as in the proof of \eqref{eq:lowerbounde1}, observing  in addition  that 
  \begin{align*}\liminf_{\epsi\to0} \delta\norm{\partial_1 u_\epsi}_{W^{1,p}(\Omega;\R^2)}^p = \liminf_{\epsi\to0} \delta\norm{R_\eps e_1}_{W^{1, p}(\Omega;\R^2)}^p \geq \delta\Vert Re_1\Vert^p_{W^{1,p}(\Omega;\R^2)} =   \delta |\Omega|
  \end{align*}
due to~\eqref{nablaueps4}  and \eqref{eq:conv-w1p} with \(R\in  SO(2)\).

\medskip 
\noindent\emph{Step~3: Upper bound}.
Let   \(u\in L^1_0(\Omega;\RR^2)\cap \Acal^\parallel\). We want to show that there is a sequence  \((u_\epsi)\subset L^1_0(\Omega;\RR^2)\)
  such that $u_\eps\to u$ in $L^1(\Omega;\R^2)$, and
\begin{equation}
\label{eq:upperbounddelta}
\begin{aligned}
E^\delta(u) \geq \limsup_{\eps \to 0} E^\delta_\eps(
u_\eps).
\end{aligned}
\end{equation}

Let $(u_{\eps})_{\eps}\subset W^{1,1}(\Omega;\R^2)\cap L_0^1(\Omega;\R^2)$ be the sequence constructed in the proof of  Proposition~\ref{prop:Aparallel},  that is, 
$u_\eps \in \Acal_\eps$ for every $\eps$ with
$$\nabla u_\eps  (x) = R\Big(\Ibb + \Big(\frac{\vartheta_a'  (x_2) }{\lambda} +\vartheta_\eps'(x_2)\Big) \mathbbm{1}_{\eps\Ysoft}(x) e_1\otimes e_2\Big),$$
 where $(\vartheta_{\eps})_{\eps}\subset W^{1,\infty}(-1,1)$  satisfies 
$$\lim_{\epsi\to0}\int_{-1}^1|\vartheta_{\eps}'|\,\dd{x_2} = |D^s\vartheta_s|(-1,1)=|D^s u|(\Omega),$$
and $u_\eps\weaklystar u$ in $BV(\Omega;\R^2)$. 
Recalling that  \(\vartheta'=\vartheta_a' + \vartheta_s'=\vartheta_a'\),
we have  
\begin{align*}
\limsup_{\eps\to 0} E_\eps^{\delta}(u_\eps)& \leq \lim_{\eps\to0}\bigg( \int_{\Omega} \frac{|\vartheta_a' (x_2)|}{\lambda} \mathbbm{1}_{\eps\Ysoft}(x)\dd{x} + \int_{-1}^1|\vartheta_{\eps}'(x_2)| \dd{x_2} + \delta \norm{Re_1}_{W^{1,p}(\Omega;\R^2)}^p  \bigg)  \\ &=  \int_{\Omega}|\vartheta'  (x_2)|\dd{x} + |D^s u|(\Omega) + \delta|\Omega|=E^\delta(u),
\end{align*} 
which proves \eqref{eq:upperbounddelta} and completes the proof of the theorem. 
\end{proof}

\begin{remark}[On compensated compacteness]\label{rk:cc}
We point out that if $u_\eps\in \Acal_\eps$, with  $\nabla u_{\eps}=R_{\eps}(\Ibb+\gamma_{\eps}e_1\otimes e_2)$  for $R_\eps\in L^\infty(\Omega;SO(2))$ and $\gamma_\eps\in L^1(\Omega)$ with $\gamma_\eps=0$ on $\eps\Yrig\cap \Omega$,  is such that $u_\eps\weaklystar u$ in $BV(\Omega;\R^2)$, and if in addition, 
\begin{align*}
R_\eps\to R \quad \text{in $C(\Omega;\R^{2\times 2})$,} 
\end{align*}
then  a weak-strong convergence argument implies  that 
\begin{align*}
\gamma_\eps  \Lcal^2\lfloor \Omega = \big[(\nabla u_\eps) e_2 \cdot R_\eps e_1\bigr] \Lcal^2\lfloor \Omega  \weaklystar (Du)e_2\cdot Re_1 \qquad \text{in $\Mcal(\Omega)$.}
\end{align*}

However, if continuity and uniform convergence of $R_\eps$ fail, the  limit representation above  may no longer be true in general, even if $R\in C(\Omega;SO(2))$.  To see this, let us consider the  basic  construction in Remark~\ref{rmk:alternativeconst}~(ii). 
In this case, 
\begin{align}
\label{eq:1st-quant}
\gamma_\eps  \Lcal^2\lfloor \Omega \weaklystar (\alpha +  \beta )\Hcal^1\lfloor\big((0,1)\times \{0\}\big) \qquad \text{in $\Mcal(\Omega)$,}
\end{align}
 whereas 
\begin{equation}
\label{eq:2nd-quant}(Du)e_2\cdot Re_1=[(\psi^+-\psi^-)\cdot Re_1]\Hcal^1\lfloor \big((0,1)\times \{0\}\big).
\end{equation} 
Recalling that $\psi^+-\psi^-=\alpha  Re_1+\beta  Se_1 $, the quantities in \eqref{eq:1st-quant} and \eqref{eq:2nd-quant} can only match if $Re_1 || Se_1$,  which  contradicts  the assumption  that $Re_1$ and $Se_1$ are linearly independent.

The role of the higher-order regularization in \eqref{eq:def-E-ep-delta} is exactly that  it helps overcome the issue discussed above. In fact, it guarantees the desired compactness properties for sequences of deformations with equibounded energies.
\end{remark}

 To conclude, we present an alternative construction for the recovery sequence in Step~3 of the proof  of  Theorem~\ref{thm:Gamma_convergence}.

\begin{proof}[Alternative proof of Theorem~\ref{thm:Gamma_convergence}]
As before, we may assume without loss
of generality that $\Omega=(0,1)\times (-1,1)$. Moreover, the compactness property and lower bound can be studied exactly
as in the proof of Theorem~\ref{thm:Gamma_convergence} above.

We are then left to show that given   \(u\in L^1_0(\Omega;\RR^2)\cap \Acal^\parallel\),   there exists a sequence  \((u_\epsi)_\epsi\subset L^1_0(\Omega;\RR^2)\)
  satisfying $u_\eps\to u$ in $L^1(\Omega;\R^2)$  and 
   \eqref{eq:upperbounddelta}.
We will proceed in three steps, building  up  complexity.

\medskip
\noindent
\textit{Step~1.} We assume  first  that $u\in L^1_0(\Omega;\R^2)\cap \Acal^{\parallel}$ is an \(SBV\)-function
with a single, constant jump line at $x_2=0$. 

This case can be treated as highlighted in Remark~\ref{rmk:alternativeconst}~(iii).
Let $R\in SO(2)$, $\gamma\in  L^1(\Omega)$ with $\partial_1\gamma=0$, and  $\psi^+,\, \psi^-\in \R^2$   with  $(\psi^+ - \psi^-)|| Re_1$ be such that
\begin{align*}
Du = R(\Ibb + \gamma e_1\otimes e_2) \Lcal^2\lfloor\Omega + (\psi^+ - \psi^-)\otimes e_2\Hcal^1\lfloor ((0,1)\times \{0\}).
\end{align*}
Note that setting 
\({\iota:= {\rm sign} ((\psi^+
-\psi^-)\cdot
Re_1) \in \{\pm1\}}\), we have   \(\psi^+ - \psi^- =\iota |\psi^+ - \psi^-| Re_1\) and  
\begin{align*}
|Du|(\Omega) = |D^au|(\Omega)| + |D^su|(\Omega)| = |D^a u|(\Omega) + |D^ju|(\Omega) = \int_\Omega |R(\Ibb + \gamma e_1\otimes e_2)| \dd{x} + |\psi^+-\psi^-|.
\end{align*} 

For each \(\epsi>0\), set
$\tau_\eps:=
\iota\frac{|D^ju|(\Omega)}{\lambda\eps} =\iota \frac{|\psi^+-\psi^-|}{\lambda\eps}$. Arguing as, for instance, in the proof of Lemma~\ref{lem:existence_admissible1}, we can find \(u_\epsi
\in L^1_0(\Omega;\R^2)\cap  \Acal_\eps\)  such that
\begin{align*}
\nabla u_\eps = \begin{cases}
R(\Ibb+\tau_\eps e_1\otimes e_2)&\text{if $x\in (0,1)\times (0, \lambda\eps)$,}\\
R(\Ibb + \frac{\gamma}{\lambda} \mathbbm{1}_{\eps\Ysoft\cap \Omega}e_1\otimes e_2) & \text{otherwise,} 
\end{cases}
\end{align*}
and  $u_\eps\weaklystar u$ in $BV(\Omega;\R^{2})$. 
 Next, we show that this construction
yields convergence of energies. Indeed, we have 
\begin{align*}
\lim_{\eps\to 0} E_\eps^{\delta}(u_\eps) & = \lim_{\eps\to 0} \bigg(\int_{(0,1)\times (0, \lambda\eps)}|\tau_\eps| \dd{x} + \int_{\Omega\setminus (0,1)\times (0, \lambda\eps)}\Big|\frac{\gamma}{\lambda}\Big|\mathbbm{1}_{\eps\Ysoft} \dd x +  \delta\norm{Re_1}_{W^{1,p}(\Omega;\R^2)}^p \bigg) \\ &= |\psi^+
-\psi^-| + \int_{\Omega}|\gamma|\dd{x+ \delta|\Omega|} =  |D^su|(\Omega)
+ \int_{\Omega} |\gamma|\dd{x}+ \delta|\Omega|= E^{\delta}(u).
\end{align*}

\noindent
\textit{Step~2.} We assume  next  that $u\in L^1_0(\Omega;\R^2)\cap \Acal^{\parallel}$ is an \(SBV\)-function
with a  finite number
of horizontal  jump lines and with constant upper and
lower approximate limits on each jump line. 

In this setting, \(\nabla u = R(\Ibb+\gamma
e_1\otimes e_2)\) with $R\in SO(2)$ and \(\gamma\in L^1(\Omega)\)  with $\partial_1 \gamma=0$, 
\(J_u=\bigcup_{i=1}^\ell  (0,1)\times \{a_i\} \) with \(\ell\in\NN\)
and \(-1< a_1 <a_2 < \cdots < a_\ell <1\),
\(D^ju= \sum_{i=1}^\ell(\psi^+_i - \psi^-_i)\otimes e_2 \Hcal^1\lfloor((0,1)\times \{a_i\}) \) with   \(\psi^\pm_i\in\RR^2\) 
 satisfying  \((\psi^+_i - \psi^-_i) || Re_1\) for all \(i\in\{1,...,\ell\}\),
and \(D^c u =0\).
Hence, 
\begin{align}\label{eq:DuSBVfinitejumps}
Du = R(\Ibb + \gamma e_1\otimes e_2) \Lcal^2\lfloor\Omega +  \sum_{i=1}^\ell(\psi^+_i
- \psi^-_i)\otimes e_2 \Hcal^1\lfloor((0,1)\times \{a_i\}) 
\end{align} 
and 
\begin{equation*}
\begin{aligned}
  |D^su|(\Omega)  =   \sum_{i=1}^\ell|\psi^+_i      - \psi^-_i|.
\end{aligned}
\end{equation*} 

As in the proof of Proposition~\ref{prop:SBVinfty}, the idea is to  perform  a construction similar to
that in Step~1 around each jump line but accounting for the possibility
that one or more of the jump lines may not intersect \(\eps\Ysoft\cap
\Omega\).

Fix \(i\in \{1, ..., \ell\} \) and \(\epsi>0\), and  let \(\kappa_\epsi^i
\in \ZZ\)  be the  integer such that \(a_i \in \epsi[ \kappa_\epsi^i, \kappa_\epsi^i+1)\).
 Since   \(a_i\not = a_j\) if \(i\not= j\), we may assume that the sets 
  $\{\epsi[ \kappa_\epsi^i, \kappa_\epsi^i+1)\}_{i}$  are pairwise disjoint for  all \(\epsi>0\) (this is true for all \(\epsi>0\)
sufficiently small).
Then,  we take $u_\eps\in L_0^1(\Omega;\R^2)  \cap \Acal_\eps$ 
such that  
\begin{align*}
\nabla u_\eps = \begin{cases}
R(\Ibb+\tau_\eps^i e_1\otimes e_2)&\text{in $ (0,1)\times \epsi(
\kappa_\epsi^i, \kappa_\epsi^i+\lambda)$,}\\
R(\Ibb + \frac{\gamma}{\lambda} \mathbbm{1}_{\eps\Ysoft\cap \Omega}e_1\otimes
e_2) & \text{otherwise,} 
\end{cases}
\end{align*}
where  $\tau_\eps^i={\iota_i} \frac{|\psi^+_i
- \psi^-_i|}{\lambda\eps}$ with
\({\iota_i:= {\rm sign} ((\psi^+_i
-\psi^-_i)\cdot
Re_1) \in \{\pm1\}}\).  As  in the proof of Proposition~\ref{prop:SBVinfty},
we  obtain that  
\begin{equation}
\label{eq:DuetoDu}
\begin{aligned}
\lim_{\eps\to 0} \int_{\Omega} \nabla  u_\eps \varphi \dd{x}
=\sum_{i=1}^\ell  \int_0^1 {\iota_i}|\psi^+_i
- \psi^-_i|( Re_1\otimes
e_2)\ffi(x_1,a_i)\dd
x_1  +
\int_\Omega R(\Ibb+\gamma e_1\otimes e_2) \varphi\dd{x}
\end{aligned}
\end{equation}
for all $\varphi\in C_0(\Omega)$. Recalling \eqref{eq:DuSBVfinitejumps} and the equalities \(\psi^+_i
- \psi^-_i={\iota_i}|\psi^+_i - \psi^-_i|
Re_1\)  for \(i\in\{1,...,\ell\}\),
\eqref{eq:DuetoDu} shows that   $Du_\eps\weaklystar Du$ in $\Mcal(\Omega;\R^{2\times
2})$. Hence,   $u_\eps\weaklystar u$ in $BV(\Omega;\R^{2})$. 

Finally, proceeding exactly as in Step~1, we conclude that this
construction also yields  convergence of
the energies. This ends Step~2.

\textit{Step~3.} We consider  now  the general case $u\in L^1_0(\Omega;\R^2)\cap \Acal^{\parallel}$.

 Similarly to  the beginning of the proof
of Proposition~\ref{prop:Aparallel}  (see~\eqref{vw}), 
 we can write
\begin{align*}
u(x) =x_1 Re_1 + \phi_a(x_2) + \phi_s(x_2), \quad x\in \Omega,
\end{align*} 
where \(\phi_a(x_2):=x_2 Re_2 + \vartheta_a(x_2)Re_1
+ c \) and \(\phi_s(x_2) := \vartheta_s(x_2)
Re_1.\) Note that \(\phi_a\in W^{1,1}(-1,1;\RR^2)\)
and \(\phi_s\in BV(-1,1;\R^2)\)
is the
sum  of  a jump
function and a Cantor function; in particular,    \(\vartheta'=\vartheta_a'\) and
\(D\phi_s = D^s\phi_s\) (see \eqref{1dsplitting}). Moreover,
\begin{equation}\label{eq:DsuDsphi}
\begin{aligned}
&\nabla u =Re_1\otimes e_1 + \nabla \phi_a \otimes e_2  = R(\Ibb +\vartheta'_a e_1\otimes
e_2)= R(\Ibb
+\vartheta' e_1\otimes
e_2), 
\\
& D^su  = \Lcal^1\lfloor(0,1){ \boldsymbol\otimes }  D\phi_s  
 ,\\
& |D^su|\lfloor\Omega  = \Lcal^1\lfloor(0,1){ \boldsymbol\otimes }
 |D\phi_s|.  
\end{aligned}
\end{equation}
By Lemma~\ref{lem:charApar}, there exists \(\varrho\in
L^1_{|D^su|}(-1,1;\mathbb{R}^2)\) with
\(|\varrho| =1\) such that 
\begin{equation}
\begin{aligned}\label{eq:Dsuvarrho}
D^su =(\varrho\otimes e_2)|D^s u| \enspace \text{
and } \enspace \varrho=(\varrho\cdot
Re_1)
Re_1.
\end{aligned}
\end{equation}
Let \(\varrho_h \in C^\infty([-1,1])\)
be such that 
\begin{align}\label{rhoh}
\lim_{h\to\infty}\int_\Omega |\varrho_h(x_2) - \varrho(x_2)|\dd |D^su|(x)
=0.
\end{align}  
Since  \(|\varrho|=1\), we  can choose  such a sequence so that \(|\varrho_h|\leq
1\).

 Due to  the properties of good representatives (see \cite[(3.24)]{AFP00})
and  \cite[Lemma~3.2]{CrDe11}, for each \(n\in\N\),  there exists
a piecewise constant function \(
\phi_n \in  BV(-1,1;\R^2)\), of the form 
\begin{equation*}
\begin{aligned}
 \phi_n  = \sum_{i=0}^{\ell_n } b_i^n  \chi_{A_i^n},  
\end{aligned}
\end{equation*}
where \(\ell_n\in\NN\), \((b_i^n)_{i=0}^{\ell_n}\subset
\RR^2\), and \((A_i^n)_{i=0}^{\ell_n}\) is a partition of \((-1,1)\)
 into  intervals with 
 \(\sup
A^n_i
= \inf A^n_{i+1}\), satisfying
\begin{align}
& J_{\phi_n}=\bigcup_{i=1}^{\ell_n} \{a_i^n\} \text{ with
} a_i^n:=\sup A^n_{i-1},\nonumber\\
& \lim_{n\to\infty} \Vert \phi_n -  \phi_s\Vert_{L^1(-1,1;\RR^2)}
= 0, \label{eq:vntobarv}\\
& \lim_{n\to\infty} |D\phi_n| (-1,1)= \lim_{n\to\infty} |D^j\phi_n|
(-1,1) = |D  \phi_s|(-1,1) =|D^su|(\Omega). \label{eq:DvntoDbarv}
\end{align}
 Indeed,  \eqref{eq:vntobarv} and \eqref{eq:DvntoDbarv} mean
that \((\phi_n)_{n\in\NN}\) converges strictly to \( \phi_s\)
in
\(BV(-1,1;\R^2)\), which implies that 
\begin{equation}
\label{eq:strctconv}
\begin{aligned}
|D\phi_n| \weaklystar |D
\phi_s|\quad  \text{ in } \Mcal(-1,1),
\end{aligned}
\end{equation}
see \cite[Proposition~3.5]{AFP00}.

Finally, for \(n\in\NN\), we define
\begin{equation*}
\begin{aligned}
 u_n(x):= x_1 Re_1 + \phi_a(x_2) + \phi_n(x_2) +c_n, \enspace
x\in\Omega, \end{aligned}
\end{equation*}
where \(c_n\in \RR^2\) are constants chosen  so  that \(\int_\Omega
u_n\dd x =0\). Note that \(c_n\to0\) as \(n\to\infty\) by \eqref{eq:vntobarv}.
Moreover, for each \(n\in\NN\),  the map  \(u_n \in L^1_0(\Omega;\R^2)\)
  has the same structure as in
Step~2 apart from the condition \((u_n^+ - u_n^-) || Re_1\)  on $J_{u_n}$,  which a priori is not  satisfied.  
Choosing \(\iota_i^n:=\varrho_h(a_i^n)\cdot
Re_1\), we can invoke Step~2 up to, and including, \eqref{eq:DuetoDu}
  to construct  a  sequence \((u^{n,h}_\epsi)_{\epsi}\subset
L^1_0(\Omega;\R^2)\cap W^{1,1}(\Omega;\RR^2)\)  that satisfies  for all \(\ffi\in C_0(\Omega)\),
\begin{equation}
\label{eq:almstwc}
\begin{aligned}
 \lim_{\eps\to 0} \int_{\Omega} \nabla  u_\eps^{n,h} \varphi
\dd{x}
&=\sum_{i=1}^{\ell_n}  \int_0^1 
{ (\varrho_h(a_i^n)\cdot Re_1) }|b^n_{i}-b^n_{i-1}|(
Re_1\otimes
e_2)\ffi(x_1,a_i^n)\dd
x_1  \\
&\qquad+
\int_\Omega R(\Ibb+  \vartheta'_a(x_2) e_1\otimes e_2) \varphi\dd{x}.
\end{aligned}
\end{equation}
 We conclude from  ~\eqref{eq:DsuDsphi}, \eqref{eq:Dsuvarrho},  \eqref{rhoh},  \eqref{eq:vntobarv},  \eqref{eq:strctconv},
and the
 Lebesgue dominated convergence theorem that
\begin{equation}
\label{eq:DuentoDu}
\begin{aligned}
&{\lim_{h\to\infty}}\lim_{n\to\infty}
\sum_{i=1}^{\ell_n}  \int_0^1
{ (\varrho_h(a_i^n)\cdot
Re_1)}|b^n_{i}-b^n_{i-1}|( Re_1\otimes
e_2)\ffi(x_1,a_i^n)\dd
x_1\\
&\quad={\lim_{h\to\infty}} \lim_{n\to\infty}\int_0^1\int_{-1}^{1}
{(\varrho_h(x_2)\cdot
Re_1)}(
Re_1\otimes
e_2)\ffi(x_1,x_2) \dd |D\phi_n|(x_2)
\dd x_1\\
&\quad={\lim_{h\to\infty}} \int_0^1\int_{-1}^{1}
{(\varrho_h(x_2)\cdot
Re_1)}(
Re_1\otimes
e_2)\ffi(x_1,x_2) \dd |D\phi_s|(x_2)
\dd x_1\\&\quad= \int_\Omega
{(\varrho(x_2)\cdot
Re_1)}(Re_1\otimes
e_2)\ffi\dd |D^s u|= \int_\Omega
{(\varrho(x_2)}\otimes
e_2)\ffi\dd |D^s u|= \int_{\Omega} \varphi \dd{D^su.}\\
\end{aligned}
\end{equation}

Recalling that \(|\varrho_h(a_i^n)\cdot
Re_1|\leq 1\), we can further  argue as in Steps~1 and ~2
 regarding the convergence
of the energies
to get 
\begin{align}\label{eq:Euepsin}
{\limsup_{\epsi\to 0}} \,E_\epsi^\delta
(u_\epsi^{n,h}) 
&  \leq  E^\delta(u_n) = \int_{\Omega}
|\vartheta'_a  (x_2) |\dd{x}+ |D^s\phi_n|(-1,1)+\delta|\Omega| \\ &= \int_{\Omega}
|\vartheta'  (x_2) |\dd{x}+ |D^j\phi_n|(-1,1)
+\delta|\Omega|. 
\end{align}

Letting \(n\to\infty\)  and $h\to \infty$  in \eqref{eq:almstwc} and \eqref{eq:Euepsin},
from \eqref{eq:DuentoDu}, \eqref{eq:DvntoDbarv}, and  
\eqref{eq:DsuDsphi}, we conclude that for all \(\ffi\in C_0(\Omega)\), %
\begin{align}
&{\lim_{h\to\infty}}\lim_{n\to\infty}
 \lim_{\epsi\to 0} \int_{\Omega} \nabla  u_\eps^{n,h} \varphi
\dd{x}= \int_{\Omega} \varphi \dd{Du}, \label{eq:limnepsi1}\\
& {\limsup_{h\to\infty} \limsup_{n\to\infty}
\lim_{\epsi\to 0}} E_\epsi^{\delta} (u_\epsi^{n,h}) \leq
 \int_{\Omega}
|\vartheta'  (x_2) |\dd{x}+ |D^s u|(\Omega) +\delta|\Omega|= E^{\delta}(u). \label{eq:limnepsi2}
\end{align}

 Owing to  the separability of \(C_0(\Omega) \)  and 
\eqref{eq:limnepsi1}--\eqref{eq:limnepsi2}, we
can use a diagonalization argument as that in \cite[proof of
Proposition~1.11 (p.449)]{FeFo120}
to find sequences   {\( (h_\epsi)_\epsi\)
 and} \( (n_\epsi)_\epsi\)  such that \({h_\epsi},\,n_\epsi\to\infty\)
as \(\epsi \to 0\) and \(\tilde u_\epsi:= u_{\epsi}^{n_\epsi,
 {h_\epsi}}
\in  L_0^1(\Omega;\R^2)\cap  W^{1,1}(\Omega;\RR^2) \)  has all the desired properties. 
\end{proof}

\subsection*{Acknowledgements} 
The work of Elisa Davoli has been funded by the Austrian Science Fund (FWF) project F65 ``Taming complexity in partial differential systems".   Carolin Kreisbeck gratefully acknowledges the support by a Westerdijk Fellowship from Utrecht University.  The research of Elisa Davoli and Carolin Kreisbeck  was  supported by the Mathematisches Forschungsinstitut Oberwolfach through the program ``Research in Pairs'' in 2017. The hospitality of King Abdullah University of Science and Technology, Utrecht University, and of the University of Vienna is acknowledged. All authors are thankful to the Erwin Schr\"odinger Institute in Vienna, where part of this work  was  developed during the workshop ``New trends in the variational modeling of failure phenomena". 


\bibliographystyle{abbrv}
\bibliography{Homogenization20Nov}
\end{document}